\documentclass[12pt]{amsart}
\usepackage{graphicx}
\graphicspath{ {./images/} }
\usepackage{float}
\usepackage{amsfonts}
\usepackage{amsmath}
\usepackage{amssymb}
\usepackage{enumerate}
\usepackage{comment}
\usepackage{hyperref}
\usepackage{csquotes}
\newtheorem{thm}{Theorem}[section]
\newtheorem{claim}[thm]{Claim}
\newtheorem{cor}[thm]{Corollary}
\newtheorem{lem}[thm]{Lemma}
\newtheorem{prop}[thm]{Proposition}
\newtheorem*{prob*}{Problem}
\newtheorem*{ques*}{Question}
\newtheorem*{thm*}{Theorem}
\newtheorem*{quest*}{Question}

\theoremstyle{definition}
\newtheorem{defn}[thm]{Definition}
\newtheorem{example}[thm]{Example}
\newtheorem{conj}[thm]{Conjecture}
\newtheorem*{defn*}{Definition}
\newtheorem{rem}[thm]{Remark}

\newtheorem{rem*}[thm]{Remark}
\numberwithin{equation}{section}

\newcommand{\Z}{\mathbb Z}

 \DeclareMathOperator{\sgn}{sgn}

\newcommand{\mG}{\mathcal {G}}
\newcommand{\mL}{\mathcal {L}}
\newcommand{\mH}{\mathcal {H}}

	\title[Finite dimensional nilpotent systems]{Host-Kra factors for $\bigoplus_{p\in P}\mathbb{Z}/p\mathbb{Z}$ actions and finite dimensional nilpotent systems}
\date{\today}
\author{Or Shalom}

\begin{document}
	\begin{abstract}
	Let $\mathcal{P}$ be a countable multiset of primes and let $G=\bigoplus_{p\in P}\mathbb{Z}/p\mathbb{Z}$. We study the universal characteristic factors associated with the Gowers-Host-Kra seminorms for the group $G$. We show that the universal characteristic factor of order $<k+1$ is a factor (Definition \ref{kext}) of an inverse limit of \emph{finite dimensional $k$-step nilpotent homogeneous spaces} (Theorem \ref{Main2:thm}). The latter is a counterpart of a $k$-step nilsystem where the homogeneous group is not necessarily a Lie group. As an application of our structure theorem we derive an alternative proof for the $L^2$-convergence of multiple ergodic averages associated with $k$-term arithmetic progressions in $G$ (Theorem \ref{convergence}) and derive a formula for the limit in the special case where the underlying space is a nilpotent homogeneous system (Theorem \ref{formula}). Our results provide a counterpart of the structure theorem of Host and Kra \cite{HK} and Ziegler \cite{Z} concerning $\mathbb{Z}$-actions and generalizes the results of Bergelson Tao and Ziegler \cite{Berg& tao & ziegler}, \cite{BTZ} concerning $\mathbb{F}_p^\omega$-actions. This is also the first instance of studying the Host-Kra factors of non-finitely generated groups of unbounded torsion. 
	\end{abstract}
	\maketitle
\section{Introduction}
The universal characteristic factors for multiple ergodic averages play an important role in ergodic Ramsey theory. For instance, in the case of $\mathbb{Z}$-actions they are related to the theorem of Szemer\'edi about the existence of arbitrary large arithmetic progressions in sets of positive upper Banach density in the integers \cite{Sz}. The universal characteristic factors associated with multiple ergodic averages in $\mathbb{Z}$-actions were studied by Host and Kra \cite{HK} and independently by Ziegler \cite{Z}. Later, Bergelson, Tao and Ziegler proved a counterpart for the non-finitely generated group $G=\mathbb{F}_p^\omega$. The goal of this paper is to generalize these results further for the group $G=\bigoplus_{p\in P}\mathbb{Z}/p\mathbb{Z}$ where $P$ is a multiset of primes. Moreover, in section \ref{generalizations} we discuss the general case where $G$ is any countable abelian group. In particular, we identify a result (Conjecture \ref{conj}) which leads to a general structure theorem for the Gowers-Host-Kra seminorms.\\

\textbf{Conventions:} We use $X$ to denote a probability space. For technical reasons we assume that any probability space $X=(X,\mathcal{B},\mu)$ is regular\footnote{Meaning that $X$ is a compact metric
space, $\mathcal{B}$ is the completion of the $\sigma$-algebra of Borel sets, and $\mu$ is a Borel measure.} and separable modulo null sets. We let $(U,\cdot)$ denote a compact abelian group and we assume that all topological groups in this paper are metrizable. Let $(G,+)$ be a countable abelain group, a $G$-system is a probability space $X=(X,\mathcal{B},\mu)$ together with an action of $G$ on $X$ by measure preserving transformations $T_g:X\rightarrow X$. Throughout most of this paper we use $G$ to denote the group $\bigoplus_{p\in P}\mathbb{Z}/p\mathbb{Z}$ where $P$ is a given countable multiset of primes.\\

Host and Kra proved that the universal characteristic factors for $\mathbb{Z}$-actions are closely related to an infinite version of the Gowers norms. The following version of the Gowers-Host-Kra (GHK) seminorms in the special case where $G=\mathbb{Z}$ was essentially introduced by Host and Kra in \cite{HK} (see \cite[Proposition 16 Chapter 8]{HKbook} for this version).
\begin{defn}
	[Gowers Host Kra seminorms]
	Let $(X,T_g)$ be a $G$-system, let $\phi\in L^\infty (X)$, and let $k\geq 1$ be an integer. The GHK seminorm $\|\phi\|_{U^k}$ of order $k$ of $\phi$ is defined recursively by the formula
	\[
	\|\phi\|_{U^1}:=\lim_{N\rightarrow\infty}\frac{1}{|\Phi_N^1|}\|\sum_{g\in\Phi_N^1}\phi\circ T_g\|_{L^2}
	\]
	for $k=1$, and
	\[
	\|\phi\|_{U^k}:=\lim_{N\rightarrow\infty}\left(\frac{1}{|\Phi_N^k|}\sum_{g\in\Phi_N^k}\|\Delta_g\phi\|_{U^{k-1}}^{2^{k-1}}\right)^{1/2^k}
	\]
	for $k>1$, where $\phi_N^1,...,\phi_N^k$ are arbitrary F{\o}lner sequences and $\Delta_g \phi(x)=\phi(T_gx)\cdot\overline{\phi(x)}$.
\end{defn}
These seminorms were first introduced by Gowers in the special case where  $G=\mathbb{Z}/N\mathbb{Z}$ in \cite{G}, where he proved quantitative bounds for Szemer\'edi's theorem \cite{Sz}. As mentioned above, Host and Kra \cite{HK} generalized these seminorms for the infinite group $\mathbb{Z}$ and proved that each seminorm corresponds to a unique factor of $X$. Later, Leibman \cite{Leib} proved that these factors coincides with the universal characteristic factors for multiple ergodic averages which were studied by Ziegler in \cite{Z}.
\begin{prop}[Existence and uniqueness of the universal characteristic factors] \label{UCF} Let $G$ be a countable abelian group, let $X$ be a $G$-system, and let $k\geq 1$. Then there exists a factor $Z_{<k}(X)=(Z_{<k}(X),\mathcal{B}_{Z_{<k}(X)},\mu_{Z_{<k}(X)},\pi^X_{Z_{<k}(X)})$ of $X$ with the property that for every $f\in L^\infty (X)$, $\|f\|_{U^{k}(X)}=0$ if and only if $E(f|Z_{<k}(X))=0$. This factor is unique up to isomorphism and is called the $k$-th universal characteristic factor of $X$. If $X=Z_{<k}(X)$ we say that $X$ is of order $<k$.
\end{prop}

In Appendix \ref{background} we summarize previous work. In particular, we survey the definitions and various results by Host and Kra \cite{HK} and Bergelson Tao and Ziegler \cite{Berg& tao & ziegler}. We also state a structure theorem for totally disconnected $\bigoplus_{p\in P}\mathbb{Z}/p\mathbb{Z}$-systems (Theorem \ref{Main:thm}) from the author's previous work \cite{OS}. This theorem will be used as a black box in this paper. 

 Another related approach for the study of the universal characteristic factors and related problems like the inverse problem for the Gowers norms (see \cite{GT1}, \cite{GTZ} and \cite{TZ} for more details), which we do not pursue in this paper, is the study of nilspaces and nilspace systems. In \cite{SzCam} Antol\'in Camarena and Szegedy introduced a purely combinatorial counterpart of the Host-Kra factors called {\em nilspaces}. The idea was to give a more abstract and general notion which describes the ``cubic structure" of an ergodic system (see Host and Kra \cite[Section 2]{HK}). They proved that connected nilspaces are inverse limits of nilmanifolds in the category of nilspaces, which justifies the \emph{nil} in nilspace. A nilspace system, is a compact nilspace equipped with a continuous action of a group which preserves its cube structure. In \cite{Sz2} Candela, Gonz\'alez-S\'anchez and Szegedy studied nilspace systems and proved that the theory of nilspaces passes through to nilspace systems when the group acting on the space if finitely generated (Gutman Manners and Varj\'u \cite{Gut3} generalized this result to compactly generated groups). In \cite{Sz3} Candela and Szegedy use nilspaces to prove a structure theorem for characteristic factors for GHK
seminorms associated with any nilpotent group. They proved that the characteristic factors for the GHK seminorms of a nilpotent group is a nilspace system and obtained from previous result an alternative proof of Host-Kra structure theorem for finitely generated nilpotent groups. In a series of papers, \cite{Gut1},\cite{Gut2},\cite{Gut3} Gutman, Manners, and Varj{\'u} studied further the structure of nilspaces. By imposing another measure-theoretical aspect to these nilspaces, Gutman and Lian \cite{Gut4} gave yet another alternative proof of Host and Kra's theorem for arbitrary finitely generated abelian groups.\\  

The structure of \textit{nilspace systems} is currently only well understood when the acting group is finitely generated (or compactly generated in general). The only exception is a new result for actions of the group $G=\mathbb{F}_p^\omega$ which was recently obtained by Candela, Gonz\'alez-S\'anchez and Szegedy \cite{CGS}. We believe that the tools developed in this paper may turn out to be useful also in the study of nilspace systems associated when the acting group is not finitely generated and of unbounded torsion. For instance, it is evident from this work that the Host-Kra factors of a non-finitely generated group with unbounded torsion is not necessarily isomorphic to an inverse limit of nilsystems. Thus, contrary to the finitely generated case, it is impossible to describe arbitrary nilspace systems as an inverse limit of nilmanifolds. The main results in this paper suggests that the notion of \emph{nilpotent systems} in Defintion \ref{nilpotent system:def} (see also the double-coset construction from \cite[Theorem 1.21]{OS2}) may replace the notion of a nilmanifold when studying arbitrary nilspace systems.\\

\noindent \textbf{Acknowledgement} I would like to thank my advisor Prof. Tamar Ziegler for many valuables discussions and encouragements. The author is supported by an ERC grant ErgComNum 682150.

\section{main results} Recall that a $k$-step nilsystem is a quadruple $(\mathcal{G}/\Gamma,\mathcal{B},\mu,R)$ where $\mathcal{G}$ is a $k$-step nilpotent Lie group, $\Gamma$ is a discrete co-compact subgroup, $\mathcal{B}$ is the Borel $\sigma$-algebra, $\mu$ is the Haar measure and $R:\mG/\Gamma \rightarrow \mG/\Gamma$ is a left multiplication by some element $r\in\mG$. Host and Kra \cite[Theorem 10.1]{HK} and independently Ziegler \cite[Theorem 1.7]{Z} proved the following structure theorem for the universal characteristic factors concerning $\mathbb{Z}$-actions.
\begin{thm} \label{HKZ}
    Let $(X,\mathcal{B},\mu,T)$ be an ergodic invertible system. Then for every $k\geq 1$, the factor $Z_{<k+1}(X)$ is isomorphic to an inverse limit of $k$-step nilsystems.
\end{thm}
Our goal is to generalize this result to $\bigoplus_{p\in P}\mathbb{Z}/p\mathbb{Z}$-actions, where $P$ is a countable multiset of primes. As a first step we explain how to interpret the results of Bergelson, Tao and Ziegler \cite{Berg& tao & ziegler} about $\mathbb{F}_p^\omega$-systems in this language. We need the following version of nilsystems.
\begin{defn} [Zero dimensional nilpotent system] \label{zerodimnilsystem}
    Let $k\geq 1$ and let $G$ be a countable abelian group. A \textit{zero dimensional $k$-step nilpotent system} is a quadruple $X=(\mG/\Gamma,\mathcal{B},\mu,(R_g)_{g\in G})$ where $\mG$ is a zero dimensional\footnote{A zero dimensional group is a topological group with a totally disconnected topology. That is, every point has a basis of clopen sets.} $k$-step nilpotent group, $\Gamma$ is a closed (not necessarily discrete) co-compact subgroup, $\mathcal{B}$ is the Borel $\sigma$-algebra, $\mu$ is a left $\mathcal{G}$-invariant measure and there exists a homomorphism $\varphi:G\rightarrow \mG$ such that for every $g\in G$ the transformation  $R_g:\mG/\Gamma\rightarrow \mG/\Gamma$ is given by a left multiplication by $\varphi(g)$.
\end{defn}
The following structure theorem for the universal characteristic factors concerning $\mathbb{F}_p^\omega$ can be derived from the work of Bergelson, Tao and Ziegler \cite{Berg& tao & ziegler}.\footnote{A proof of this result is not explicitly given in \cite{Berg& tao & ziegler}. One way to prove this theorem is by following the arguments in this paper in the simple case where $P=\{p,p,p,...\}$.}
\begin{thm} \label{BTZthm}
    Let $k\geq 1$, any ergodic $\mathbb{F}_p^\omega$-system of order $<k+1$ is a zero dimensional $k$-step nilpotent system whenever $k<p$. In the low characteristic case $k\geq p$, our argument shows that there exists some $m=O_k(1)$ and an $m$-extension (see Definition \ref{kext}) isomorphic to a zero dimensional $k$-step nilpotent $\mathbb{Z}/p^m\mathbb{Z}$-system.
\end{thm}
\begin{rem}
We cautiously note that it is possible that the group $\mathcal{G}$ in definition \ref{zerodimnilsystem} is not locally compact. For instance consider the ergodic $\mathbb{F}_p^\omega$-system $$X=\prod_{i=1}^\mathbb{N} C_p \times_\sigma \prod_{i=1}^\mathbb{N}C_p$$ of order $<3$, where $\sigma$ is any phase polynomial of degree $<2$. If $\mathcal{G}(X)$ is the Host-Kra group of $X$ (see Definition \ref{HKgroup:def}), then since $\sigma$ is a phase polynomial of degree $<2$, one can show that $\mathcal{G}(X)$ is a semi-direct product of $\prod_{i=1}^\mathbb{N}C_p$ with $\prod_{i=1}^\mathbb{N}C_p\oplus \text{Hom}(\prod_{i=1}^\mathbb{N}C_p,\prod_{i=1}^\mathbb{N}C_p)\cong \prod_{i=1}^\mathbb{N}C_p\oplus \left(\mathbb{F}_p^\omega\right)^{\mathbb{N}}$. The group $\Gamma=\left(\mathbb{F}_p^\omega\right)^{\mathbb{N}}$ is a non-locally compact totally disconnected subgroup and $\mathcal{G}(X)/\Gamma \cong X$.
\end{rem}
Let $P$ be a countable multiset of primes and $G=\bigoplus_{p\in P}\mathbb{Z}/p\mathbb{Z}$. In order to prove a counterpart of the theorem above for $G$-systems, we introduce a new notion of extensions (Definition \ref{kext}). Our main theorem (Theorem \ref{Main2:thm}) asserts, roughly speaking, that every ergodic $G$-system of order $<k+1$ is a (special) factor of an inverse limit of $k$-step \textit{finite dimensional nilpotent systems}.
\subsection{$k$-extensions and finite dimensional groups}
Let $(X,T_g)$ be a $G$-system
and $\varphi:H\rightarrow G$ be a homomorphism from a countable abelian group $H$ onto $G$. The homomorphism $\varphi$ gives rise to an action of $H$ on $X$ by $S_h x = T_{\varphi(h)} x$. This observation allows to define extensions outside of the category of $G$-systems.
\begin{defn} [$k$-extensions] \label{kext} Let $P=\{p_1,p_2,...\}$ be a multiset of primes and $G=\bigoplus_{p\in P}\mathbb{Z}/p\mathbb{Z}$. For a natural number $k\geq 1$, we define $G^{(k)}=\bigoplus_{p\in P} \mathbb{Z}/{p^k}\mathbb{Z}$ and let
\begin{align*}\varphi_k&:G^{(k)}\rightarrow G\\
\varphi(g)_i& = g_i \mod p_i 
\end{align*}
We say that a $G^{(k)}$-system $Y$ is a  $k$-extension of $X$ if it is an extension of $(X,G^{(k)})$.
\end{defn}
\begin{example} \label{example}
	    Let	$G=\Z/2\Z$ and $X=\{-1,1\}$ and define an action of $G$ on $X$ by $T_gx = (-1)^gx$. Similarly, let $H=\Z/4\Z$ and $Y=\{-1,-i,i,1\}$, then $H$ acts on $Y$ by $S_h y=i^hy$. The system $(Y,H)$ defines a $2$-extension of $(X,G)$ with respect to the the homomorphism
		\[
		\begin{aligned}
		&\varphi:H\rightarrow G \\ \varphi(h)&=h\mod 2
		\end{aligned}
		\]
		and the factor map $\pi:Y\rightarrow X$, where  $\pi(y)=y^2$.
	\end{example}
There are multiple notions of dimension for topological spaces in the literature, which do not coincide for non-locally compact groups (see \cite{AM} for a survey about the different notions of dimension and the history behind them). Throughout we say that a topological space $X$ is totally disconnected if for every $x,y\in X$, there exists a clopen subset $C\subseteq X$ such that $x\in C$ and $y\not\in C$.  Another possible definition for a totally disconnected space $X$ is to require that all connected components are singletons. These definitions do not coincide in general, but in this paper the results remain the same if we interchange one definition with another. It is well known that all products and closed subsets of totally disconnected sets are totally disconnected.\\ Since there is no concrete notion of dimension for non-locally compact groups we will use the following natural definition instead.
\begin{defn}\label{finitedim:def}
A topological group $H$ is said to be \textit{zero dimensional} if it is totally disconnected. A topological group $H$ is said to be \textit{finite dimensional} if it is contained in the family of groups $\mathcal{FD}$, where $\mathcal{FD}$ is the minimal family satisfying that
\begin{itemize}
    \item[(i)] $\mathcal{FD}$ contains all Lie groups and all totally disconnected groups.
    \item[(ii)] If $K\leq \mathcal{FD}$ then any closed subgroup of $K$ is in $\mathcal{FD}$.
    \item[(iii)] If $L\leq K$ is a closed subgroup and  $K/L,L\in\mathcal{FD}$, then $K\in\mathcal{FD}$.
    \end{itemize}
\end{defn}
In the specific case of compact abelian groups, this is equivalent to the following definition.
\begin{prop} [finite dimensional compact abelian groups] \label{FD:prop} \cite[Theorem 8.22]{HM} The following conditions are equivalent for a compact abelian group $U$ and a natural number $n$:
	\begin{enumerate}
		\item {$U$ is of dimension $n$.}
		\item {There exists a compact totally disconnected subgroup $\Delta$ of $U$ and a short exact sequence $$1\rightarrow \Delta \rightarrow U\rightarrow (S^1)^n\rightarrow 1.$$}
		\item {There exists a compact zero dimensional subgroup $\Delta$ of $U$ and a continuous surjective homomorphism $\varphi:\Delta\times\mathbb{R}^n\rightarrow U$ such that $\Gamma:=\ker\varphi$ is discrete. Hence, $U\cong (\Delta\times \mathbb{R}^n)/ \Gamma$ as topological groups.}
	\end{enumerate}
\end{prop}
We generalize Definition \ref{zerodimnilsystem}.
\begin{defn}[Nilpotent systems] \label{nilpotent system:def}
     Let $k\geq 1$ and let $G$ be a countable abelian group. A quadruple $X=(\mG/\Gamma,\mathcal{B},\mu,(R_g)_{g\in G})$ where $\mG$ is a $k$-step nilpotent Polish group, $\Gamma$ is a closed co-compact zero dimensional subgroup, and $\mathcal{B},\mu$ and $R_g$ as in Definition \ref{zerodimnilsystem} is called a \textit{$k$-step nilpotent system}. If in addition $\mG$ is a finite dimensional group, we say that $X$ is a finite dimensional $k$-step nilpotent system.
\end{defn}
We note that any zero dimensional subgroup of a Lie group is discrete. Therefore, if $\mG$ is a Lie group, then the nilpotent system $X$ is a nilsystem. Moreover, even though the notion of dimension in non-locally compact groups may exhibit some pathologies, the quotient space $\mG/\Gamma$ is compact and so it is of finite dimension with respect to any natural notion of dimension of compact topological spaces (e.g. Lebesgue covering dimension or the small or large inductive dimension).
\subsection{The Host-Kra group}
The Host-Kra group plays an important role in this paper. We generalize the definition from \cite[Section 5]{HK} to arbitrary countable abelian group $G$.
\begin{defn} \label{HKgroup:def}
    Let $G$ be a countable abelian group and let $(X,G)$ be a $G$-system. We denote by $\mG(X)$ the group of all transformations $t:X\rightarrow X$ with the property that for every $l>0$, the measure $\mu^{[l]}$ is $t^{[l]}$-invariant and $t^{[l]}$ acts trivially on $\mathcal{I}^{[l]}(X)$.
    \end{defn}
The measure $\mu^{[l]}$, the transformation $t^{[l]}:X^{[l]}\rightarrow X^{[l]}$ and the $\sigma$-algebra $\mathcal{I}^{[l]}(X)$ are defined in Appendix \ref{background}. We note that if $X$ is a systems of order $<k+1$ (i.e. $X=Z_{<k+1}(X)$), then $\mG(X)$ is a $k$-step nilpotent locally compact Polish group \cite[Corollary 5.9]{HK}.\\

In \cite{HK} Host and Kra proved the following stronger version of Theorem \ref{HKZ}.
\begin{thm} \label{HK}
Let $k\geq 0$. Let $X$ be an ergodic $\mathbb{Z}$-system of order $<k+1$. Then, for every $n\in\mathbb{N}$ there exists a factor $X_n$ of $X$ such that: 
\begin{enumerate}
    \item{$X_n$ is an increasing sequence (i.e. $X_n$ is a factor of $X_{n+1}$ for every $n$) and $X$ is the inverse limit of $X_n$.}
    \item {For each $n$, $X_n$ is isomorphic to the system $(\mG(X_n)/\Gamma(X_n),\mathcal{B}_n,\mu_n,S_n)$ where the Host-Kra group $\mG(X_n)$ is a $k$-step nilpotent locally compact Lie group, $\Gamma(X_n)$ is a discrete co-compact subgroup of $\mG(X_n)$, $\mathcal{B}_n$ is the Borel $\sigma$-algebra and $\mu_n$ is the Haar measure. The action of $S_n$ on $\mG(X_n)/\Gamma(X_n)$ is given by left multiplication by an element in $\mG(X_n)$.}
\end{enumerate}
\end{thm}
Our main result is the following counterpart of Theorem \ref{HK} for the group $G=\bigoplus_{p\in P}\mathbb{Z}/p\mathbb{Z}$.
\begin{thm}[Structure Theorem] \label{Main2:thm} 
	Let $k\geq 0$ and let $X$ be an ergodic $\bigoplus_{p\in P}\mathbb{Z}/p\mathbb{Z}$-system of order $<k+1$. Then for some $m=O_k(1)$\footnote{We use $O_k(1)$ to denote a quantity which is bounded by a constant depending only on $k$.} there exists an $m$-extension $Y$ of $X$ that is an inverse limit of finite dimensional $k$-step nilpotent systems. Moreover, for each $n\in\mathbb{N}$ there exists a factor $Y_n$ of $Y$ such that:
	\begin{enumerate}
	 \item{$Y_n$ is an increasing sequence and $Y$ is the inverse limit of $Y_n$.}
    \item {For each $n$, $Y_n$ is isomorphic to the system $(\mG(Y_n)/\Gamma(Y_n),\mathcal{B}_n,\mu_n,(S_{n,g})_{g\in \bigoplus_{p\in P}\mathbb{Z}/p^m\mathbb{Z}})$ where the Host-Kra group $\mG(Y_n)$ is a finite dimensional $k$-step nilpotent group, $\Gamma(Y_n)$ is a zero dimensional closed co-compact subgroup of $\mG(Y_n)$, $\mathcal{B}_n$ is the Borel $\sigma$-algebra and $\mu_n$ is a left $\mG(Y_n)$-invariant measure. For every $g\in \bigoplus_{p\in P}\mathbb{Z}/p^m\mathbb{Z}$, the action of $S_{n,g}$ on $\mG(Y_n)/\Gamma(Y_n)$ is given by left multiplication by an element in $\mG(Y_n)$.}
    \end{enumerate}
\end{thm}
In particular, it follows that if $X$ is a $G$-system where $G=\bigoplus_{p\in P}\mathbb{Z}/p\mathbb{Z}$, then for every $k\in\mathbb{N}$ there exists $m=O_k(1)$ and an $m$-extension $(Y,G^{(m)})$ such that the following diagram commutes.
	\begin{figure}[H]
		\includegraphics{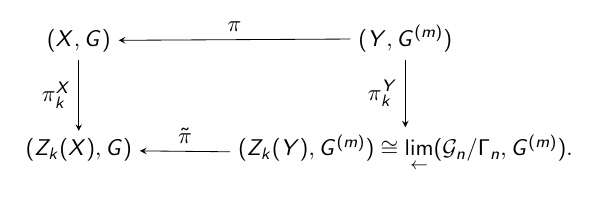}
	\end{figure}
The case $k=2$ of Theorem \ref{Main2:thm} was established by the author in \cite{OS} without the use of extensions (see also Theorem \ref{order<3}). We do not know whether this result (without $m$-extensions) holds for higher values of $k$. In section \ref{special case} we explain how $m$-extensions are used to overcome certain difficulties in the simple case where $k=3$. 
\subsection{Convergence of multiple ergodic averages and limit formula}
As an application of our structure theory, we derive an alternative proof for the convergence of some multiple ergodic averages and a limit formula in the special case where the underlying system is a nilpotent system and the homogeneous group is the Host-Kra group. More concretely, fix $k\geq 0$ and let $(X,T_g)$ be an ergodic $G$-system where $G=\bigoplus_{p\in P}\mathbb{Z}/p\mathbb{Z}$ and $f_1,...,f_{k+1}\in L^\infty(X)$. We study the limit of the following multiple ergodic averages as $N$ goes to infinity.
    \begin{equation} \label{average} \mathbb{E}_{g\in\Phi_N} T_g f_1 T_{2g} f_2\cdot...\cdot T_{(k+1)g} f_{k+1}.
	\end{equation}
	In the case of $\mathbb{Z}$-actions, Host and Kra \cite{HK} and Ziegler \cite{Z} proved the convergence of these averages by studying the universal characteristic factors. In the special case where $X$ is a nilsystem, Ziegler \cite{Z1} proved a limit formula for average (\ref{average}), see equation (\ref{formulaequation}) below. A simpler proof of this result and some applications to multiple recurrence can be found in \cite{BHK}. In \cite{BTZ} Bergelson, Tao and Ziegler proved a variant of this formula for $\mathbb{F}_p^\omega$-systems from which they deduced a Khintchine type recurrence for various configurations (see \cite[Theorem 1.13]{BTZ} for more details). In the special case where $k=2$ this formula and the multiple recurrence results were generalized to other abelian groups in \cite{OS2} and \cite{ABB}. We note that the norm convergence of average (\ref{average}) as $N\rightarrow\infty$ was proved by Walsh \cite{Walsh} for any countable nilpotent group $G$. We give an alternative proof for this result in the special case where $G=\bigoplus_{p\in P}\mathbb{Z}/p\mathbb{Z}$.
\begin{thm}[Convergence of the multiple ergodic averages] \label{convergence}
    Let $(X,T_g)$ be an ergodic $G$-system. Let $k\geq 0$ be such that $k+1\leq \min_{p\in P} p$ and $f_1,...,f_{k+1}\in L^\infty(X)$. Then, the multiple ergodic average (\ref{average}) convergence in $L^2(X)$ as $N$ goes to infinity.
\end{thm}
The properties of the Host-Kra group are needed in the proof of the following formula for the limit of average (\ref{average}) in the special case where the underlying system $X$ is a nilpotent system. In other words, it is important that the homogeneous groups in theorem \ref{Main2:thm} are the Host-Kra groups of the system.\\

 Let $\mG$ be a $k$-step nilpotent group and $\mu_{\mG/\Gamma}$ be a left $\mG$-invariant measure on $\mG/\Gamma$. Let $\mG_1=\mG$ and for every $2\leq r\leq k$ let $\mG_r$ be the closed subgroup generated by all the $r$-commutators $[...[[x_1,x_2],x_3],...x_r]$ where $x_1,...x_r\in \mG$ . Let $\Gamma_r = \mG_r\cap \Gamma$ and let $m_r$ be a left $\mG_r$-invariant measure on the quotient space $\mG_r/\mG_{r+1}\Gamma_r$. We have the following formula for the limit of average (\ref{average}).
\begin{thm} [Limit formula] \label{formula}
	Let $m\geq 1$ and let $G=\bigoplus_{p\in P}\mathbb{Z}/p^{m}\mathbb{Z}$. Fix $k\geq 0$ with $k+1<\min_{p\in P} p$ and let $X=\mG(X)/\Gamma$ be a $k$-step nilpotent $G$-system, where $\mathcal{G}(X)$ is the Host-Kra group of $X$. Then, for every $f_1,...,f_{k+1}\in L^\infty(X)$, every F{\o}lner sequence $\Phi_N$ of $G$ and $\mu_{\mG/\Gamma}$-almost every $x\in \mG/\Gamma$ we have
\begin{equation} \label{formulaequation}
\begin{split}
	\lim_{N\rightarrow\infty}\mathbb{E}_{g\in\Phi_N} T_g f_1(x) T_{2g} f_2(x)\cdot...\cdot T_{(k+1)g} f_{k+1}(x)&=\\ \int_{\mG/\Gamma} \int_{\mG_2/\Gamma_2}...\int_{\mG_k/\Gamma_k}  \prod_{i=1}^{k+1}f_i(x\cdot y_1^i\cdot y_2^{\binom{i}{2}}\cdot ...\cdot & y_i^{\binom{i}{i}})d\prod_{i=1}^k m_i(y_i\Gamma_i)
	\end{split}
	\end{equation}
 with the abuse of notation $f(x)=f(x\Gamma)$.
\end{thm}
\subsection{Discussion of the main steps in the proof of Theorem \ref{Main2:thm} and generalizations to other countable abelian groups.} \label{generalizations}
We summarize the main steps in the proof of  Theorem \ref{Main2:thm}. In each step we survey previous work concerning $\mathbb{Z}$-actions \cite{HK},\cite{Z} and $\mathbb{F}_p^\omega$-actions \cite{Berg& tao & ziegler}, describe the counterpart that we prove for $\bigoplus_{p\in P}\mathbb{Z}/p\mathbb{Z}$-actions and discuss the general case of $G$-actions where $G$ is any countable abelian group.\\

Let $m\geq 0$. We denote by $Z^1_{<m}(G,X,S^1)$ the group of cocycles of type $<m$ (Definition \ref{type:def}) and by $PC_{<m}(G,X,S^1) = P_{<m}(G,X,S^1)\cap Z^1(G,X,S^1)$ the phase polynomial cocycles of degree $<m$. The first step in the proof of Theorem \ref{Main2:thm} is to study how \textit{large} this subgroup is.\\

\textbf{Step 1}: (Theorem \ref{countable}) Let $X$ be an ergodic $\bigoplus_{p\in P}\mathbb{Z}/p\mathbb{Z}$-system. Then  $PC_{<m}(\bigoplus_{p\in P}\mathbb{Z}/p\mathbb{Z},X,S^1)\cdot B^1(\bigoplus_{p\in P}\mathbb{Z}/p\mathbb{Z},X,S^1)$ has at most countable index in $Z^1_{<m}(\bigoplus_{p\in P}\mathbb{Z}/p\mathbb{Z},X,S^1)$.\\

A well known theorem of Moore and Schmidt states that for any countable abelian group $G$ and any ergodic $G$-system $X$, $$Z^1_{<1}(G,X,S^1)=PC_{<1}(G,X,S^1)\cdot B^1(G,X,S^1).$$ In the special case where $G=\mathbb{F}_p^\omega$ this equality holds for higher values of $m$. Formally, Bergelson, Tao and Ziegler \cite{Berg& tao & ziegler} proved that $$Z^1_{<m}(\mathbb{F}_p^\omega,X,S^1)=PC_{<m}(\mathbb{F}_p^\omega,X,S^1)\cdot B^1(\mathbb{F}_p^\omega,X,S^1)$$ for every $m<p$ and an ergodic $\mathbb{F}_p^\omega$-system $X$.\\
This equality fails in general. For instance Furstenberg and Weiss proved that there exists an ergodic $\mathbb{Z}$-system $X$ and a $\mathbb{Z}$-cocycle $\rho$ of type $<2$ that is not cohomologous to a phase polynomial of any order \cite{Furstenberg}. In previous work \cite{OS}, we find an if and only if criterion for this equality to hold for $\bigoplus_{p\in P}\mathbb{Z}/p\mathbb{Z}$-systems and construct a counterpart of the Furstenberg and Weiss' example in the case where $P$ is an unbounded multiset of primes.
If $(X,T)$ is a group rotation, Host and Kra proved that $P_{<1}(\mathbb{Z},X,S^1)\cdot B^1(\mathbb{Z},X,S^1)$ is of countable index in $Z^1_{<2}(\mathbb{Z},X,S^1)$ and mentioned that the counterpart for higher values of $m$ fails. Instead, they proved the following result \cite[Lemma 9.2]{HK}: let $(X,T)$ be an ergodic $\mathbb{Z}$-system and let $(\Omega,P)$ be a probability space. Let $\omega\mapsto \rho_\omega$ be a measurable map into $Z^1_{<m}(\mathbb{Z},X,S^1)$. Then there exists a set of positive measure $\mathcal{A}\subseteq \Omega$ such that $\rho_\omega/\rho_{\omega'}$ is cohomologous to a constant for every $\omega,\omega'\in A$.\\ We thus conjecture that the following general version holds.
\begin{conj} \label{conj}
    Let $G$ be a countable abelian group and $(X,G)$ be an ergodic $G$-system. Let $m\geq 1$ and $(\Omega,P)$ a probability space. Then, for any measurable map $\omega\mapsto \rho_\omega$ from $\Omega$ to $Z^1_{<m}(G,X,S^1)$, there exists a set of positive measure $\mathcal{A}\subseteq \Omega$ such that $\rho_\omega/\rho_{\omega'}\in PC_{<m}(G,X,S^1)\cdot B^1(G,X,S^1)$.
\end{conj}
The next step in the proof of Theorem \ref{Main2:thm} is to reduce matters to the case where $X$ is a \textit{finite dimensional system} (see Definition \ref{fdsystems:def}) using inverse limits. Recall (Proposition \ref{abelext:prop}) that every ergodic system $X$ of order $<k$ is a tower of abelian extensions. Namely, $$X=U_0\times_{\rho_1} U_1\times...\times_{\rho_{k-1}}U_{k-1}.$$ We refer to the compact abelian groups $U_0,....,U_{k-1}$ as the structure groups of the system $X$.\\

\textbf{Step 2:} (Theorem \ref{InvFD:thm}) Let $k\geq 1$, then any ergodic $\bigoplus_{p\in P}\mathbb{Z}/p\mathbb{Z}$-system is an inverse limit of finite dimensional systems (see Definition \ref{fdsystems:def}).\\

Bergelson, Tao and Ziegler \cite{Berg& tao & ziegler} proved that the structure groups of an ergodic $\mathbb{F}_p^\omega$-system of order $<k$ are totally disconnected (zero dimensional). We refer to these systems as totally disconnected systems. The following definition is due to Host and Kra: a system $X$ of order $<k$ is called \textit{toral} if $U_1$ is a Lie group and any other structure group is a finite dimensional torus. In the case of $\mathbb{Z}$-actions Host and Kra proved that $X$ is an inverse limit of toral systems. In the generality of countable abelian groups, it is impossible to approximate every system with finite dimensional systems. As a counterexample consider the action of the group $\mathbb{Z}^\omega$ on $(S^1)^\mathbb{N}$ by $R_nx = (\alpha^{n_1}x,\alpha^{n_2}x,...)$ where $n=(n_1,n_2,...)$. If $\alpha$ is irrational then the action is ergodic. Let $(e_1,e_2,...)$ denote the natural basis for $\mathbb{Z}^\omega$ and $\rho:\mathbb{Z}^\omega\times (S^1)^\mathbb{N}\rightarrow S^1$ be the unique cocycle with $\rho(e_i,x) = x_i$. Then, the extension $(S^1)^\mathbb{N}\times_\rho S^1$ is an ergodic system of order $<3$ that is not an inverse limit of finite dimensional systems.\\

\textbf{Step 3:} The last and most technically difficult step in the proof of Theorem \ref{Main2:thm} is solving the following lifting problem.\\
Let $X=Z_{<k}(X)\times_\rho U$ be a finite dimensional ergodic system of order $<k+1$. Using a proof by induction and passing to an extension, we may assume that $Z_{<k}(X)=\mG/\Gamma$ is a nilpotent system. Let $\mG_{k}\leq \mG_{k-2}\leq...\leq\mG_2\leq \mG_1=\mG$ be the lower central series for $\mG$. We adapt an inductive argument of Host and Kra \cite{HK}, where in step $j$ we lift some elements from the group $\mG_{k-j+1}$ to transformations on $X$ which belongs to $\mG(X)$. The following difficulties arise in the case where $X$ is a finite dimensional system that is not a Toral system. 
\begin{enumerate}
    \item {The near-action defined by group generated by the connected component of $\mG(X)$ and $\{T_g:g\in G\}$ may not be transitive on $X$.}
    \item{The cocycles in the process may take values in a compact abelian group $U$ which is not necessarily a torus. In particular, it is difficult to apply the results from step $1$.}
\end{enumerate}
To deal with these difficulties we use extensions and $k$-extensions (see Definition \ref{kext}). For instance, in order to overcome the second difficulty, we prove the following result (Lemma \ref{fixCL:lem}): Let $G=\bigoplus_{p\in P}\mathbb{Z}/p\mathbb{Z}$, and $\rho:G\times X\rightarrow U$ be a cocycle into some finite dimensional group $U$. If $\chi\circ\rho \in PC_{<k}(G,X,S^1)\cdot B^1(G,X,S^1)$ for every $\chi\in\widehat U$, then there exists some $m=O_k(1)$ and an $m$-extension $\pi:Y\rightarrow X$ with the property that $\rho\circ\pi \in PC_{<k}(G^{(m)},Y,U)\cdot B^1(G^{(m)},Y,U)$. To deal with the first difficulty we also need extensions, but of different type. Roughly speaking, we show that for every countable set of transformations of bounded torsion in $\mG(X)$, there exists an extension $Y$ where these transformations has a lift in $\mG(Y)$. (see Lemma \ref{mainlem}, and section \ref{Main2:proof2} for more details).
We note that the passing to extensions of $X$ leads to other difficulties which we discuss in detail in section \ref{Main2:proof1}.\\
The reduction to the case where $X$ is a finite dimensional system is only necessary to ensure that $U$ is finite dimensional and $Y$ is an $m$-extension. This is no longer necessary when working in the generality of all countable abelian groups. Therefore, by following the arguments in this paper and assuming that Conjecture \ref{conj} holds we can prove a structure theorem for all countable abelian groups.
\begin{thm} [Structure theorem]
    Let $G$ be a countable abelian group and $k\geq 1$. Let $X$ be an ergodic $G$-system of order $<k+1$ and assume that Conjecture \ref{conj} holds. Then $X$ is a factor of a $k$-step nilpotent system $(Y,H)$, where $H$ is a countable abelian group extending $G$ and the homogeneous group $\mG(Y)$ is a $k$-step nilpotent Polish group.
\end{thm}
Since any countable abelian group is a quotient of $\mathbb{Z}^\omega$, one can always take $H=\mathbb{Z}^\omega$ in the theorem above.
The case $k=2$ of this theorem was established in \cite{OS2}. The proof of the theorem above follows by arguing as in sections \ref{Main2:proof1} and \ref{Main2:proof2} together with the counterpart of Lemma \ref{fixCL:lem} and Theorem \ref{extensionlowchar:thm} for abelian groups that is given in \cite[Proposition 3.8 and Theorem 3.14]{OS2}.
\subsection{A simple case of Theorem \ref{Main2:thm} and $k$-extensions.} \label{special case}
We emphasize the part of the proof where we used $k$-extensions. For the sake of the example, let $X$ be an arbitrary $G$-system of order $<4$. By Proposition \ref{abelext:prop}, there exist compact abelian groups $Z$, $U$ and $W$ and cocycles $\sigma$ and $\rho$ such that $$X=Z\times_\sigma U\times_\rho W.$$
In section \ref{inv:sec} we reduce matters to the case where $X$ is a finite dimensional system (see Definition \ref{fdsystems:def}) using inverse limits. Let us therefore assume that $Z$ and $U$ are finite dimensional groups and $W$ is a Lie group. Since $W$ is a Lie group, modifying the arguments of Host and Kra one can show that $X$ is (a factor of) a $3$-step nilpotent system with respect to its Host Kra group if and only if $Y=Z_{<3}(X)$ is (a factor of) a $2$-step nilpotent system with respect to its Host-Kra group. In that case the Host-Kra group takes the following convenient form:
$$\mG(Y) = \{ S_{s,F} : s\in Z, F\in \mathcal{M}(Z,U) \text{ s.t } \exists c\in \text{Hom}(G,U) \text{ with } \Delta_s\sigma(g,x) = c(g)\cdot \Delta_g F(x)\}$$
where $S_{s,F}:Y\rightarrow Y$ is the transformation $S_{s,F}(z,u) = (sz,F(z)u)$.\\

If $U$ is a torus, then for every $s\in Z$ there exists a character $c_s:G\rightarrow U$ and a measurable map $F_s:Z\rightarrow U$ such that $\Delta_s\sigma= c_s\cdot \Delta F_s$ (see Theorem \ref{type0}). Equivalently, this means that the transformation $s\in Z$ has a lift $S_{s,F_s}$ in $\mG(Y)$. If this holds for every $s\in Z$ then the action of $\mG(Y)$ on $Y$ is transitive and $Y$ is a nilpotent system. Below we discuss the case where $U$ is a finite dimensional group that is not a Lie group (If $U$ is a Lie group then it is a direct product of a torus and a finite group. The case where $U$ is finite is covered in \cite{OS}). 

Observe that if $\chi:U\rightarrow S^1$ is a character, then by the torus case there exist $c_{s,\chi}\in\text{Hom}(G,S^1)$ and $F_{s,\chi}\in \mathcal{M}(Z,S^1)$ such that \begin{equation} \label{CLcoordinate}\Delta_s \chi\circ\sigma(g,x)  = c_{s,\chi}(g)\cdot \Delta_g F_{s,\chi}(x).
\end{equation}
 By Pontryagin dual, $s$ has a lift in $\mG(Y)$ if and only if there is a choice of $F_{s,\chi}$ and $c_{s,\chi}$ for which equation (\ref{CLcoordinate}) holds and $\chi\mapsto F_{s,\chi}$ is a homomorphism. Equivalently, for every $s\in Z$ we can consider the map $k_s:\widehat U\times \widehat U \rightarrow P_{<2}(Z,S^1)$ where $$k_s(\chi,\chi') = \frac{ F_{s,\chi\cdot \chi'}}{F_{s,\chi}\cdot F_{s,\chi'}}.$$ The map $k_s$ defines an abelian multiplication on the Cartesian product $\widehat U \times P_{<2}(Z,S^1)$ by $(\chi,p)\cdot (\chi',p')  = (\chi\chi', k_s(\chi, \chi')pp')$. We denote this group by $U\times_{k_s} P_{<2}(Z,S^1)$ and observe that the short exact sequence 
 \begin{equation}\label{short} 1\rightarrow P_{<2}(Z,S^1)\rightarrow \widehat U\times_{k_s} P_{<2}(Z,S^1) \rightarrow \widehat U \rightarrow 1
 \end{equation} splits if and only if $s$ has a lift in $\mG(Y)$.
 
 Since $U$ is not a torus, $\widehat U$ is not a projective object in the category of discrete abelian groups. Moreover, it is not necessary that $P_{<2}(Z,S^1)$ is a divisible group (an injective object). In other words, other properties of $k_s$ must be used in order to prove that this extension always splits. Instead, we chose a different method which involves $k$-extensions.
 
Let $n=\dim U$, a careful analysis of the finite dimensional group $U$ shows that we can find a multiset of primes $P$ and vectors $v_p\in\mathbb{Z}^n$ for every $p\in P$ such that $\widehat U \cong \mathbb{Z}^n[\frac{1}{p}\cdot v_p]$. It is easy to find a cross-section from the subgroup $\mathbb{Z}^n$ to $\widehat U \times_{k_s} P_{<2}(X,S^1)$ and so it is left to find $p$-th roots for certain phase polynomials in $P_{<2}(X,S^1)$. We do not know whether or not such roots exist and therefore we use extensions. In section \ref{ext:sec} we prove, roughly speaking, that by extending $Y$ to a $2$-extension $Y'$, we can find a $p$-th root for any phase polynomial of degree $<2$ on $Z$ in $P_{<2}(Z_{<2}(Y'),S^1)$  (see Theorem \ref{extension:thm} or Theorem \ref{extensionlowchar:thm}). This means that by passing to an extension and replacing $P_{<2}(Z,S^1)$ with $P_{<2}(Z_{<2}(Y'),S^1)$, the short exact sequence (\ref{short}) splits and we can lift $s$ to $\mG(Y')$. Then, since we passed from $Y$ to $Y'$, we need to make sure that we can also lift all of the new transformations which arise from the extension $Z_{<2}(Y')\rightarrow Z$. A formal proof is given in section \ref{Main2:proof1} and \ref{Main2:proof2}.\\

\section{Conze-Lesigne type equation} \label{CL:sec}
Throughout, fix a multiset of primes $P$.
Let $m\geq 0$, let $G$ be a countable abelian group and denote by $Z^1_{<m}(G,X,S^1)$ the group of $(G,X,S^1)$-cocycles of type $<m$ and by $PC_{<m}(G,X,S^1)$ the phase polynomial cocycles of degree $<m$. It follows by Lemma \ref{PP} that \begin{equation}\label{subgroup}PC_{<m}(G,X,S^1)\cdot B^1(G,X,S^1)\leq Z^1_{<m}(G,X,S^1)
\end{equation}
The following theorem is the main result in this section.
\begin{thm} \label{countable}
    Let $X$ be an ergodic $\bigoplus_{p\in P}\mathbb{Z}/p\mathbb{Z}$-system. Then for every $m\geq 0$, the subgroup $PC_{<m}(\bigoplus_{p\in P}\mathbb{Z}/p\mathbb{Z},X,S^1)\cdot B^1(\bigoplus_{p\in P}\mathbb{Z}/p\mathbb{Z},X,S^1)$ is of at most countable index in $Z^1_{<m}(\bigoplus_{p\in P}\mathbb{Z}/p\mathbb{Z},X,S^1)$. 
\end{thm}
We recall some relevant results from previous work. In the case where $m=1$, we have the following theorem by Moore and Schmidt \cite{MS}.
\begin{lem} [Cocycles of type $<1$ are cohomologous to constants] \label{type0}
	Let $G$ be a countable abelian group, and let $X$ be an ergodic $G$-system. Suppose that $\rho:G\times X\rightarrow S^1$ is a cocycle of type $<1$. Then, there exists a character $c:G\rightarrow S^1$ such that $\rho$ is $(G,X,S^1)$-cohomologous to $c$. Equivalently, $Z^1_{<1}(G,X,S^1)= PC_{<1}(G,X,S^1)\cdot B^1(G,X,S^1)$.
\end{lem}
The following result of Host and Kra \cite[Corollary 7.9]{HK}  allow us to reduce matters to the case where $X$ is of order $<m$.
\begin{prop} \label{finiteorder}
Let $G$ be a countable abelian group and $X$ an ergodic $G$-system. If $m\geq 0$ and $\rho:G\times X\rightarrow U$ is a cocycle of type $<m$ into some compact abelian group $U$, then $\rho$ is $(G,X,U)$-cohomologous to a cocycle $\rho':G\times X\rightarrow U$ that is measurable with respect to $Z_{<m+1}(X)$.
\end{prop} From this and Theorem \ref{Main:thm} we conclude the following result.
\begin{thm}
    Let $m\geq 0$ and $P$ be a multiset of primes. If $X$ is an ergodic totally disconnected $\bigoplus_{p\in P}\mathbb{Z}/p\mathbb{Z}$-system (see Definition \ref{TD:def}) then every cocycle $\rho:\bigoplus_{p\in P}\mathbb{Z}/p\mathbb{Z}\times X\rightarrow S^1$ of type $<m$ is $(\bigoplus_{p\in P}\mathbb{Z}/p\mathbb{Z},X,S^1)$-cohomologous to a phase polynomial of degree $<d$ for some $d=O_m(1)$.
\end{thm}
If hypothetically we knew that the quantity $d$ in the theorem equals to $m$, then this would imply Theorem \ref{countable} for a totally disconnected system $X$. In fact, in this case $$PC_{<m}(\bigoplus_{p\in P}\mathbb{Z}/p\mathbb{Z},X,S^1)\cdot B^1(\bigoplus_{p\in P}\mathbb{Z}/p\mathbb{Z},X,S^1) = Z^1_{<m}(\bigoplus_{p\in P}\mathbb{Z}/p\mathbb{Z},X,S^1)$$.\\
In order to deal with the fact that $d$ is potentially higher than $m$ we need the following generalization of Lemma \ref{sep:lem}.
\begin{lem} \label{sep2:lem}
	Let $X$ be an ergodic $G$-system. Let $d\geq m\geq 0$ and $p:G\times X\rightarrow S^1$ be a phase polynomial of degree $<d$ and type $<m$. Write $d^{[m]}p=\Delta Q$, and let $r=d-m$. If $\|Q-1\|_{L^2(X^{[m]},\mu^{[m]})}<\sqrt{2}/2^{r+d-1}$ then $p$ is a phase polynomial of degree $<m$.
\end{lem}
\begin{proof}
If $m=0$ then the claim follows by Lemma \ref{sep:lem}. Assume that $m\geq 1$. We have,\\
	\textbf{Claim}: Let $P:X^{[m]}\rightarrow S^1$ be a phase polynomial of degree $<r$ and suppose that \begin{equation}
	\label{distance}
\|P-1\|_{L^2(X^{[m]},\mu^{[m]})}<\sqrt{2}/2^{r-1}	\end{equation} then $P$ is invariant with respect to the diagonal action of $G$ on $X^{[m]}$ if and only if for every $(m-1)$-face $\alpha$ and $g\in G$, $\Delta_{g_\alpha^{[m]}} P$ is invariant with respect to that action.
	\begin{proof}
    Let $g\in G$ and $\alpha$ be an $(m-1)$-face. Since $g_\alpha^{[m]}$ is measure preserving (Lemma \ref{facetransformations}) and commutes with $h^{[m]}$ the first direction follows. We prove the other direction. By the ergodic decomposition theorem there exists a probability measure $(\Omega_m,P_m)$ such that \begin{equation} \label{decomposition}\mu^{[m]}=\int_{\Omega_m} \mu_\omega dP_m(\omega)
		\end{equation}
	Let $\alpha$ be an $(m-1)$-dimensional face. By Lemma \ref{ergdec}, the transformation $g_\alpha^{[m]}$ defines an isomorphism of ergodic components and an action on $(\Omega_m,P_m)$. Moreover, by the same lemma, the action generated by these transformations for every $g\in G$ and $(m-1)$-dimensional face $\alpha$ is ergodic. Let $$A=\{\omega\in \Omega_m :\Delta P = 1 \text{ }\mu_{\omega}\text{-a.e} \}.$$
		Since $\Delta_{g_{\alpha}^{[m]}} P$ is invariant, we have that $\omega\in A$ if and only if ${g_\alpha^{[m]}}w\in A$. In other words, $A$ is invariant. On the other hand, from (\ref{decomposition}) and (\ref{distance}) we conclude that the set
		 $$B:=\{\omega\in \Omega_m : \|P-1\|_{L^2(\mu_\omega)}<\sqrt{2}/2^{r-1}\}$$ is of positive measure with respect to $P_m$. Since $\mu_{\omega}$ is ergodic it follows by Lemma \ref{sep:lem} that $B\subseteq A$ and therefore by ergodicity (Lemma \ref{ergdec}) $A$ is of measure $1$. This proves the claim.
	\end{proof}
	We return to the proof of the lemma. For any $g_1,...,g_d\in G$ and $(m-1)$-dimensional faces, $\alpha_1,...,\alpha_d$ we have,
	$$d_{\alpha}^{[m]} \Delta_{g_1}...\Delta_{g_d}p = \Delta \Delta_{{g_1}_{\alpha_1}^{[m]}}...\Delta_{{g_d}_{\alpha_d}^{[m]}} Q$$ where $\alpha$ is the intersection of $\alpha_1,...,\alpha_d$. Since $p$ is of degree $<d$ we conclude that
		 $\Delta_{{g_1}_{\alpha_1}^{[m]}}...\Delta_{{g_d}_{\alpha_d}^{[m]}} Q$ is invariant with respect to the diagonal action of $G$. The claim above imply that if $d$-derivatives of $Q$ are invariant and these derivatives are sufficiently close to $1$, then only $d-1$ derivatives of $Q$ are invariant. Repeating this claim iteratively $d$ times, we deduce that when $Q$ sufficiently small (as in the lemma), it is invariant with respect to that action. Since $d^{[m]}p = \Delta Q=1$, Lemma \ref{PP} implies that $p$ is of degree $<m$ as required.
\end{proof}
We have the following reduction of theorem \ref{countable}.
\begin{lem} \label{reductioncountable}
Let $G=\bigoplus_{p\in P}\mathbb{Z}/p\mathbb{Z}$.  In order to prove Theorem \ref{countable} it is enough to show that for some $d>m$, $$PC_{<d}(G,X,S^1)\cdot B^1(G,X,S^1)\cap Z^1_{<m}(G,X,S^1) \leq Z^1_{<m}(G,X,S^1)$$ is of at most countable index.
\end{lem}
\begin{proof}
Let $\mathcal{W}\subseteq L^2(X^{[m]})$ denote the space of phase polynomials of degree $<d-m+1$ in $X^{[m]}$. Since $L^2(X^{[m]},\mu^{[m]})$ is separable, we can decompose $\mathcal{W}$ into a countable union of balls $\{B_i\}_{i\in\mathbb{N}}$ of diameter $<\sqrt{2}/2^{2d-m+1}$. For each ball, we choose (if exists) a cocycle $\rho_i\in P_{<d}(G,X,S^1)\cdot B^1(G,X,S^1)$ such that $\rho_i = p_i\cdot \Delta F_i$ for a measurable map $F_i:X\rightarrow S^1$ and a phase polynomial cocycle $p_i\in P_{<d}(G,X,S^1)$ which satisfies that $d^{[m]}p_i=\Delta Q_i$ where $Q_i\in B_i$. We conclude that for every $\rho\in PC_{<d}(G,X,S^1)\cdot B^1(G,X,S^1)$, there exists $i$ such that $\rho/\rho_i\in PC_{<m}(G,X,S^1)\cdot B^1(G,X,S^1)$. Indeed, write $\rho = p\cdot \Delta F$ for some $p\in PC_{<d}(G,X,S^1)$. Since $p$ is cohomologous to $\rho$, it is of type $<m$. Therefore, we can write $d^{[m]} p = \Delta Q$ for some phase polynomial $Q:X^{[m]}\rightarrow S^1$ of degree $<d-m+1$ (by lemma \ref{PP}). We conclude that $Q\in \mathcal{W}$ and there exists $i$ such that $Q\in B_i$. Let $\rho_i$ as above. We conclude that $\rho/\rho_i = p/p_i\cdot \Delta F/F_i$ and $d^{[m]} p/p_i = \Delta Q/Q_i$. Since $$\|Q/Q_i - 1 \|_{L^2(\mu^{[m]})} = \|Q-Q_i\|_{L^2(\mu^{[m]})}  < \sqrt{2}/2^{2d-m+1}$$ Lemma \ref{sep2:lem} implies that $p/p_i$ is of degree $<m$. It follows that $PC_{<m}(G,X,S^1)\cdot B^1(G,X,S^1)$ is of countable index in $PC_{<d}(G,X,S^1)\cdot B^1(G,X,S^1)\cap Z^1_{<m}(G,X,S^1)$. By the assumption, the latter is of at most countable index in $Z^1_{<m}(G,X,S^1)$. We conclude that $PC_{<m}(G,X,S^1)\cdot B^1(G,X,S^1)$ is of countable index in $Z^1_{<m}(G,X,S^1)$, as required.
\end{proof}
The main difficulty in the proof of Theorem \ref{countable} is therefore to reduce matters to the case where $X$ is totally disconnected. Before we turn to the proof of Theorem \ref{countable} we prove some corollaries. Since we prove Theorem \ref{countable} by induction on the order of $X$ we will be able to use these corollaries for systems of smaller order.
\subsection{Corollaries} Throughout and unless specified otherwise $G=\bigoplus_{p\in P}\mathbb{Z}/p\mathbb{Z}$. We begin with the following counterpart of \cite[Lemma 9.2]{HK}.
\begin{thm} \label{HK1}
	Let $m\geq 0$ be a natural number and $X$ be an ergodic $G$-system. Let $(\Omega,P)$ be a probability space and $\omega\mapsto \rho_\omega$ be a measurable map from $\Omega$ into $Z_{<m}^1(G,X,S^1)$. Then, there exists a set of positive measure $\mathcal{A}\subset \Omega$, such that $\rho_\omega/\rho_{\omega'}\in PC_{<m}(G,X,S^1)\cdot B^1(G,X,S^1)$  for every $\omega,\omega'\in \mathcal{A}$.
\end{thm}
We need some notation: An analytic subset of a measurable space $X$ is the continuous image of a Polish space in $X$. Lusin separation theorem \cite[Theorem 14.7]{Ke} implies that if $A$ and $X\backslash A$ are analytic, then $A$ is Borel measurable.
\begin{proof}
Observe that since the analytic set $PC_{<m}(G,X,S^1)\cdot B^1(G,X,S^1)$ is of at most countable index in $Z_{<m}^1(G,X,S^1)$ then by the separation theorem it is Borel. Therefore the map $\omega\mapsto \rho_\omega \cdot PC_{<m}(G,X,S^1)\cdot B^1(G,X,S^1)$ is a measurable map into a countable set. We conclude that there exists a measurable set $\mathcal{B}\subseteq\Omega$ of positive measure such that for every $\omega,\omega'\in \mathcal{B}$, $\rho_\omega$ and $\rho_{\omega'}$ belong to the same co-set. This completes the proof. 
\end{proof}
As a corollary we have the following result.
\begin{thm} \label{CL:thm}
	Let $m\geq 0$, $X$ be an ergodic $G$-system and $U$ be a compact abelian group which acts on $X$ by automorphisms. Let $\rho:G\times X\rightarrow S^1$ be a cocycle and suppose that for every $u\in U$, $\Delta_u \rho$ is of type $<m$. Then there exists an open subgroup $U'\leq U$ such that $\Delta_u \rho\in PC_{<m}(G,X,S^1)\cdot B^1(G,X,S^1)$ for all $u\in U'$.
\end{thm}
\begin{proof}
	From the cocycle identity it is easy to see that 
	$$U' = \{u\in U : \Delta_u \rho\in PC_{<m}(G,X,S^1)\cdot B^1(G,X,S^1)\}$$
	is a subgroup of $U$. We use Theorem \ref{HK1} with $\Omega = U$ and $\rho_u = \Delta_u \rho$. We see that there exists a set of positive measure $\mathcal{A}\subseteq U$ such that  $\Delta_u \rho/\Delta_{u'}\rho\in PC_{<m}(G,X,S^1)\cdot B^1(G,X,S^1)$ for all $u,u'\in\mathcal{A}$. A direct computation shows that $\Delta_{uu'^{-1}} \rho = V_{u'^{-1}}\Delta_u \rho/\Delta_{u'}\rho$ for every $u,u'\in U$. Since $U$ commutes with $G$ we conclude that $\mathcal{A}\cdot\mathcal{A}^{-1}\subseteq U'$. Therefore, by Lemma \ref{A.Weil}, $U'$ is an open subgroup and the proof is complete. 
\end{proof}
Given a compact abelian group $U$ and an integer $m$ we define $U^m :=\{u^m : u\in U\}$. Observe that the subgroup $U'$ in the previous theorem depends on the cocycle $\rho$. In the next lemma we compute this group for a root of $\rho$.
\begin{lem} \label{Clroot:lem}
      Let $X$ be an ergodic $G$-system. Let $U$ be a compact abelian group which acts on $X$ by automorphisms and let $n,m,d\in\mathbb{N}$. If $\rho:G\times X\rightarrow S^1$ is a cocycle of type $<m$ and  $\Delta_u\rho^n \in PC_{<d}(G,X,S^1)\cdot B^1(G,X,S^1)$ for every $u\in U$, then $\Delta_u \rho \in PC_{<d}(G,X,S^1)\cdot B^1(G,X,S^1)$ for every $u\in U^{nm}$.
\end{lem}
\begin{proof}
    The claim follows immediately by induction on $m$ and the following identity
    $$\Delta_{u^n} \tilde{\rho} = \Delta_u \tilde{\rho}^n \cdot \prod_{k=0}^{n-1} \Delta_u \Delta_{u^k} \tilde{\rho}$$ We have that $\Delta_u \tilde{\rho}^n\in PC_{<d}(G,X,S^1)\cdot B^1(G,X,S^1)$ and for every $1\leq k \leq n$, $\Delta_u \Delta_{u^k} \tilde{\rho}$ is of smaller type. 
\end{proof}
\subsection{The proof of Theorem \ref{countable}} We briefly explain the method in the proof. We prove the claim by induction on $m$; The case $m=0$ is trivial, and the case $m=1$ follows from Theorem \ref{type0}. Fix $m\geq 2$ and assume inductively that the theorem holds for smaller values of $m$. Let $\rho:G\times X\rightarrow S^1$ a cocycle of type $<m$. By Proposition \ref{finiteorder} we can assume without loss of generality that $X$ is of order $<m+1$. By Proposition \ref{abelext:prop} we can find compact abelian groups $U_0,U_1,U_2,...,U_{m}$ where $U_0$ is trivial and $$X=U_0\times_{\rho_1}U_1\times...\times_{\rho_{m}}U_{m}.$$
We construct a sequence of factors $$X=X_m\rightarrow X_{m-1}\rightarrow X_{m-2}\rightarrow...\rightarrow X_0$$ where in each step we quotient out the connected component of the identity in the next structure group. The last factor, $X_0$ is a totally disconnected system.\\
Observe that the factor maps define a sequence of injections
$$Z_{<m}^1(G,X_0,S^1)\hookrightarrow Z_{<m}^1(G,X_1,S^1)\hookrightarrow...\hookrightarrow Z_{<m}^1(G,X,S^1)$$
Adapting the arguments of Host and Kra, we show that the image of each of these embeddings is of at most countable index in the next group. Then we apply Theorem \ref{Main:thm} to the system $X_0$ and complete the proof.\\

One difficulty which arise in this process is that the connected component of the identity of the structure groups may not act on $X$ by automorphisms. For this reason we study under which conditions we can lift an automorphism from $Z_{<k}(X)$ to $Z_{<k+1}(X)$ for every $1\leq k \leq m$. We have,
\begin{lem} [Going up] \label{up}
	Let $X$ be an ergodic $G$ systems. Let $U$ be a compact abelian group and $Y=X\times_\rho U$ be an extension of $X$ by a cocycle $\rho:G\times X\rightarrow U$. Let $A$ be a connected compact abelian group of automorphisms of $X$ and suppose that for every $a\in A$, $\Delta_a \rho \in B^1(G,X,U)$. Then, there exists a compact connected abelian group of automorphisms $\tilde{A}$ of $Y$ such that the induced action of $\tilde{A}$ on $X$ coincides with the action of $A$.
\end{lem}
\begin{proof}
	Let $X,\rho,A$ as above. For every $a\in A$ and a measurable map $F:X\rightarrow U$ we define a measure preserving transformation $S_{a,F}$ on $X\times U$ by $S_{a,F}(x,u) := (ax,F(x)u)$. Direct computation shows that the group $$\mathcal{K}:= \{S_{a,F}: \Delta_a \rho = \Delta F\}$$ acts on $X\times U$ by automorphisms. Indeed, $$S_{a,F}T^Y_g (x,u) = (a\cdot T_g^X x, F(T_gx)\rho(g,x)u) = (T_g^X ax, F(x)\rho(g,ax) u) = T_g^Y S_{a,F}(x,u).$$  Equipped with the topology of convergence in probability $\mathcal{K}$ is a Polish group with respect to the multiplication $S_{a,F}\circ S_{a',F'} = S_{aa',F' V_{a'} F}$ (see e.g. \cite[Appendix A]{HK}). By the assumption the projection map $p:\mathcal{K}\rightarrow A$ is onto. Moreover, by ergodicity we see that the kernel of $p$ is isomorphic to $U$. In particular, it follows from Corollary \ref{compact} that $\mathcal{K}$ is a compact Polish group. Finally, since $A$ is abelian, direct computation reveals that $\mathcal{K}$ is $2$-step nilpotent. We let $\tilde{A}$ be the connected component of $\mathcal{K}$. By all of the above and Proposition \ref{connilabel:prop} we deduce that $\tilde{A}$ is a compact connected abelian group. Since $p$ is open (Theorem \ref{open}) it maps the connected group $\tilde{A}$ onto $A$. This completes the proof.
\end{proof}
Given a system $X$ and a group $A$ acting freely on $X$. We define the quotient space $X/A$ to be the space of all equivalent classes $[x]:=\{ax:a\in A\}$ with the quotient $\sigma$-algebra. We let the measure $\mu_{X/A}$ be the push-forward of $\mu_X$ under the factor map $\pi:X\rightarrow X/A$, $\pi(x)=[x]$. Finally, if $gAg^{-1}\subseteq A$ for every $g\in G$ then the action of $G$ on $X/A$ by $T_g[x] = [T_gx]$ is well defined.
\begin{lem} [Going down] \label{down}
	Let $Y=X\times_{\rho} U$ be an ergodic abelain extension of a $G$-system $X$ by a compact abelian group. Let $A$ be a compact connected abelian group of automorphisms of $X$ and suppose that $A$ acts freely on $X$ and $\Delta_a\rho\in B^1(G,X,S^1)$ for every $a\in A$. Then $\tilde{A}$ from the previous lemma acts freely on $Y$ and the factor $Y'=Y/\tilde{A}$ is an extension of $X':=X/A$ by some quotient of $U$.
\end{lem}
\begin{proof}
	Let $Y'$ as in the lemma. The $G$-action on $Y'$ is given by $g[y] = [gy]$ where $[y]=\{ay:a\in \tilde{A}\}$. This action is well defined since the action of $\tilde{A}$ on $Y$ commutes with the action of $G$. Let $\mathcal{K}$ be as in the proof of the lemma above, and $p:\mathcal{K}\rightarrow A$ the projection map. As before we can identify $U$ with the kernel of $p:\mathcal{K}\rightarrow A$. We show that $Y'\cong X'\times_\sigma U/V$ for some cocycle $\sigma$ where $V=\tilde{A}\cap U$.\\
	
	Let $\overline{\rho}:G\times X\rightarrow U/V$ be the composition of $\rho$ with the projection map $U\rightarrow U/V$ and consider the extension $X\times_{\overline{\rho}} U/V$. For every $a\in A$ choose a measurable cross-section $a\mapsto F_a$ such that $S_{a,F_a}\in\tilde{A}$. Since, 
	\begin{equation} \label{down:eq1} \Delta_ a \rho = \Delta F_a
	\end{equation} the cocycle identity implies that $\frac{F_{aa'}}{F_a V_a F_a'}$ is a constant. Since $S_{1,\frac{F_{aa'}}{F_a V_a F_a'}}\in \tilde{A}$ we conclude that $\frac{F_{aa'}}{F_a V_a F_a'}\in V$. Now, let $\overline{F}_a$ be the projection of $F_a$ to $U/V$. It follows that $\overline{F}_{aa'} = \overline{F}_a V_a \overline F_{a'}$. Since $A$ acts freely on $X$ we can write $X=X_0 \times A$ measurably. Choose some generic point $a_0\in A$ and set $F(x,aa_0):= \overline{F}_a (x,a_0)$. A direct computation reveals that $\Delta_a F = \overline{F}_a$. From equation (\ref{down:eq1}) we conclude that $\overline{\rho}/\Delta F$ is invariant to $A$. Let $\sigma:G\times X'\rightarrow U/V$ be the push-forward of $\overline{\rho}/\Delta F$ and let $\tilde{F}$ be any measurable lift of $F$ to a map into $U$. Then, for every $x\in X$ we have that $\tilde{F}(ax)/\tilde{F}(x)\cdot F_a(x)\in V$. This implies that $\pi([x]_A,uV) = [(x,\tilde{F}(x)^{-1}u)]_{\tilde{A}}$ is a well defined isomorphism from $X'\times_{\sigma} U/V$ to $Y'$. Since $A$ acts freely on $X$ we can write $X\cong X'\times A$ and therefore $Y\cong X'\times A \times U/V\times V$. By identifying $A\times V$ (measurably) with $\tilde{A}$ we see that $\tilde{A}$ acts freely on $Y$, as required.
\end{proof}
Next we modify Host and Kra's argument from \cite[Proposition 8.9]{HK} for connected groups in the context of $\bigoplus_{p\in P}\mathbb{Z}/p\mathbb{Z}$-actions.
\begin{lem} \label{HKargcon}
    Let $X$ be a $G$-system of order $<m+1$ for some $m\geq 0$ and let $U$ be a connected compact abelian group which acts freely on $X$ by automorphisms. We abuse notation and identify $Z^1_{<m}(G,X/U,S^1)$ with a subgroup of $Z^1_{<m}(G,X,S^1)$. Then, $Z^1_{<m}(G,X/U,S^1)\cdot B^1(G,X,S^1)$ is of at most countable index in $Z^1_{<m}(G,X,S^1)$.
\end{lem}
\begin{proof}
    Equip $\mathcal{M}(X,S^1)$ with the $L^2$ topology. For every cocycle $\rho\in Z^1_{<m}(G,X,S^1)$ we consider the group $$\mathcal{H}_\rho:=\{S_{u,F}: u\in U, F\in \mathcal{M}(X,S^1), \exists p\in P_{<m}(G,X,S^1) \text{ s.t } \Delta_u\rho = p_u\cdot \Delta F\}.$$ Equipped with the topology of convergence in probability $\mathcal{H}_\rho$ is a Polish group and we have a short exact sequence $$1\rightarrow P_{<m+1}(X,S^1)\rightarrow\mathcal{H}_\rho\rightarrow U\rightarrow 1.$$ By Corollary \ref{compact},  $\mathcal{H}_\rho$ is locally compact. Moreover, since $U$ is connected, $\mathcal{H}_\rho$ is $2$-step nilpotent. To see this observe that if $S_{u,F},S_{u',F'}\in \mathcal{H}_\rho$, then $[S_{u,F}, S_{u',F'}] = S_{1,\frac{\Delta_{u'}F}{\Delta_u F'}}$. Since phase polynomials cocycles are invariant with respect to translations by connected groups (Proposition \ref{pinv:prop}), we conclude that $$\Delta \frac{\Delta_u F_u'}{\Delta_{u'} F_u} = \frac{\Delta_u \Delta_{u'} \rho}{\Delta_{u'} \Delta_u \rho} = 1.$$ Therefore by ergodicity $\frac{\Delta_u F_u'}{\Delta_{u'} F_u}$ is a constant and $S_{1,\frac{\Delta_{u'}F}{\Delta_u F'}}$ is in the center of $\mathcal{H}_\rho$.\\
    
	Let $\mathcal{M}(X,S^1)$ denote the space of all measurable maps on $X$ with values in $S^1$, equipped with the topology of convergence in probability. Let $\mathcal{F}$ be the group of all continuous maps from $U$ to $\mathcal{M}(X,S^1)/P_{<m+1}(X,S^1)$ where $\mathcal{M}(X,S^1)/P_{<m+1}(X,S^1)$ is equipped with the quotient metric (i.e. $d(f,g) = \inf_{p\in P_{<m+1}(X,S^1)} d_{\mathcal{M}(X,S^1)}(f-p,g)$). Equipped with the supremum metric, $\mathcal{F}$ is a Polish group. We define a map $\Phi_\rho\in\mathcal{F}$ by giving each $u\in U$ the equivalence class of $F_u$ in $\mathcal{M}(X,S^1)/P_{<m+1}(X,S^1)$.
If $\Phi_{\rho}$ is sufficiently small (say $\|\Phi_{\rho}\|_{\mathcal{F}}<\frac{1}{20^m}$) we show that we can linearize the term $u\mapsto F_u$:\\
	For such $\rho$ we can define a subset $\mathcal{K}\leq\mathcal{H}_\rho$ by $$\mathcal{K}:= \{S_{s,F}: \text{ There exists } c\in S^1 \text{ such that } |F-c|\leq \frac{1}{10^m}\}.$$
A direct computation shows that $\mathcal{K}$ is a closed subgroup of $\mathcal{H}_\rho$ (see \cite[Proposition 8.9]{HK} for the details). Observe that since $\|\Phi_{\rho}\|_{\mathcal{F}}<\frac{1}{20^m}$ the projection $p_\mathcal{K}:\mathcal{K}\rightarrow U$ is onto. Since $U$ is connected and $p$ is open (Theorem \ref{open}) the same holds if we restrict ourselves to the connected component of the identity $\mathcal{K}_0$ of $\mathcal{K}$. Since $\mathcal{K}_0$ is $2$-step nilpotent and connected, it is abelian (Proposition \ref{connilabel:prop}) and so it splits as $\mathcal{K}_0\cong S^1\times U$. In other words, for every $u\in U$ we can find $F_u$ such that $(u,F_u)\in\mathcal{K}_0$ and $F_{uv} = F_u V_u F_v$.\\ Since the group $U$ acts freely on $X$ we can write $X=Y\times U$. Fix any generic point $u_0\in U$ and define $F(y,uu_0):=F_u(y,u_0)$ for all $y\in Y, u\in U$. It follows that, 
	\begin{equation}
	\label{pu}
\Delta_u (\rho/\Delta F)=p'_u
	\end{equation}
	 for some phase polynomial $p'_u\in P_{<m}(G,X,S^1)$.\\
	 
	 The phase polynomial term $p_u'$ is in fact trivial. To see this notice that $u\mapsto p'_u$ is a cocycle. Since $U$ is connected, by Theorem \ref{pinv:prop} $u\mapsto p'_u$ is a homomorphism. Since there are only countably many polynomials up to constants and $U$ is connected we conclude that $p'_u$ is a constant in $x$. Finally, since $p'_u$ is a $(G,X,S^1)$-cocycle and a constant in $x$, it can be identified with an element in $\widehat G$. Therefore, $u\mapsto p'_u$ is a continuous homomorphism from $U$ to $\widehat G$, hence trivial. We conclude from (\ref{pu}) that $\Delta_u(\rho/\Delta F)$ is a coboundary for every $u\in U$. By Lemma \ref{cob:lem}, $\rho$ is $(G,X,S^1)$-cohomologous to a function that is invariant under $U$.\\
	
	Now, since $\mathcal{F}$ is separable we can decompose $\mathcal{F}$ as a union of countably many balls $\{B_i\}_{i\in\mathbb{N}}$ of diameter $<1/20^m$. For each ball $B_i$ choose (if exists) a cocycle $\rho_i$ of type $<m$ such that $\Phi_{\rho_i}\in B_i$. We conclude that if $\rho$ is an arbitrary cocycle of type $<m$, then there exists $\rho_i$ such that $\|\phi_{\rho/\rho_i}\|_{\mathcal{F}}<1/20^m$. Therefore $\rho/\rho_i$ is cohomologous to a cocycle which is invariant to the action of $U$. By Lemma \ref{Cdec:lem} the push-forward of $\rho/\rho_i$ to $X/U$ is a cocycle of type $<m$. This completes the proof.
\end{proof}
It is left to prove Theorem \ref{countable}.
\begin{proof}
We already considered the cases 
$m=0$ and $m=1$ above. Fix $m\geq 2$ and let $X$ be as in the theorem. By Proposition \ref{finiteorder} we can assume that the $G$-system $X$ is of order $<m+1$. Therefore, by proposition \ref{abelext:prop} $X$ can be written as $X=U_0\times_{\rho_1}U_1\times...\times_{\rho_{m}}U_{m}$ for some compact abelian groups $U_0,...,U_{m}$ and cocycles $\rho_1,...,\rho_m$. Let $l=l(X)$ denote the smallest number for which $U_{l+1},...,U_{m}$ are totally disconnected. We prove the claim by induction on $l$. If $l=0$ then $X$ is a totally disconnected system. In this case the proof follows by Theorem \ref{Main:thm} and Lemma \ref{reductioncountable}. Fix $l\geq 1$ and suppose that the claim holds for all smaller values of $l$. Let $U_{l,0}$ be the connected component of the identity of $U_l$ and recall that $U_{l,0}$ acts freely on $Z_{<l+1}(X)$ by vertical rotations. In particular, if $l=m$ then $U_{l,0}$ acts by automorphisms on $X$. Otherwise suppose that $l<m$. In this case we lift $U_{l,0}$ to a group of automorphism of $X$ using Lemma \ref{up}. We argue as follows: let $\chi\in \widehat U_{l+1}$, using the induction hypothesis we know that the claim in Theorem \ref{countable} holds for $Z_{<l}(X)$ and so we can apply Theorem \ref{CL:thm}. We conclude that there exists a phase polynomial $p_u\in P_{<k}(G,X,S^1)$ (in fact we can take $p_u$ of degree $<1$) and a measurable map $F_u:X\rightarrow S^1$  such that
 \begin{equation} \label{eq} \Delta_u \chi\circ \rho_{l+1}=p_u\cdot \Delta F_u\end{equation}
 for every $u\in U_{l,0}$. By assumption, $U_{l+1}$ is totally disconnected and therefore there exists some $n\in\mathbb{N}$ such that $\chi^n=1$ (Corollary \ref{chartdg}). Let $u\in U_{l,0}$, the cocycle identity gives that $$\Delta_{u^n} \chi\circ \rho_{l+1} = \left(\Delta_u \chi\circ\rho_{l+1}\right)^n \cdot \prod_{k=0}^{n-1}\Delta_{u^k}\Delta_u \chi\circ \rho$$ Since $\chi^n=1$, the first term in the right hand side of the equation above vanishes. By Proposition \ref{pinv:prop} and equation (\ref{eq}) the other term (the product) is a coboundary. Therefore, we see that $p_{u^n}$ is a $(G,X,S^1)$-coboundary for every $u\in U_{l,0}$. Since connected groups are divisible this implies that $p_u$ is a coboundary for every $u\in U_{l,0}$. From this and equation (\ref{eq}) we see that $\chi(\Delta_u \rho_{l+1})$ is a coboundary for every $u\in U_{l,0}$.  As $\chi$ was arbitrary Theorem \ref{MScob} implies that $\Delta_u \rho_{l+1}$ is a $(G,Z_{<l+1}(X),U_{l+1})$-coboundary.\\
 Therefore we are in a situation as in Lemma \ref{up} and so we can lift $U_{l,0}$ to a group of automorphisms on $Z_{<l+2}(X)$. Repeating this argument iteratively we conclude that we can lift $U_{l,0}$ to a compact abelian connected group of automorphisms on $X=Z_{<m+1}(X)$. We denote this group by $\mathcal{H}_l$ and let $X'=X/\mathcal{H}_l$. Now, by applying Lemma \ref{down} iteratively we see that $\mathcal{H}_l$ acts freely on $X$ and $l(X')=l-1$. Therefore, by Lemma \ref{HKargcon} we have that $Z^1_{<m}(G,X',S^1)\cdot B^1(G,X,S^1)$ is of at most countable index in $Z^1_{<m}(G,X,S^1)$. Moreover, by the induction hypothesis $PC_{<m}(G,X',S^1)\cdot B^1(G,X',S^1)$ is of at most countable index in $ Z^1_{<m}(G,X',S^1)$. We conclude that $PC_{<m}(G,X',S^1)\cdot B^1(G,X,S^1)$ is of at most countable index in $Z^1_{<m}(G,X,S^1)$ (we identify $PC_{<m}(G,X',S^1)$ with a subgroup of $PC_{<m}(G,X,S^1)$ using the factor map). Finally, by Proposition \ref{pinv:prop} phase polynomial cocycles are invariant with respect to the action of connected groups and therefore $PC_{<m}(G,X',S^1) = PC_{<m}(G,X,S^1)$ and the claim follows.
\end{proof}

\section{Inverse limit of finite dimensional systems} \label{inv:sec}
We begin with the following definition of a finite dimensional system. The main result in this section (Theorem \ref{InvFD:thm}) asserts that every ergodic $G$-system of order $<k$ is an inverse limit of these systems.
\begin{defn} [Finite dimensional system] \label{fdsystems:def}
	Let $k\geq 1$. An ergodic $G$-system $X$ of order $<k$ is called a finite dimensional system if for every $1\leq r \leq k-1$ the system $Z_{<r+1}(X)$ is an extension of $Z_{<r}(X)$ by a finite dimensional compact abelian group. 
\end{defn}
Note that by Proposition \ref{abelext:prop} this means that we can write $X=U_0\times_{\rho_1}U_1\times...\times_{\rho_{k-1}} U_{k-1}$ where $U_0,U_1,...,U_{k-1}$ are finite dimensional compact abelian groups.

\noindent We are particularly interested in finite dimensional systems which also have a finite \textit{exponent}.
\begin{defn} [exponent of a finite dimensional system] \label{exponent}
Let $m\geq 0$.
\begin{itemize}
    \item {A totally disconnected group $\Delta$ is said to be of exponent $m$ if for any prime $p$, the $p$-sylow subgroup of $\Delta$ is a $p^m$-torsion group. Equivalently, by theorem \ref{torsion}, $\Delta \cong \prod_{p\in I} C_{p^{d_p}}$ for some multiset of primes $I$ and $d_p\leq m$.}
    \item{We say that a compact abelian finite dimensional group $U$ is of exponent $m$ if there exists a closed totally disconnected subgroup $\Delta$ of exponent $m$ such that $U/\Delta$ is a Lie group.}
    \item{A finite dimensional system $X$ is of exponent $m$ if the structure groups are of exponent $m$.}
\end{itemize}
\end{defn}
We prove that every ergodic $G$-system of order $<k$ is an inverse limit of finite dimensional systems of some bounded exponent.
\begin{thm} [Systems of order $<k$ are inverse limits of finite dimensional systems] \label{InvFD:thm}
	Let $X$ be an ergodic $G$-system of order $<k$. There exists $m=O_k(1)$ and a sequence $X_n$ of increasing factors of $X$ such that for each $n\in\mathbb{N}$, $X_n$ is a finite dimensional system of exponent $m$ and $X$ is the inverse limit of the sequence $X_n$.
\end{thm}

Let $X=Z_{<k-1}(X)\times_\rho U$ be an ergodic $G$-system of order $<k$. Since every compact abelian group is an inverse limit of compact abelian Lie groups (Theorem \ref{GY:thm}) we can assume that $U$ is a torus times a finite group (Theorem \ref{structureLieGroups}). We note that in general replacing a structure group of $X$ with one of its quotients will not necessarily be a factor of $X$ and therefore this approximation is only possible for the last structure group.\\
In the next lemma we study cocycles with values in a Lie group. By taking coordinates it is enough to study cocycles into the torus and into a finite cyclic group.
\begin{lem} \label{step 1}
	Let $X$ be an ergodic $G$-system of order $<k$. Suppose that $U$ is a compact abelian group which acts freely on $X$ by automorphisms. Let $H$ be either $S^1$ or $C_{p^n}$ for some prime $p$ and a natural number $n$ and let $\rho:G\times X\rightarrow H$ be a cocycle of type $<m$. Then there exists a subgroup $V\leq U$ such that $U/V$ is a finite dimensional group of exponent $m$ and $\rho$ is $(G,X,H)$-cohomologous to a cocycle that is invariant with respect to the action of $V$.
\end{lem}

\begin{proof} If $m=0$ then $\rho$ is a coboundary and the claim follows. We assume that $m\geq 1$. By embedding $H$ in $S^1$ (if necessary), and applying Theorem \ref{CL:thm} we see that there exists an open subgroup $U'\leq U$ such that for every $u\in U$ we have $$\Delta_u \rho = p_u\cdot \Delta F_u$$ for some phase polynomial $p_u\in P_{<m}(G,X,S^1)$ and a measurable map $F_u$.\\
	Using Lemma \ref{lin:lem} and then Theorem \ref{GY:thm} we can find a closed subgroup $J\leq U'$ such that $U'/J$ and $U/J$ are Lie groups and $p_{jj'}=p_j V_j p_{j'} = p_j p_{j'}\cdot \Delta_j p_{j'}$ for every $j,j'\in J$. Since $\Delta_j p_{j'}$ is a phase polynomial of degree $<m-1$, we conclude that the map $j\mapsto p_j\cdot P_{<m-1}(G,X,S^1)$ from $J$ to the quotient $P_{<m}(G,X,S^1)/P_{<m-1}(G,X,S^1)$ is a homomorphism. Write $G=\bigoplus_{p\in \mathcal{P}}\mathbb{Z}/p\mathbb{Z}$. For every prime $q\in \mathcal{P}$, we denote by $G_q$ be the $q$-component of $G$ (i.e. $G_q = \{g\in G :  qg=0\}$). By Lemma \ref{PPC1} we know that for every $g\in G_q$ and for all $j\in J$ we have that $p_j(g,\cdot)^{q}=p_{j^{q}}(g,\cdot)$  is phase polynomial cocycle of degree $<m-1$. Inductively (see the proof of Corollary \ref{ker:cor}), we have that $p_{j^{q^m}}(g,\cdot)=1$. Let $J_q:=J^{q^m}$ and $J'=\bigcap_{q\in\mathcal{P}} J_q$. The quotient $J/J'$ is a totally disconnected group of exponent $m$ and we have that $\Delta_j \rho \in B^1(G,X,S^1)$ for every $j\in J'$. Therefore, by Lemma \ref{cob:lem}, $\rho$ is $(G,X,S^1)$-cohomologous to a cocycle $\rho'$ that is invariant with respect to the action of some open subgroup $J''\leq J'$. To complete the proof we notice that since $J''\leq J'$ is an open subgroup and $J/J'$ is a totally disconnected group of exponent $m$ then the groups $U/J'$ and $U/J''$ are finite dimensional of exponent $m$. Thus, if $H=S^1$ we can take $V=J''$ and the proof is complete. Otherwise suppose that $H=C_{p^n}$. By embedding $H$ in $S^1$ and arguing as before we can find a measurable map $F:X\rightarrow S^1$ such that 	
	$$\Delta_j \rho = \Delta \Delta_j F$$ for all $j\in J''$. Our goal is to replace $F$ with a function which takes values in $C_{p^n}$. Since $\rho^{p^n}=1$ the equation above implies that $\Delta_j F^{p^n}$ is a constant. Since $j\mapsto \Delta_j F^{p^n}$ is a cocycle we conclude that there exists a character $\chi:J''\rightarrow S^1$ such that $\Delta_j F^{p^n}=\chi(j)$ for every $j\in J''$. Let $V:=\ker (\chi)$, since $J''$ is totally disconnected Corollary \ref{chartdg} implies that the image of $\chi$ is discrete. In particular, it follows that $V$ is an open subgroup of $J''$ and therefore $U/V$ is a finite dimensional group of exponent $m$. Finally as $F^{p^n}$ is invariant under $V$ we can find a measurable map $\tilde{F}:X\rightarrow S^1$ that is invariant under $V$ with $\tilde{F}^{p^n}=F^{p^n}$. Since $\tilde{F}$ is invariant to $V$ we conclude that $$\Delta_j\rho = \Delta_j \Delta F/\tilde{F}$$ for every $j\in V$. Moreover, since $F/\tilde{F}^{p^n}=1$, $\rho/\Delta(F/\tilde{F})$ is $(G,X,C_{p^n})$-cohomologous to $\rho$ and is invariant under $V$, as required. 
\end{proof}
From the lemma above we conclude the following result.
\begin{prop} \label{step1}
	Let $X$ be an ergodic $G$-system and let $H$ be a compact abelian group which acts on $X$ by automorphisms. Let $m,l\geq 0$ be integers, let $U$ be a finite dimensional group exponent $m$ and let $\rho:G\times X\rightarrow U$ be a cocycle of type $<l$. Then, there exists a subgroup $\tilde{H}\leq H$ such that for every $h\in\tilde{H}$, $\Delta_h \rho\in B^1(G,X,U)$ and $H/\tilde{H}$ is finite dimensional of exponent $<m\cdot l$.
\end{prop}
\begin{proof}
Let $U$ be as in the proposition and find a closed totally disconnected subgroup $\Delta\leq U$ of exponent $m$ such that $U/\Delta$ is a Lie group. By Theorem \ref{structureLieGroups} we can write $U/\Delta = (S^1)^n\times \prod_{i=1}^{r} C_{p_i^{a_i}}$ where $r\in \mathbb{N}$, $p_1,...,p_r$ are primes and $a_1,...,a_r\in\mathbb{N}$. Let $\chi_1,...,\chi_n,\tau_1,...,\tau_r$ be the coordinate maps and lift each of them to a character of $U$.\\
By Lemma \ref{step 1}, we can find a subgroup $H'$ of $H$ such that $H/H'$ is of exponent $l$ and $\Delta_h \chi\circ \rho$ is a coboundary for all $h\in H'$ and $\chi\in \left<\chi_1,...,\chi_n,\tau_1,...,\tau_r\right>$. \\
	The group $\Delta$ is of exponent $m$ and so we can write it as $\prod_{i\in I} C_{p_i^{b_i}}$, where $b_i\leq m$ (see Definition \ref{exponent}). Let $\{\pi_i:i\in I\}$ denote the coordinates of $\Delta$ and lift each of them arbitrarily to a character of $\widehat U$. Then the countable set $\chi_1,...,\chi_n,\tau_1,...,\tau_r,\pi_1,\pi_2,...$ generates $\widehat U$.\\
	
	Fix a coordinate $\pi$ of $\Delta$ of order $p^b$. Then $\pi^{p^b}$ is invariant under $\Delta$ and therefore is a character of $U/\Delta$. In particular, $\Delta_h \pi\circ\rho^{p^b} $ is a coboundary for all $h\in H'$. We conclude by Lemma \ref{Clroot:lem} that $\Delta_{h^{p^b\cdot l}}\rho\in B^1(G,X,S^1)$ for every $h\in H'$. For every $i\in I$, let $H_{p_i}=H^{p_i^{b_i}\cdot l}$ and $\tilde{H} := \bigcap_{i\in I} H_{p_i}$. It follows that the quotient $H/\tilde{H}$ is a totally disconnected group of exponent $m\cdot l$. Since $\Delta_h\chi\circ\rho$ is a coboundary for every $h\in \tilde{H}$ and $\chi\in\widehat U$, the claim follows by Theorem \ref{MScob}.
\end{proof}
We also need the following technical group theoretical lemma.
\begin{lem} \label{abl:lem}
	Let $H$ and $K$ be compact abelian groups and suppose that $K$ is finite dimensional of exponent $m$ for some $m\in\mathbb{N}$. We give $\text{Hom}(H,K)$ the topology of uniform convergence. Then, for any continuous homomorphism $\varphi:H\rightarrow \text{Hom} (H,K)$ we have that the group $H/\ker \varphi$ admits a totally disconnected open subgroup of exponent $m$. Moreover, if the group $K$ is totally disconnected of exponent $m$, then $H/\ker \varphi$ is totally disconnected of exponent $m$.
\end{lem}
\begin{proof}
	Let $\Delta$ be a totally disconnected subgroup of $K$ of exponent $m$ such that $K/\Delta$ is a Lie group. Let $\tilde{\varphi}:H\rightarrow \text{Hom}(H,K/\Delta)$ be the composition of $\tilde{\varphi}$ with the projection $\text{Hom}(H,K)\rightarrow \text{Hom}(H,K/\Delta)$. Since $K/\Delta$ is embedded in a finite dimensional torus we conclude that $\text{Hom}(H,K/\Delta)$ is discrete. It follows that $\ker \tilde{\varphi}$ is an open subgroup of $H$. We denote by $\tilde{H}$ the kernel of $\tilde{\varphi}$ and by $\varphi'$ the restriction of $\varphi$ to $\tilde{H}$. Then the map $\varphi':\tilde{H}\rightarrow \text{Hom}(H,K)$ takes values in $\text{Hom}(H,\Delta)$. We prove that $\tilde{H}/\ker\varphi'$ is totally disconnected of exponent $m$. First recall that $\Delta$ can written as $\Delta = \prod_{i\in I} C_{p_i^{b_i}}$ for some $b_i\leq m$ and a countable set of indices $I$. For each prime $p$ let $\pi_p:\text{Hom}(H,\Delta)\rightarrow \text{Hom}(H,\Delta_p)$ be the projection map where $\Delta_p$ is the $p$-sylow subgroup of $\Delta$. Clearly, the group $H_{p}:=\tilde{H}^{p^m}$ is in the kernel of $\pi_p\circ\varphi$. Let $H'$ be the intersection of all $H_p$ over all primes. We see that $H'\leq \ker \varphi'$ and $\tilde{H}/H'$ is a totally disconnected group of exponent $m$ from which we conclude that $\tilde{H}/\ker \varphi'$ is totally disconnected of exponent $m$. Since $\tilde{H}$ is open in $H$, we have that $\tilde{H}/\ker\varphi'$ is open in $H/\ker\varphi$. In the case where $K$ is totally disconnected we get that $\tilde{H}=H$ and so the claim follows from the same argument.
\end{proof}
We can now prove Theorem \ref{InvFD:thm}.
\begin{proof}
	Let $X$ be a $G$-system of order $<k$. and write $X=Z_{<k-1}(X)\times_\rho U$. Let $U_n$ be a sequence of Lie groups such that $U=\underset{\longleftarrow}{\lim}  U_n$.
	
 The idea is to construct a sequence of $k$ factors of $Z_{<k-1}(X)\times_{\rho_n} U_n$ each time replacing one of the structure groups of $Z_{<k-1}(X)$ with a finite dimensional group. More concretely, we construct systems $X_{l,n}$ recursively as follows: let $X_{k,n}:=X_n$. Fix $l\leq k$ and suppose inductively that we have already constructed $$X_{l,n} = U_0\times_{\rho_{1}}U_1\times...\times_{\rho_{l-1}} U_{l-1}\times_{\rho_{l,n}}U_{l,n}\times...\times_{\rho_{{k-1},n}}U_{k-1,n}$$
	where $U_{l,n},...,U_{k-1,n}$ are finite dimensional quotients of $U_l,...,U_{k-1}$ respectively of exponent $m=O_k(1)$ and that $X_{l,n}$ is a factor of $X_{l,n+1}$ for every $n\in\mathbb{N}$. To construct $X_{l-1,n}$, we use a similar argument as in Lemma \ref{up} to lift a subgroup of the vertical rotations by $U_{l-1,n}$ to automorphisms of $X_{l,n}$. We argue as follows: the group $U_{l-1}$ acts on $Z_{<l}(X)$ by automorphisms. Therefore, by Proposition \ref{step1} we can find a subgroup $\tilde{U}_{l-1,n}$ such that $\Delta_u \rho_{l,n}\in B^1(G,Z_{<l}(X),U_{l,n})$ for every $u\in \tilde{U}_{l-1,n}$. We consider the group $$\mathcal{H}_{l-1,n}:=\{S_{u,F}:u\in U_{l-1,n}, F\in \mathcal{M}(Z_{<l}(X),U_{l,n}), \Delta_u \rho_{l,n}=\Delta F\}.$$ Since $\tilde{U}_{l-1,n}$ is abelian, $\mathcal{H}_{l-1,n}$ is a $2$-step nilpotent group. Let $p:\mathcal{H}_{l-1,n}\rightarrow U_{l-1,n}$ be the projection map. The kernel of $p$ consists of transformations of the form $S_{1,c}$ where $c$ is a constant in $U_{l,n}$. We can therefore identify $\ker p$ with the compact group $U_{l,n}$. By Theorem \ref{open} and Lemma \ref{compact} we conclude that $\mathcal{H}_{l-1}$ is compact and $\mathcal{H}_{l-1,n}/U_{l,n}\cong \tilde{U}_{l-1,n}$. This implies that the commutator map on $\mathcal{H}_{l-1,n}$ induces a bilinear map  $b:\tilde{U}_{l-1,n}\times \tilde{U}_{l-1,n}\rightarrow U_{l,n}$. Using Pontryagin duality, we see that $b$ can be identified with a continuous homomorphism $\tilde{U}_{l-1,n}\rightarrow \text{Hom} (\tilde{U}_{l-1,n},U_{l,n})$. Since $U_{l,n}$ is finite dimensional of exponent $m=O_k(1)$ we conclude from the previous lemma that the kernel is a subgroup $U_{l-1,n}'\leq \tilde{U}_{l-1,n}$ such that the quotient $\tilde{U}_{l-1,n}/U_{l-1,n}'$ admits a totally disconnected group of exponent $m'=O_k(1)$ as an open subgroup. By shrinking $U_{l-1,n}'$ if necessary we can assume that $\tilde{U}_{l-1,n}/U_{l-1,n}'$ is increasing in $n$. The pre-image of $U_{l-1,n}'$ under the projection $p$ is a compact abelian finite dimensional group and it acts by automorphisms on $Z_{<l+1}(X_{l,n})$. We repeat this process inductively another $k-l$ steps each time lifting a subgroup of $U_{l-1,n}$ to a group of automorphisms of the next universal characteristic factor of $X_{l,n}$. At the end of the day we obtain a compact abelian finite dimensional group of exponent $m''=O_k(1)$, $\tilde{H}_{l-1,n}$ which acts by automorphisms on $X_{l,n}$. Let $\tilde{p}:\tilde{H}_{l-1,n}\rightarrow U_{l-1,n}$ be the projection map, we see that  $U_{l-1,n}/\tilde{p}(\tilde{H}_{l-1,n})$ is a finite dimensional group of exponent $m''=O_k(1)$. Now in the last step we can use the fact that $U_{k-1,n}$ is a Lie group. In that step we invoke Lemma \ref{step 1} instead of Proposition \ref{step1}. We conclude that $\rho_{k-1,n}$ is $(G,X,U_{l,n})$-cohomologous to a cocycle that is invariant under the action of $\tilde{\mathcal{H}}_{l-1,n}$. Let $X_{l-1,n} = X_{l,n}/\tilde{\mathcal{H}}_{l-1,n}$ (as in Lemma \ref{down}). Clearly, $X_{l-1,n}$ is a factor of $X_{l,n}$ and the $(l-1)$-th structure group of $X_{l-1,n}$ is finite dimensional of exponent $m''=O_k(1)$. Since in every step in the proof we extended the structure groups of $X_{l-1,n}$ to exceeds those of $X_{l-1,n-1}$ we have that $X_{l-1,n}$ is increasing in $n$. This completes the proof.
\end{proof}
\subsection{Proof of Theorem \ref{Main2:thm} for systems of order $<3$} \label{case3}
 The proof of Theorem \ref{Main2:thm} for systems of order $<3$ is significantly easier than the general case. 

\begin{thm}\label{order<3}
	Let $G=\bigoplus_{p\in P}\mathbb{Z}/p\mathbb{Z}$, and $X$ be an ergodic $G$-system of order $<3$. Then $X$ is an inverse limit of finite dimensional $2$-step nilpotent systems.
\end{thm}
A version of this theorem without the finite dimensional result was given in \cite{OS}. For the sake of completeness we repeat the proof here. Recall that a system of order $<3$ takes the form $X=Z\times_\rho U$ where $U$ and $Z$ are compact abelian groups and $Z$ is the Kronecker factor. By Theorem \ref{InvFD:thm} we may assume that $Z$ is finite dimensional and $U$ is a Lie group.  

\begin{defn} [Host and Kra group, for systems of order $<3$] Let $X$ be a system of order $<3$. Let $s\in Z$ and $F:Z\rightarrow U$ be a measurable map. We denote by $S_{s,F}$ the measure preserving transformation $(z,u)\mapsto (sz,F(z)u)$, and let $\mathcal{G}(X)$ denote the group of all such transformations with the property that there exists $c_s:G\rightarrow U$ such that $\Delta_s \rho = c_s \cdot \Delta F_s$.
\end{defn}
Host and Kra \cite{HK} proved that this definition of $\mathcal{G}(X)$ coincides with Definition \ref{HKgroup:def} for systems of order $<3$. In particular, this means that $\mG(X)$ is a $2$-step nilpotent Polish group.\\
Observe that the kernel of the projection $p:\mathcal{G}(X)\rightarrow Z$ can be identified with $P_{<2}(Z,U)$. We claim that in order to prove Theorem \ref{order<3} it is enough to show that $p$ is onto. Indeed, in that case, $p$ is an open map (Theorem \ref{open}) and the group $\mathcal{G}(X)$ acts transitively on $X$. Moreover, if $Z$ is finite dimensional and $U$ is a Lie group then $P_{<2}(Z,U)$ is finite dimensional. Since the projection $p:\mG(X)\rightarrow Z$ is onto this implies that $\mG(X)$ is also a finite dimensional group. At this point we would want to use Theorem \ref{Effros}, but unfortunately we only have a near-action of $\mG(X)$ on $X$. Yet, the identification $X\cong \mG(X)/\Gamma$ for $\Gamma=\{S_{1,\chi}:\chi\in\text{Hom}(Z,U)\}$ was obtained by Meiri in \cite[Theorem 3.21]{Meiri}. This completes the proof of the claim. We note that for the higher order case we will use a different argument (see section \ref{identification:sec}).
\begin{proof}[Proof of Theorem \ref{order<3}]
	 The projection $p:\mG(X)\rightarrow Z$ is onto if and only if for every $s\in Z$ there exist a measurable map $F_s:Z\rightarrow U$ and a homomorphism $c_s:G\rightarrow U$ such that $\Delta_s \rho = c_s\cdot \Delta F_s$. Since $U$ is a Lie group, by studying each coordinate separately it is enough to show that the equation holds in the case where $U$ is a torus or equals to $C_{p^n}$ for some prime $p$ and $n\in\mathbb{N}$. By Lemma \ref{dif:lem} the cocycle $\Delta_s\rho$ is of type $<1$, therefore if $U$ is a torus the equation follows by Lemma \ref{type0}. Otherwise ,assume that $U=C_{p^n}$. By embedding $C_{p^n}$ in $S^1$ and applying Lemma \ref{type0}, we see that for every $s\in Z$,
	 \begin{equation} \label{cl}\Delta_s \rho = c_s\cdot \Delta F_s
	\end{equation} for some constant $c_s:G\rightarrow S^1$ and $F_s:Z\rightarrow S^1$.\\
	 Our goal is to replace $F_s$ and $c_s$ with some $F_s'$ and $c_s'$ such that equation (\ref{cl}) holds and $F_s',c_s'$ takes values in $C_{p^n}$.\\
	
	As a first step, we show that $\rho$ is $(G,Z,S^1)$-cohomologous to a phase polynomial of degree $<2$. Observe that by the cocycle identity we have, $$\Delta_{s^{p^n}}\rho = \Delta_s\rho^{p^n} \cdot \prod_{k=0}^{p^n-1} \Delta_s \Delta_{s^k} \rho$$
	From equation (\ref{cl}) we see that $\prod_{k=0}^{p^n-1} \Delta_s \Delta_{s^k} \rho$ is a coboundary. Since $\rho$ takes values in $C_{p^n}$, the term $\Delta_s\rho^{p^n}$ vanishes and we conclude that $\Delta_{s^{p^n}} \rho$ is a coboundary for every $s\in Z$. Let $Z_0$ be the connected component of the identity in $Z$. Since connected groups are divisible (Lemma \ref{divisible}), we conclude that $\Delta_s\rho$ is a $(G,Z,S^1)$-coboundary for every $s\in Z_0$. By Lemma \ref{cob:lem}, $\rho$ is $(G,Z,S^1)$-cohomologous to a cocycle $\rho'$ that is invariant with respect to the action of $Z_0$. Let $\pi_\star \rho'$ be the push-forward of $\rho'$ to $Z/Z_0$. By Lemma \ref{Cdec:lem}, $\pi_\star \rho'$ is of type $<2$. Therefore, by Theorem \ref{Main:thm} it is cohomologous to a phase polynomial of degree $<2$.\footnote{This result requires a slightly stronger version of Theorem \ref{Main:thm}. If $l<\min_{p\in\mathcal{P}} p$ then one can replace the quantity $O_{k,m,l}(1)$ in Theorem \ref{Main:thm} with $l$. A proof for this can be found in \cite{OS}.} Lifting everything back to $Z$ we conclude that $\rho'$ and $\rho$ are $(G,Z,S^1)$-cohomologous to a phase polynomial $Q:G\times Z\rightarrow S^1$ of degree $<2$. Moreover, $Q$ is invariant to translations by $Z_0$. We write \begin{equation} \label{Poly}\rho = Q\cdot \Delta F\end{equation} for some $F:Z\rightarrow S^1$.\\
	Since $\rho$ takes values in $C_{p^n}$, we have that \begin{equation} \label{Cl2}1=Q^{p^n}\cdot\Delta F^{p^n}
	\end{equation}
	By taking the derivative of both sides of the equation above by $s\in Z$, we conclude that $\Delta_s F^{p^n}$ is a phase polynomial of degree $<2$. Our next goal is to replace $F$ with a function $F'$ such that $F'/F$ is a phase polynomial of degree $<3$ and at the same time $\Delta_s F'^{p^n}$ is a constant.\\
	We study the phase polynomial $Q$. It is a fact that every phase polynomial of degree $<2$ is a constant multiple of a homomorphism. Therefore, we can write $Q(g,x)=c(g)\cdot q(g,x)$ where $c:G\rightarrow S^1$ is a constant and $q:G\times Z\rightarrow S^1$ is a homomorphism in the $Z$-coordinate. Since $Q$ is a cocycle $$c(g+g')q(g+g',x)=c(g)c(g')\Delta_{g'}q(g,x)\cdot q(g,x)\cdot q(g',x).$$ It follows that $q$ is bilinear in $g$ and $x$. Let $$Z'_{p} = \ker(q^{p^n}) = \{s\in Z : q(g,s)^{p^n}=1 \text{ for every } g\in G\}.$$ 
	Then $Z/Z'_p$ is isomorphic to a subgroup of $\widehat G^{p^n}=\prod_{p\not= q\in \mathcal{P}} C_q$. By taking the derivative of both sides of equation (\ref{Cl2}) by $s\in Z_p'$, we conclude by the ergodicity of the Kronecker factor that $\Delta_s F^{p^n}$ is a constant. Recall that $F^{p^n}$ is a phase polynomial. Therefore, by Corollary \ref{ker:cor} and the above we see that there exists an open subgroup $Z'\leq Z$, which contains $Z_p'$, such that $\Delta_s F^{p^n}$ is a constant for every $s\in Z'$. By the cocycle identity we conclude that $\Delta_s F^{p^n}=\chi(s)$ for some character $\chi:Z'\rightarrow S^1$. Lift $\chi$ to a character of $Z$ arbitrarily. We conclude that $F^{p^n}/\chi$ is a phase polynomial which is invariant under translations by $Z'$. Since $Z'$ is open the quotient $Z/Z'$ is a finite group. Moreover, since $Z'$ contains $Z_p'$ we conclude that the order of $Z/Z'$ is co-prime to $p$. Let $m=|Z/Z'|$. Since $F^{p^n}/\chi$ is of degree $<3$, we conclude that $\Delta F^{p^n}/\chi$ is a constant multiple of a homomorphism from $Z$ to $S^1$. This homomorphism is invariant to $Z'$ and therefore $\Delta (F^{p^n}/\chi)^m$ is a constant. This implies that $(F^{p^n}/\chi)^m$ is a polynomial of degree $<2$ and by the same argument we conclude that $(F^{p^n}/\chi)^{m^2}$ is a constant. Let $c\in S^1$ be an $m$-th root of that constant, we conclude that $F^{p^n}/c\cdot \chi$ take values in $C_{m^2}$. Since $p$ and $m$ are co-prime we can find an integer $l$ such that $l\cdot p^n = 1 \mod m^2$. We conclude that $R:= (F^{p^n}/c\cdot \chi)^l$ is a phase polynomial of degree $<3$ and that $R^{p^n} = F^{p^n}/c\cdot\chi$. Let $Q':=Q\cdot \Delta R$ and $F':=F/R$. Then, as in equation (\ref{Poly}), we have $$\rho = Q'\cdot \Delta F'$$ and $\Delta_s F'^{p^n} = c\cdot \chi(s)$.\\
	
	Now, by taking the derivative by $s\in Z$ on both sides of the equation above, we conclude that $$\Delta_s \rho = \Delta_s Q' \cdot \Delta \Delta_s F'.$$ Observe that $c_s':=\Delta_s Q'$ is a character of $G$ and $$c_s'^{p^n} = \Delta_s Q^{p^n} \cdot \Delta \Delta_s F^{p^n}/c\cdot \chi = 1$$ where the last equality follows from (\ref{Cl2}) and the fact that $\Delta \Delta_s c\cdot \chi$ vanishes. It is left to change the term $\Delta_s F$. Set $F_s' := \Delta_s F'/\phi(s)$ where $\phi(s)$ is a $p^n$-th root of $\chi(s)$ in $S^1$. Then, as before we have that $$\Delta_s \rho = c_s'\cdot \Delta F_s'$$ but this time $c_s'^{p^n} = F_s'^{p^n}=1$. This implies that $S_{s,F_s'}\in \mG(X)$ and therefore $p$ is onto.
\end{proof}
We note that this argument fails for systems of higher order. The main reason for this is that Theorem \ref{InvFD:thm} only allows to approximate the last structure group by Lie groups. In particular, we do not know how to prove a counterpart of this result in the case where $U$ is not a Lie group without the use of $k$-extensions.
\section{Extension trick} \label{ext:sec}
Let $X$ be an ergodic $G$-system. Under certain conditions we show that there exist an extension $\pi:Y\rightarrow X$ with the following property: for any prime $p$, a natural number $n\in\mathbb{N}$ and a phase polynomial $P:X\rightarrow S^1$, there exists a phase polynomial $Q:Y\rightarrow S^1$ such that $Q^{p^n}=P\circ\pi$.\\
We begin with an example which illustrates the idea. 
\begin{example} \label{example2}
Let $X=S^1$ and  $G=\bigoplus_{p\in P}\mathbb{Z}/p\mathbb{Z}$ where $P$ is an infinite subset of odd prime numbers. We define a homomorphism $\varphi:G\rightarrow X$ by $$\varphi((g_p)_{p\in P}) = \prod_{p\in P} w_p^{g_p}$$ where $\omega_p$ is the first $p$-th root of unity. The action of $G$ on $X$ by $T_g x = \varphi(g)x$ is ergodic, and the identity map $\chi:X\rightarrow S^1$ is a phase polynomial of degree $<2$. Our goal is to construct an extension $\pi:Y\rightarrow X$ and a phase polynomial $Q:Y\rightarrow S^1$ of degree $<2$ such that $Q^2 = \chi\circ \pi$.\\ We let $c(g):=\Delta_g \chi$ and observe that since $2\not\in P$, there exists some $d\in\widehat G$ such that $d^2=c$. Fix any measurable map $F:X\rightarrow S^1$ with $F^2=\chi$ and let $\tau=d\cdot \Delta \overline{F}$. The cocycle $\tau:G\times X\rightarrow C_2$ defines an extension  $Y:=X\times_{\tau} C_2$. On $Y$ we have that $d$ is an eigenvalue of the eigenfunction $Q(x,u):=u\cdot F(x)$. Moreover, $Q^2 = \chi\circ \pi$.\\ Since $X$ is ergodic and $\tau$ is minimal (see Lemma \ref{minimal}), we conclude that $Y$ is ergodic. Moreover $\tau$ is of type $<1$ and therefore $Y$ is of order $<2$. In fact, it is easy to see that $(x,u)\mapsto u \cdot \overline F(x)$ defines an isomorphism of $Y$ and $(S^1,G)$ where the action of $G$ on $S^1$ is given by $T_gy = d(g)y$. 
\end{example}
Let $G=\bigoplus_{p\in P}\mathbb{Z}/p\mathbb{Z}$ as usual. We state a simple technical proposition about $G$-cocycles.\footnote{The proposition below is the cocycle-counterpart of the fact that any homomorphism is uniquely determined by the values it gives to a generating set.}
\begin{prop}[Conditions for cocycle] \label{cocycle:prop}
	Let $X$ be a $G$-system and $f:G\times X\rightarrow S^1$ be a function such that 
	\begin{equation} \label{commute0}
	\Delta_h f(g,x) = \Delta_g f(h,x)
	\end{equation}
	for all $h,g\in G$. Let $E=\{e_1,e_2,...\}$ denote the natural basis of $G$ and suppose that for every $g\in E$ we have,
	\begin{equation} \label{linecocycle0}
	\prod_{k=0}^{\text{order}(g)-1} T_g^k f(g,x)=1
	\end{equation}
	Then the function $\tilde{f}:G\times X\rightarrow S^1$ below is a cocycle which agrees with $f$ for every $g\in E$. 
	\begin{equation} \label{cocycle:eq0}
\tilde{f}(g,x):= \prod_{i=1}^\infty T_{g_1}...T_{g_{i-1}} \prod_{k=0}^{g_i-1} T_{ke_i} f(e_i,x)
	\end{equation}
	where $g=(\overline{g}_1,\overline{g}_2,...)$, the constants $g_1,g_2,...$ are any representatives of $\overline{g}_1,\overline{g}_2,...$ in $\mathbb{N}$ respectively and the product $\prod_{k=0}^{-1}(\text{anything})$ is assumed to be $1$.
\end{prop}
\textbf{Convention:} We refer to an element $g\in E$ as a generator.\\

We need the following results about group extensions.
\begin{defn} [Image and minimal cocycles]
	Let $X$ be a $G$-system and $\rho:G\times X\rightarrow U$ be a cocycle into a compact abelian group $U$. The subgroup $U_\rho\leq U$ generated by $\{\rho(g,x):g\in G,x\in X\}$ is called the \textit{image} of $\rho$. We say that $\rho$ is \textit{minimal} if it is not $(G,X,U)$-cohomologous to a cocycle $\sigma$ with $U_\sigma\lneqq U_\rho$.
\end{defn}
\begin{lem} \label{minimal}\cite[Corollary 3.8]{Zim}
	Let $X$ be an ergodic $G$-system. Let $\rho:G\times X\rightarrow U$ be a cocycle into a compact abelian group. Then,
	\begin{itemize}
		\item{There exists a minimal cocycle $\sigma:G\times X\rightarrow U$ such that $\rho$ is $(G,X,U)$-cohomologous to $\sigma$.}
		\item{$X\times_\rho U$ is ergodic if and only if $X$ is ergodic and $\rho$ is minimal with image $U_\rho = U$.}
	\end{itemize}
\end{lem}

We describe an obstacle. Suppose that $P\in P_{<2}(X,S^1)$, then $\Delta P$ can be identified with an element in $\widehat G$. If for some $g\in G$ of order $p$ we have that $\Delta_g P\not = 1$ then $\Delta P$ does not have a $p$-th root in $\widehat G$. In particular, it is impossible to find a $p$-th root for $P$ in $P_{<2}(X,S^1)$, even if one passes to an extension of the original system. We deal with this problem later using $k$-extensions (Definition \ref{kext}), but as for now we assume that there is no such obstacle.
\begin{thm} [Roots for phase polynomials in an extension]\label{extension:thm}
Let $X$ be an ergodic $G$-system. Fix $d\geq 1$ and suppose that $P_1,P_2,...$ are at most countably many $(X,S^1)$-phase polynomials of degree $<d$. Let $p_1,p_2,...$ be (not necessarily distinct) prime numbers and assume that for every $i\in\mathbb{N}$, $\Delta_g P_i = 1$ for all $g\in G$ of order $p_i$. Then, there exist a totally disconnected group $\Delta$ and a cocycle $\tau:G\times X\rightarrow \Delta$ of type $<d-1$ such that the extension $Y=X\times_\tau \Delta$ is ergodic and for every $n\in\mathbb{N}$ and $i=1,2,..$ there exist $(G,X,S^1)$-phase polynomials $Q_{i,n}:Y\rightarrow S^1$ of degree $<d$ such that $Q_{i,n}^{p_i^{n}}=P_i\circ \pi$ where $\pi:Y\rightarrow X$ is the factor map.
\end{thm}
\begin{proof}
Let $p$ be a prime number. Let $P:X\rightarrow S^1$ be a polynomial of degree $<d$ and assume that $P$ is $T_g$-invariant for every $g\in G$ of order $p$. Let $c(g,x):=\Delta_g P(x)$ and observe that by Proposition \ref{PPC} the phase polynomial $c(g,x)$ takes values in $C_m$ for some $m=\text{order}(g)^{d-1}$. Let $n\in\mathbb{N}$ be a natural number, fix $g\in G$ of order coprime to $p$ and let $m=\text{order}(g)^{d-1}$. Since $p^n$ and $m$ are co-prime we can find a natural number $l_g(n)$ such that $p^n\cdot l_g(n) = 1$ modulo $m$. It follows that the phase polynomial $d_{n}(g,x):=c(g,x)^{l_g(n)}$ is a $p^n$-th root of $c(g,x)$. We extend $d_n$ to $G$ by decomposing every $g\in G$ as $g=g_p+g'$ where $g_p$ is of order $p$ and $g'$ of order co-prime to $p$, and setting $d_n(g,x):=d_n(g',x)$. Our next goal is to replace $d_n$ with a cocycle using Proposition \ref{cocycle:prop}. Observe first that since $d_n(g,\cdot)$ is a power of $c(g,\cdot)$, we have that \begin{equation} \label{linecocycle}
	\prod_{i=0}^{\text{order}(g)-1} d_n(g,T_{g^i}x) = 1.
	\end{equation} Now we claim that for every $g,h\in G$,  
	\begin{equation}\label{commute}
	\frac{\Delta_h d_n(g,x)}{\Delta_g d_n(h,x)}=1.
	\end{equation}
	On one hand, $d_n^{p^n}=c$ and therefore this quotient is of order $p^n$. On the other hand, $d_n(g,x)$ and $d_n(h,x)$ are of order co-prime to $p$, hence the quotient is trivial. Therefore by Proposition \ref{cocycle:prop} there exists a cocycle $\tilde{d}_n:G\times X\rightarrow S^1$ which agrees with $d_n$ on a generating set. Since ${d_n}^{p^n}=c$ and $c$ is a cocycle we conclude that $\tilde{d}_n^{p^n}=c$.\\
	Now, we apply the argument above for each of the polynomials in the theorem. Set $c_i(g,x):=\Delta_g P_i$. We conclude that for every $i,n\in\mathbb{N}$ there exists a phase polynomial cocycle $\tilde{d}_{i,n}$ of degree $<d-1$ such that $\tilde{d}_{i,n}^{p_i^n} = c_i.$ For each $i,n\in\mathbb{N}$ fix a measurable map\footnote{One way to do so is by identifying $S^1$ with $\mathbb{R}/\mathbb{Z}$ and setting $F_{i,n}(x) = \frac{\{P_i(x)\}}{n}$, where $\{\cdot\}$ is the fractional part.} $F_{i,n}:X\rightarrow S^1$ such that $F_{i,n}^{p_i^n}=P_i$ and let $\tau:=(\tilde{d}_{i,n}\cdot \Delta F_{i,n})_{i,n\in\mathbb{N}}$ be a cocycle, $\tau: G\times X \to \prod_{i,n\in\mathbb{N}}C_{p_i^n}$.\\
	
The extension of $X$ by $\tau$ may not be ergodic, so we choose a minimal cocycle $\tau'$ that is $(G,X,\prod_{i,n\in\mathbb{N}}C_{p_i^n})$-cohomologous to $\tau$ (Lemma \ref{minimal}) and write $\tau' = \tau\cdot \Delta F$ for some measurable map $F:X\rightarrow\prod_{i,n\in\mathbb{N}}C_{p_i^n}$. We denote the image of $\tau'$ by $\Delta$ and consider the extension $Y:=X\times_{\tau'}\Delta$. The closed subgroup $\Delta \leq \prod_{i,n\in\mathbb{N}}C_{p_i^n}$ is totally disconnected. Moreover, it follows from the construction that the system $Y$ is ergodic. Finally, since $\tau$ is of type $<d-1$, we conclude that so is $\tau'$. For $i,n\in\mathbb{N}$ let $\pi_{i,n}:\prod_{i,n\in\mathbb{N}}C_{p_i^n}\rightarrow C_{p_i^n}$ be the $(i,n)$-th coordinate map. We conclude that the function $$Q_{i,n}(x,u) := \pi_{i,n}|_{\Delta'}(u)\cdot \overline{F}_{i,n}(x)\cdot \pi_{i,n}\circ \overline{F}(x)$$ is a phase polynomial in $Y$ (whose derivative is $\tilde{d}_{i,n}$) and it satisfies that $Q_{i,n}^{p_i^n} = P_i \circ \pi$, as required.\\
\end{proof}
\begin{rem} \label{inv} Following the same argument as above we have the following generalizations:
	\begin{itemize}
	\item {
If $P$ is a phase polynomial and $\Delta_g P=1$ for every $g\in G$ of order $p$ and of order $q$ then we can adapt the proof and find an $p^nq^m$ root for all $n,m\in\mathbb{N}$. The same goes for multiple primes. }
\item {If instead of $G=\bigoplus_{p\in P}\mathbb{Z}/p\mathbb{Z}$ we take an extension $G^{(l)}= \bigoplus_{p\in P} \mathbb{Z}/p^l\mathbb{Z}$ then the same proof still holds.}
\end{itemize}
Moreover, observe that if $P$ is $T_g$-invariant for some $g\in G$, then by choosing $d_n$ in the proof above with $d_n(g,\cdot)=1$ for these $g$'s, we can construct an $n$-th root of $P$ which is also $T_g$-invariant.
\end{rem}
Now, we want to remove the hypothesis that $\Delta_g P_i = 1$ for every $g\in G$ of some order $p_i$. To do this we use $m$-extensions (see Definition \ref{kext}). We begin with the following definitions.
\begin{defn}[Multi-cocycles]
    Let $m\geq 1$. Let $X$ be an ergodic $G$-system and $U$ a compact abelian group. We say that a function $q:G^m\times X\rightarrow U$ is a multi-cocycle if it is a cocycle in each coordinate. Namely, for every $1\leq i \leq m$, every $g_1,g_2,...,g_m\in G$ and $g_i'\in G$ we have,
    $$q(g_1,...,g_{i-1},g_i\cdot g_i',g_{i+1},...,g_m,x) = q(g_1,...,g_i,...,g_m,x)\cdot q(g_1,...,g_i',...,g_m,T_{g_i}x).$$
    We say that $q$ is symmetric if it is invariant under permutations of coordinates of $G^m$ and we denote by $\text{SMC}_m(G,X,U)$ the group of symmetric multi-cocycles $q:G^m\times X\rightarrow U$.\\
    If the multi-cocycle $q$ is a constant in $x$ then we say that $q$ is multi-linear and denote by $\text{SML}_m(G,U)$ the group of symmetric multi-linear maps $\lambda:G^m\rightarrow U$. 
\end{defn}
We say that a multi-cocycle $q:G^m\times X\rightarrow U$ is a phase polynomial of degree $<r$ if for every $g_1,...,g_m\in G$ the map $x\mapsto q(g_1,...,g_m,x)$ is a phase polynomial of degree $<r$. We have the following result.
\begin{lem} \label{multicocycle}
    Let $X$ be an ergodic $G$ system and let $m,r\in\mathbb{N}$. Let $q\in \text{SMC}_m(G,X,S^1)$ be a phase polynomial of degree $<r$. Then, there exists an $O_{m,r}(1)$-extension $Y$ with factor map $\pi:Y\rightarrow X$ and a phase polynomial $Q$ of degree $<r+m$ such that $q(g_1,...,g_m,\pi(y))=\Delta_{g_1}...\Delta_{g_m} Q(y)$.
\end{lem}
\begin{proof}
	We prove the lemma by induction on $m$. For $m=1$, $q$ is a cocycle. Therefore, by Lemma \ref{minimal}, $q$ is cohomologous to a minimal cocycle $\sigma$. Let $V$ be the image of $\sigma$ and consider the extension $Y=X\times_\sigma V$. Arguing in Theorem \ref{extension:thm}, we see that $q$ is a coboundary in $Y$ and the claim follows. Now, let $m\geq 2$ and suppose that the claim holds for all smaller values of $m$. Let $q: G^m\times X\rightarrow S^1$, be a multi-cocycle. Then, for every $g_1,...,g_{m-1}$, the map $g\mapsto p(g,g_1,...,g_{m-1},x)$ is a cocycle. Therefore, as in the case $m=1$ we can find an extension $\tilde{\pi}:\tilde{X}\rightarrow X$ and phase polynomials $Q_{g_1,...,g_m}:\tilde{X}\rightarrow S^1$ such that $q(g,g_1,...,g_{m-1},\tilde{\pi}(x)) = \Delta_g Q_{g_1,...,g_{m-1}}(x)$ for every $g,g_1,...,g_{m-1}\in G$. By choosing the same $Q_{g_1,...,g_{m-1}}$ for any permutation of $g_1,...,g_{m-1}$, we can assume that $(g_1,...,g_{m-1})\mapsto Q_{g_1,...,g_{m-1}}$ is symmetric. It is left to show that we can choose $Q_{g_1,...,g_{m-1}}$ to be a cocycle in every coordinate. 
The fact that $q$ is symmetric implies that
	\begin{equation} \label{commute'}
	\Delta_h Q_{g_1,...,g_{m-1}} = \Delta_{g_i} Q_{g_1,...,g_{i-1},h,g_{i+1},...,g_{m-1}}
	\end{equation}
	for every $1\leq i \leq m-1$ and $h,g_1,...,g_{m-1}\in G$. Observe that by Proposition \ref{PPC} the order of the left hand side is some power of $\text{order}(h)$ and the order of the right hand side some power of $\text{order}(g_i)$. It follows that if one of the $g_i$'s is of order co-prime to $p$, then $Q_{g_1,...,g_{m-1}}$ is invariant with respect to the action of the subgroup $G_p=\{g\in G : pg=0\}$. In particular, if $g_i$ is coprime to $g_j$ then $Q_{g_1,...,g_{m-1}}$ is a constant. By changing the choice of $Q_{g_1,...,g_{m-1}}$ we can assume that $Q_{g_1,...,g_{m-1}} = \prod_{p\in P} Q_{g_1^{(p)},...,g_{m-1}^{(p)}}$ where $g_i^{(p)}$ is the $p$-component of $g_i$ (we note that the infinite product is well defined because all but finitely many $p$-components are trivial).\\
	Suppose now that all $g_1,...,g_{m-1}$ are of order $p$. Then, since $q$ is a multi-cocycle, for every $1\leq i \leq k$ we get that 
	\begin{equation} \label{linep}
	    \prod_{k=0}^{p-1} T_{g_i}^k Q_{g_1,...,g_{m-1}} = c_p(g_1,...,g_{m-1})
	\end{equation}
	where $c_p(g_1,...,g_{m-1})$ is symmetric and independent of $i$. Let $c_p'(g_1,...,g_{m-1})$ be a $p$-th root of $c_p(g_1,...,g_{m-1})$. By picking the same root for all permutations of $g_1,...,g_{m-1}$ we can assume that $c'_p$ is symmetric. Let $Q'_{g_1,...,g_{m-1}} = Q_{g_1,...,g_{m-1}}/c_p'(g_1,...,g_{m-1})$ whenever $g_1,...,g_{m-1}$ are of order $p$. Then, for every $1\leq i \leq m-1$
	\begin{equation} \label{linecocycle'}
	\prod_{k=0}^{p-1} T_{g_i}^k Q'_{g_1,...,g_{m-1}}=1.
	\end{equation}
	Now set $Q''_{g_1,...,g_{m-1}} = \prod_{p\in P} Q'_{g_1^{(p)},...,g_{m-1}^{(p)}}$. By Proposition \ref{cocycle:prop} (applied for all coordinates), there exists a symmetric multi-cocycle $Q''_{g_1,...,g_{m-1}}$ that agrees with $Q'_{g_1,...,g_{m-1}}$ whenever $g_1,...,g_{m-1}$ are elements in a basis of $G$.\\ Since $Q''_{g_1,...,g_{m-1}}$ is symmetric, it is a cocycle in every coordinate. In other words $Q''\in SML_{m-1}(G,X,S^1)$ is a phase polynomial of degree $<r+1$. Therefore, by the induction hypothesis we can find an $O_{m,r}(1)$-extension $\pi:Y\rightarrow \tilde{X}$ and a phase polynomial $Q:Y\rightarrow S^1$ of degree $<r+m$ such that $Q''_{g_1,...,g_{m-1}}(\pi(y)) = \Delta_{g_1}...\Delta_{g_{m-1}}Q(y)$. Observe moreover that since $q$ is a multi-cocycle, we have that $\Delta_g Q''_{g_1,...,g_{m-1}} = \Delta_g Q_{g_1,...,g_{m-1}}$ for every $g_1\in G$. Therefore, $Q''_{g_1,...,g_{m-1}}/Q_{g_1,...,g_{m-1}}$ is a constant and we conclude that for every $g\in G$ and every $g_1,g_2,...,g_{m-1}$ in a basis of $G$, we have that 
	$$q(g,g_1,...,g_{m-1},\tilde{\pi}(\pi(y))) = \Delta_g \Delta_{g_1}...\Delta_{g_{m-1}} Q(y)$$ since $q$ is a cocycle in each coordinate, the same holds if the generators $g_1,g_2,...,g_{m-1}$ are replaced with any elements of $G$. This completes the proof.
\end{proof}
We can finally prove the desired result.
\begin{thm} [Roots for phase polynomials in a $k$-extension]\label{extensionlowchar:thm}
	Let $X$ be an ergodic $G$-system. Let $P_1,P_2,...$ be at most countably many $(G,X,S^1)$-phase polynomials of degree $<m$ and let $p_1,p_2,...$ be prime numbers. Then, for every natural number $l$ there exists an $O_{m,l}(1)$-extension $Y$ with the following property: for every $n=p_1^{n_1}\cdot p_2^{n_2}\cdot...\cdot p_j^{n_j}$ where $j\in\mathbb{N}$ and  $n_1,n_2,...n_j\leq l$, there exist $(G^{(O_{m,l}(1))},Y,S^1)$ phase polynomials $Q_{i,n}:Y\rightarrow S^1$ of degree $<m$ such that $Q_{i,n}^{n} = P_i\circ \pi$, where $\pi:Y\rightarrow X$ is the factor map.
\end{thm}
\begin{proof}
	We prove the theorem by induction on $m$. If $m=1$, then by ergodicity $P_1,P_2,...$ are constants and the claim follows without extensions. Fix $m\geq 1$ and assume inductively that the claim holds for this $m$ and let $P$ be a phase polynomial of degree $<m+1$. For every $g_1,...,g_m\in G$ we have that $$c(g_1,...,g_m):=\Delta_{g_1}...\Delta_{g_m} P$$ is a symmetric multi-linear map. Let $\tilde{G}=G^{(l)}$ be the extension of $G$. We can lift $c$ to an element in $\text{SML}_m(\tilde{G},S^1)$. Once lifted, we can find $d\in \text{SML}_m(\tilde{G},S^1)$ with $d^n = c$. By the previous lemma applied to $d$ there exists an extension $\pi:Y\rightarrow X$ and a phase polynomial $Q$, such that $\Delta_{g_1}...\Delta_{g_m} Q^n = \Delta_{g_1}...\Delta {g_m} P\circ\pi$ and so by ergodicity $P\circ \pi/Q^n$ is of degree $<m-1$ and the claim for this $P$ follows by induction hypothesis. The same argument holds if applied to all $P_1,P_2,...$ simultaneously, as required.
\end{proof}
The results in Theorem \ref{extension:thm} and Theorem \ref{extensionlowchar:thm} require to extend $X$ by a zero dimensional group multiple times. For this reason it will be convenient to use the following definition.
\begin{defn} [Zero dimensional extension]
	Let $X$ be an ergodic $G$-system. We say that an extension $Y$ is a zero dimensional extension of $X$ if there exists finitely many zero dimensional groups $\Delta_1,...,\Delta_n$ such that $Y=\left((X\times_{\rho_1} \Delta_1) \times_{\rho_2} \Delta_2\times...\right)\times_{\rho_n} \Delta_n$ for some cocycles $\rho_1,...,\rho_n$. We say that a zero dimensional extension is of exponent $l$ if $\Delta_1,...,\Delta_n$ are of exponent $l$.
\end{defn}
We note that by the Mackey-Zimmer theory, $Y$ can be written as a single extension of $X$ by a zero dimensional group, but we do not use this here.\\
Below we prove various corollaries of Theorem \ref{extension:thm} and Theorem \ref{extensionlowchar:thm}. We begin with the following important lemma which allows us to reduce any Conze-Lesigne type equation to the torus.
\begin{lem} \label{fixCL:lem}
       Let $l,m\geq 1$, $X$ be an ergodic $G$-system and $U$ be a finite dimensional compact abelian group of exponent $l$. Let $\rho:G\times X\rightarrow U$ a cocycle of type $<m$ and suppose that for every $\chi\in\widehat U$ the cocycle $\chi\circ\rho$ is $(G,X,S^1)$-cohomologous to a phase polynomial of degree $<m$. Then, there exists some $r=O_{m,l}(1)$ and a zero dimensional $r$-extension $\pi:Y\rightarrow X$ such that $\rho\circ \pi$ is $(G^{(r)},Y,U)$-cohomologous to a phase polynomial of degree $<m$. Moreover, the extension $Y$ is independent of $\rho$.
\end{lem}
\begin{proof}
	Let $\Delta$ be a zero dimensional subgroup of exponent $l$ such that $U/\Delta\cong (S^1)^\alpha\times \prod_{i=1}^\beta C_{p_i^{\beta}}$ is a Lie group (i.e. $\alpha,\beta$ are finite natural numbers). We denote by $\chi_1,...,\chi_n$ the lifts of the coordinate maps of $U/\Delta$ to $\widehat U$.\\
	
	By assumption, for every $\chi\in \widehat U$ there exists a phase polynomial $P_\chi$ and a measurable map $F_\chi$ such that 
	\begin{equation} \label{CL}\chi\circ\rho = P_\chi \cdot \Delta F_\chi.
	\end{equation}
	Our goal is to find a choice of $P_\chi$ and $F_\chi$ such that $\chi\mapsto P_\chi$ and $\chi\mapsto F_\chi$ are homomorphisms. Recall that the dual group of $U/\Delta$ takes the form $\mathbb{Z}^\alpha \oplus \mathbb{Z}/p_1^{n_1}\mathbb{Z}\oplus...\oplus \mathbb{Z}/p_\beta^{n_\beta}\mathbb{Z}$. As a first step we show that for every prime $p$ and a natural number $n$, if $\chi:U/\Delta\rightarrow S^1$ is a character of order $p^n$ then we can replace $P_\chi,F_\chi$ such that equation (\ref{CL}) holds for the new replacements and at the same time $P_
	\chi$ and $F_\chi$ takes values in $C_{p^n}$.\\ From Equation (\ref{CL}) and the fact that $\chi$ is of order $p^n$, it follows that $\Delta F_{\chi}^{p^n} = \overline{P_{\chi}} ^{p^n}$. In particular, $F_{\chi}^{p^n}$ is a phase polynomial of degree $<m+1$. If $n$ is sufficiently large with respect to $m$ (Proposition \ref{PPC}) then $P_{\chi}^{p^n}(g,\cdot)$ is trivial for every $g\in G$ of order $p$. For such $n$ apply Theorem \ref{extension:thm}. Otherwise, $n=O_m(1)$ and we apply Theorem \ref{extensionlowchar:thm}. We conclude that there exists a zero dimensional $O_m(1)$-extension $\pi_1:\tilde{X}\rightarrow X$ and a phase polynomial $Q_{\chi} :\tilde{X}\rightarrow S^1$ of degree $<m+1$ such that $Q_{\chi}^{p^n}= F_{\chi}^{p^n}\circ \pi_1$. Now, we replace $F_{\chi}$ with $F_{\chi}\circ\pi_1/Q_{\chi}$ and $P_{\chi}$ with $P_{\chi}\circ \pi_1\cdot \Delta Q_{\chi}$. We can therefore assume that 
	\begin{equation} \label{CL2}\chi\circ (\rho \circ \pi_1) = P_{\chi} \cdot \Delta F_{\chi}
	\end{equation} and at the same time $F_{\chi},P_{\chi}$ takes values in $C_{p^n}$, as desired.\\
	We conclude that there exist homomorphisms $\chi\mapsto P_\chi$ and $\chi\mapsto F_\chi$ from the dual of $U/\Delta$ to $P_{<m}(G,\tilde{X},S^1)$ and $\mathcal{M}(\tilde{X},S^1)$, respectively such that equation (\ref{CL2}) holds.\\ Our next step is to extend these homomorphisms to $\widehat U$. Since $\Delta$ is of exponent $l$, by Theorem \ref{torsion} there exists a countable multiset of primes $I=\{p_1,p_2,...\}$ such that $\Delta = \prod_{i\in\mathbb{N}} C_{p_i^{n_i}}$ where $n_i\leq l$ for every $i$. Let $\tau_1,\tau_2,...\in\widehat \Delta$ be the coordinate maps. Using the Pontryagin duality, we lift each of the $\tau_i$'s to $\widehat U$ arbitrarily. Abusing notation, we denote the lifts of these characters by $
	\tau_1,\tau_2,...$ as well. Observe that $\tau_1,\tau_2,...$ and $\chi_1,...,\chi_n$ form a generating set of $\widehat U$. Therefore, in order to extend the homomorphisms above to $\widehat U$ it is left to work out the relations of the form $\tau_i^{p_i^{n_i}} = \chi_1^{l_1}...\chi_n^{l_n}$ where $n_i\leq l$ is as before and $l_1,...,l_n$ are natural numbers. \\
	Fix some $\tau:U\rightarrow S^1$ as above and suppose that $\tau^{p^r}=\chi_1^{l_1}\cdot...\cdot\chi_n^{l_n}$, for some prime $p$ and a natural number $r\leq l$. Equation (\ref{CL2}) implies, $$P_{\tau} ^{p^{r}} \cdot \Delta F_{\tau}^{{p}^{r}} = \prod_{i=1}^n \chi_i^{l_i} P_{\chi_i}^{l_i}\cdot \Delta F_{\tau_i}^{l_i}$$ 
	and we conclude that $$\frac{F_{\tau}^{p^{r}}}{\prod_{i=1}^nF_{\chi_i}^{l_i}}$$ is a phase polynomial of degree $<m+1$.\\
	Therefore by Theorem \ref{extensionlowchar:thm} there exists a zero dimensional $O_l(1)$-extension $\pi_2:Y\rightarrow \tilde{X}$ of $\tilde{X}$ (which is independent of $\tau$) of exponent $l$ and a phase polynomial $R_{\tau}:Y\rightarrow S^1$ such that $R_{\tau}^{p^{r}} = \left( \frac{F_{\tau}^{p^{r}}}{\prod_{i=1}^nF_{\chi_i}^{l_i}} \right)\circ \pi_2$.
	Now, we replace $F_{\tau}$ with $\frac{F_{\tau}\circ \pi_2}{R_{\tau}}$ and $P_{\tau}$ with $P_{\tau}\circ \pi_2\cdot \Delta R_{\tau}$ and we lift the other polynomials to $Y$ as well. After the replacement, we have that $P_{\tau}^{p^{r}}= \prod_{i=1}^{n} P_{\chi_i}^{l_i}$. This means that we can find homomorphisms $\chi\mapsto P_\chi$ and $\chi\mapsto F_\chi$ from $\widehat U$ to $P_{<m}(G,Y,S^1)$ and $\mathcal{M}(Y,S^1)$ respectively, such that (\ref{CL}) holds. Using the Pontryagin duality, we see that there exists a phase polynomial $P\in P_{<m}(G,Y,U)$ and a measurable map $F\in\mathcal{M}(Y,U)$ such that $\rho\circ \pi_2 = P\cdot \Delta F$. This completes the proof. 
\end{proof}
In a similar manner we have the following result.
\begin{lem} \label{polylift}
    Let $l,m\geq 1$ be natural numbers, let $X$ be an ergodic $G$-system and $U$ be a compact abelian group. Let $\varphi:V\rightarrow U$ be a surjective homomorphism from a compact abelian group $V$ onto $U$ and suppose that the kernel of $\varphi$ is a totally disconnected group of exponent $l$. Then, there exists an $O_{m,l}(1)$-extension $\pi:Y\rightarrow X$ with the property that for every phase polynomial $p:X\rightarrow U$ of degree $<m$ there exists a phase polynomial $\tilde{p}:Y\rightarrow V$ such that $\varphi\circ\tilde{p}=p\circ \pi$.
 \end{lem}
 \begin{proof}
 Using the Pontryagin duality, we see that the surjective homomorphism $\varphi$ gives rise to an injective homomorphism $\widehat\varphi : \widehat U \rightarrow \widehat V$. Therefore, we can assume without loss of generality that $\widehat U\leq \widehat V$. Let $p:X\rightarrow U$ be as in the theorem, then for every $\chi\in\widehat U$, $\chi\circ p:X\rightarrow S^1$ is also a phase polynomial of degree $<m$ and $\chi\mapsto \chi\circ p$ is clearly a homomorphism. Arguing as in the previous lemma we see that by passing to an extension we can extend this homomorphism to $\widehat V$. Namely, there exists an $O_{m,l}(1)$-extension $\pi:Y\rightarrow X$ and a homomorphism $\chi\mapsto \tilde{p}_\chi$ from $\widehat V$ to $P_{<m}(X,S^1)$ such that $\tilde{p}_\chi = p_\chi\circ \pi$ for every $\chi\in\widehat U$. By the Pontryagin duality, there exists a phase polynomial $\tilde{p}:Y\rightarrow V$ such that $\chi\circ\tilde{p} = \tilde{p}_\chi$. This completes the proof.
 \end{proof}
We recall some definitions from \cite{Berg& tao & ziegler}.
\begin{defn}  [Quasi-cocycles]\label{qcoc:def}
	Let $X$ be a $G$-system and $k\geq 0$ be a natural number. We say that $f$ is a quasi-cocycle of order $<k$ if $d^{[k]}f:G\times X^{[k]}\rightarrow S^1$ is a cocycle. 
\end{defn}
Note that by Lemma \ref{PP} this means that for all $g,g'\in G$ there exists a phase polynomial $p_{g,g'}$ of degree $<k$ such that $$\frac{f(g+g',x)}{f(g,x)\cdot f(g',T_gx)} = p_{g,g'}(x).$$
We also need the following definition.
\begin{defn}
	We say that a function $f:G\times X\rightarrow S^1$ is a line-cocycle if for any $g\in G$ we have $$\prod_{k=0}^{\text{order(g)-1}} f(g,T_g^k x) = 1.$$
\end{defn}
We weaken the assumptions in Theorem \ref{HK1}.
\begin{thm} \label{HK1functions}
	Theorem \ref{HK1} holds with the weaker assumption that each $\rho_\omega$ is a quasi-cocycle of order $<k-1$ and a line-cocycle.
\end{thm}

\begin{proof}
	The proof is the same as in Theorem \ref{HK1} and therefore will not be repeated. The main observation is that in the proof of Theorem \ref{countable} we only used the fact that $\rho$ is a cocycle for two purposes: First, so we can apply Lemma \ref{Cdec:lem} which we now can replace with Lemma \ref{dec:lem} and second, to eliminate the term $p'_u$ in equation (\ref{pu}). This time, $p'_u:G\rightarrow S^1$ is a line-cocycle. In particular,  $p'_u(g)$ is of finite order for every $g\in G$. Since $\mathcal{H}_l$ is connected, it is divisible and therefore the homomorphism $u\mapsto p'_u$ is trivial.
\end{proof}
Let $X$ be an ergodic $G$-system. In the next lemma we see how the extension theorems are useful to construct line-cocycles from arbitrary functions of finite type.
\begin{lem}\label{line} Let $k,m\geq 1$ and $X$ be an ergodic $G$-system of order $<k$. Then there exists $r=O_{k,m}(1)$ an ergodic zero-dimensional $O_{k,m}(1)$-extension $\pi:Y\rightarrow X$, with the following property: for every function $f:G\times X\rightarrow S^1$ of type $<m$ there exists a phase polynomial $p:G^{(r)}\times Y\rightarrow S^1$ of degree $<m$ such that $f\circ \pi/p$ is a line cocycle.
\end{lem}

\begin{proof}
	Let $g\in G$ and denote by $n$ the order of $g$. Let $f$ be of type $<m$. Since $d^{[m]} f$ is a coboundary, it is also a cocycle. We conclude that $$d^{[m]} \prod_{k=0}^{n-1} f(g,T_g^k x) = 1.$$
	Lemma \ref{PP} implies that $\prod_{k=0}^{n-1} f(g,T_g^k x)$ is a $T_g$-invariant phase polynomial of degree $<m$. We apply Theorem \ref{extensionlowchar:thm} for every $g\in G$ (simultaneously). We see that there exist an $O_{k,m}(1)$-extension $Y$ and a phase polynomial $p:G\times Y\rightarrow S^1$ of degree $<m$ such that $p(g,\cdot)$ is $T_g$-invariant for every $g\in G$ (see Remark \ref{inv}) and $p(g,\cdot)^n = \prod_{k=0}^{n-1} f(g,T_g^k x)$. It follows that $f\circ\pi/p$ is a line-cocycle, as required.
\end{proof}
The following theorem summarizes the main results in this section.
\begin{thm} \label{everything}
    Let $k,m,l,\alpha \geq 1$ and let $X$ be an ergodic system of order $<k$. Let $U$ be a finite dimensional group of exponent $\alpha$ and $\varphi:V\rightarrow U$ a surjective homomorphism such that $\ker \varphi$ is a zero dimensional group of exponent $l$. Then, there exists a zero dimensional $O_{k,m,l,\alpha}(1)$-extension $Y$ of $X$ with a factor map $\pi:Y\rightarrow X$ such that the following properties hold.
    \begin{itemize}
        \item { Let $\rho:G\times X\rightarrow U$ be a cocycle. If $\chi\circ\rho\in P_{<m}(G,X,S^1)\cdot B^1(G,X,S^1)$ for every $\chi\in\widehat U$, then $\rho\circ\pi \in P_{<m}(G,Y,U)\cdot B^1(G,Y,U)$.}
        \item {For every phase polynomial $p:X\rightarrow U$ of degree $<m$, there exists a phase polynomial $\tilde{p}:Y\rightarrow V$ such that $\varphi\circ\tilde{p} = p\circ \pi$.}
        \item {For every function $f:G\times X\rightarrow S^1$ of type $<m$, there exists a phase polynomial of degree $<m$, $p:G\times X\rightarrow S^1$ such that $f\circ \pi/p$ is a line cocycle. }
    \end{itemize}
\end{thm}
\section{Proof of Theorem \ref{Main2:thm} Part I} \label{Main2:proof1}
Throughout the rest of this paper we let $G$ denote a group of the form $\bigoplus_{p\in P}\mathbb{Z}/p^m\mathbb{Z}$ where $P$ is a multiset of primes and $m\in\mathbb{N}$. Moreover, we will no longer deal with general nilpotent systems and so whenever we say a nilpotent system we implicitly assume that the homogeneous group is the Host-Kra group. Our goal is to prove the following result.
\begin{thm} [Any finite dimensional system is a factor of a nilpotent system] \label{InvFDr1:thm}
	Let $k,\alpha\in\mathbb{N}$ and let $X$ be an ergodic finite dimensional $G$-system of order $<k+1$ and exponent $\alpha$. Then, there exists a finite dimensional ergodic $O_{k,\alpha}(1)$-extension $Y$ of exponent $\beta=O_{k,\alpha,m}(1)$ such that $Y\cong \mathcal{G}(Y)/\Gamma$, for some totally disconnected group $\Gamma$.
\end{thm}
Let $\alpha\in\mathbb{N}$ and $X$ be as in the theorem above. Then,  $X=Z_{<k}(X)\times_\rho U$ where $U$ is a finite dimensional compact abelian group of exponent $\alpha$. Proving the theorem above by induction, we can replace $Z_{<k}(X)$ with a finite dimensional nilpotent system $\mG/\Gamma$. It is natural to ask when an element $s\in \mG$ has a lift in $\mG(X)$.
\begin{defn}
Let $\mG/\Gamma$ be a $k$-step nilpotent system, $\rho:G\times \mG/\Gamma\rightarrow U$ be a cocycle into some compact abelian group $U$ and $Y=\mG/\Gamma\times_\rho U$. We say that an element $s\in \mG$ has a lift in $\mG(Y)$ if there exists a transformation $\overline{s}\in\mG(Y)$ which induces the same action as $s$ on $\mG/\Gamma$. In this context we let $$\mG^\star=\{s\in \mG : s \text{ has a lift in } \mG(Y)\}.$$
\end{defn}
Note that translations by $U$ are automatically in $\mG(Y)$. Therefore, if the action of $\mG^\star$ on $\mG/\Gamma$ is transitive then the action of $\mG(Y)$ on $Y$ is transitive.\\
We need the following easy lemma.
\begin{lem} \label{open=full}
    Let $k\geq 1$ and $X=\mG/\Gamma$ be a $k$-step nilpotent $G$-system. Let $\mL$ be an open subgroup of $\mG(X)$ which contains $T_g$ for every $g\in G$. Then the action of $\mL$ on $X$ is transitive.
\end{lem}
\begin{proof}
Let $\Gamma_{\mL} = \Gamma\cap \mL$. The quotient $\mL/\Gamma_{\mL}$ can be identified with a $G$-invariant open and closed subset of $\mG/\Gamma_{\mG}$. Ergodicity implies that $\mL/\Gamma_{\mL} = \mG/\Gamma_{\mG}$ under the identification $l\cdot \Gamma_{\mL}\mapsto l\cdot \Gamma_{\mG}$. In particular, the action of $\mL$ on $X$ is transitive.
\end{proof}
We use the following criterion for lifting due to Host and Kra \cite[Lemma 10.6]{HK}.\footnote{Lemma 10.6 in \cite{HK} is formulated only in the case where $U$ is a torus. However, since this assumption does not play a role in their proof, the claim holds for every compact abelian group $U$.}
\begin{lem} \label{CLproperty:lem}
	Let $k\geq 0$ and let $X$ be an ergodic $G$-system of order $<k+1$. Write  $X=Z_{<k}(X)\times_\rho U$ for some compact abelian group $U$ and a cocycle $\rho:G\times Z_{<k}(X)\rightarrow U$ of type $<k$ with $d^{[k]}\rho = \Delta F$. Let $t\in\mathcal{G}(X)$, if there exists a map $\phi:Z_{<k}(X)\rightarrow U$ with the property that  \begin{equation}
	\label{CLequation}
	\Delta_{t^{[k]}} F = d^{[k]}\phi
	\end{equation} Then the transformation  \begin{equation}\label{form}\overline{t}(x,u) := (tx,\phi(x)u)
	\end{equation} is a lift of $t$ in $\mathcal{G}(X)$.\\
	Conversely, every element in $\mathcal{G}(X)$ is of the form (\ref{form}).
\end{lem}
Let $p:\mG(X)\rightarrow \mG(Z_{<k}(X))$ be the projection map (Lemma \ref{induce}). By the lemma above, the kernel of $p$ consists of transformations of the form $S_{1,F}$ where $F\in P_{<k+1}(Z_{<k}(X),U)$. As a consequence we have the following result.
\begin{lem}
    If $X$ is a finite dimensional system, then $\mG(X)$ is a finite dimensional group (See definition \ref{finitedim:def}).
\end{lem}
\begin{proof}
Using a proof by induction it is enough to show that $P_{<k+1}(Z_{<k}(X),U)$ is finite dimensional. Let $\Delta\leq U$ be a zero dimensional group such that $U/\Delta$ is a finite dimensional torus. The projection $U\rightarrow U/\Delta$ gives rise to a short exact sequence
\begin{equation}\label{A,B}1\rightarrow A\rightarrow P_{<k+1}(Z_{<k}(X),U) \rightarrow B\rightarrow 1
\end{equation} where $A\leq P_{<k+1}(Z_{<k}(X),\Delta)$ and $B\leq P_{<k+1}(Z_{<k}(X),U/\Delta)$ are closed subgroups. By Lemma \ref{sep:lem} we conclude that the subgroup of constants $P_{<1}(Z_{<k}(X),U/\Delta)\cong U/\Delta$ is an open subgroup of $P_{<k+1}(Z_{<k},U/\Delta)$ and therefore, $B$ is finite dimensional. Moreover, since every profinite group is embedded in a direct product of finite groups we can assume that $\Delta\leq \prod_{i=1}^\infty F_i$, which implies that  $P_{<k+1}(Z_{<k}(X),\Delta)\leq \prod_{i=1}^\infty P_{<k+1}(Z_{<k}(X),F_i)$. For every $i$, $P_{<k+1}(Z_{<k}(X),F_i)$ is discrete, this implies that the product $\prod_{i=1}^\infty P_{<k+1}(Z_{<k}(X),F_i)$ is totally disconnected. In particular, the closed subgroup $P_{<k+1}(Z_{<k}(X),\Delta)$ and the closed subgroup $A$ are zero dimensional. The short exact sequence (\ref{A,B}) implies that $P_{<k+1}(Z_{<k}(X),U)$ is finite dimensional, as required.
\end{proof}
The proof of Theorem \ref{InvFDr1:thm} is reduced to solving the following lifting problem.
\begin{thm} \label{mainclaim}
    Let $k,m\geq 1$. Let $X=\mG/\Gamma$ be a finite dimensional $k$-step nilpotent system of exponent $m$ and suppose that $\mG$ is an open subgroup of $\mG(X)$ which contains $T_g$ for every $g\in G$. Let $\rho:G\times X\rightarrow U$ be a cocycle of type $<k+1$ into some finite dimensional compact abelian group $U$ of exponent $\alpha$. Then, there exists an ergodic zero-dimensional $O_{k,m}(1)$-extension $\pi:Y_0\rightarrow X$ of order $<k+1$ and exponent $m'=O_{k,m}(1)$ which is independent on $\rho$ such that the following properties hold.
    \begin{enumerate}
        \item{The subgroup $$\mathcal{L}_0 = \{s\in \mG : \text{ there exists a lift of } s \text{ in } \mG(Y_0)^\star\}$$ is open, where $\mG(Y_0)^\star$ is defined with respect to the extension $Y_0\times_{\rho\circ\pi} U$.}
		\item { The subgroup of $\mathcal{G}(Y_0)^\star$ generated by all of the lifts of the elements in $\mathcal{L}_0$ acts transitively on $Y_0$. In particular, $Y_0$ is a finite dimensional $k$-step nilpotent system.}
    \end{enumerate}
\end{thm}
Given this result we prove Theorem \ref{InvFDr1:thm} in section \ref{identification:sec}.\\
The extension theorems from section \ref{ext:sec} play an important role in the proof of Theorem \ref{mainclaim}. However the use of these theorems require passing to an extension (or $k$-extension) of the original system. This leads to various difficulties which we will explain soon. First, we need to distinguish between two different types of extensions. These are, weakly mixing extensions and degenerate extensions. We begin with a definition for the former.  
\begin{defn} [Weakly mixing abelian extensions] Let $X$ be a system of order $<k$, $U$ be a compact abelian group and $\sigma:G\times X\rightarrow U$ a cocycle of type $<k$. The extension $X\times_{\sigma} U$ is called weakly mixing if $Z_{<k}(X\times_\sigma U) = X$. In this case we also say that $\sigma:G\times X\rightarrow U$ is weakly mixing.
\end{defn}
A classical result of Host and Kra \cite[Corollary 7.7]{HK} asserts that an extension of a system of order $<k$ by a cocycle of type $<k-1$ is also of order $<k$. We say that this kind of extensions are \textit{degenerate} and note that all of the extensions from section \ref{ext:sec} are degenerate.\\

Moreover, a cocycle $\sigma:G\times X\rightarrow U$ is not weakly mixing if and only there exists a character $1\not=\chi\in\widehat U$ such that $\chi\circ\sigma$ is of type $<k-1$ \cite[Proposition 4]{HK02}. We deduce the following result.
\begin{lem} \label{extweakmixing}
	Let $X$ be a system of order $<k$, and $\sigma:G\times X\rightarrow U$ a weakly mixing cocycle on $X$. Let $\pi:Y\rightarrow X$ be an ergodic extension of order $<k$, then $\sigma\circ\pi$ is weakly mixing.
\end{lem}
\begin{proof}
Suppose by contradiction that for some $1\not =\chi\in\widehat U$, $\chi\circ\sigma\circ\pi$ is of type $<k-1$. Therefore, the extension $Y\times_{\chi\circ\sigma\circ\pi}\chi(U)$ is of order $<k$. Since $X\times_{\chi\circ\sigma}\chi(U)$ is a factor of $Y\times_{\chi\circ\sigma\circ\pi}\chi(U)$, it is also of order $<k$. But then we have a contradiction because $Z_{<k}(X\times_\sigma U)$ must be a non-trivial extension of $X$.
\end{proof}
The following result will allow us to strengthen Lemma \ref{fixCL:lem} for weakly mixing cocycles.
\begin{lem} \label{bounded}
    Let $m,d\geq 1$. Fix a multiset of primes $P$ and let $G=\bigoplus_{p\in P}\mathbb{Z}/p^m\mathbb{Z}$ and $X$ be an ergodic $G$-system. Then for every prime $p\in P$ and a natural number $n\in\mathbb{N}$, if $\rho:G\times Z_{<d}(X)\rightarrow C_{p^n}$ is a weakly mixing cocycle and $(G,Z_{<d}(X),S^1)$-cohomologous to a phase polynomial of degree $<d$, then $n=O_{m,d}(1)$. Moreover, if $p\not\in P$ then $n=0$.
\end{lem}

\begin{proof}
Write $\rho  = P \cdot \Delta F$ where $P:G\times Z_{<d}(X)\rightarrow S^1$ is a phase polynomial cocycle of degree $<d$ and $F:Z_{<d}(X)\rightarrow S^1$. By assumption $P^{p^n}\cdot \Delta F^{p^n} = 1$, we conclude that $F^{p^n}$ is a phase polynomial of degree $<d+1$. If $n<d$ the claim follows, otherwise by Proposition \ref{PPC} we see that $F^{p^n}$ is $T_g$-invariant for every $g\in G$ of order $p$ (this part is trivial if $p\not\in P$). By Lemma \ref{extension:thm} we can pass to an extension of $X$ and assume that $F^{p^n}$ has a phase polynomial root $Q$ of degree $<d$. Let $P' = P\cdot \Delta Q$ and $F=F/Q$ then $\rho = P'\cdot \Delta F'$. Moreover, $P'$ takes values in $C_{p^n}$ and therefore by Proposition \ref{PPC} in $C_{p^l}$ for some $l=O_{m,d}(1)$ (or that $P'$ is trivial if $p\not\in P$). We conclude that $\rho^{p^l}$ is a coboundary on an extension of $Z_{<d}(X)$. This contradicts Lemma \ref{extweakmixing}, unless $n=l$ (or $n=0$ if $p\not\in P$).
\end{proof}
The role of Theorem \ref{extension:thm} in the proof of Lemma \ref{fixCL:lem} is to show that if $\chi$ takes values in some $C_{p^n}$ then we can replace $P_\chi$ and $F_\chi$ in equation (\ref{CL}) with $P_\chi'$ and $F_\chi'$ which takes values in $C_{p^{O_{m,d}(1)}}$. Therefore, the argument in the proof of Lemma \ref{fixCL:lem} requires passing to extensions by totally disconnected groups of potentially unbounded exponent. However now we know (by the lemma above) that if the cocycle is weakly mixing then the quantity $n$ must be bounded.  In other words we have the following stronger version of Lemma \ref{fixCL:lem}.
\begin{cor}
    In the settings of Lemma \ref{fixCL:lem}. If $X$ is of order $<k$ then there exists $\alpha=O_{k,m,l}(1)$ such that the extension $Y$ is a (bounded) tower of extensions of $X$ by totally disconnected groups of exponent $\alpha$.
\end{cor}
It is classical (see \cite[Proposition 7.6]{HK}) that every extension can be decomposed as a weakly mixing extension and a \textit{degenerate} extension. More formally, let $X$ be a system of order $<k$ and let $\sigma:G\times X\rightarrow U$ be a cocycle of type $<k+1$. If we let $W$ be the annihilator of $\{\chi\in\widehat U : \chi\circ\sigma \text{ is of type } <k\}$, then $X\times_{\sigma} U = X\times_{\sigma\mod W} U/W\times_\tau W$ where $X\times_{\sigma\mod W} U/W$ is a system of order $<k$ and $\tau$ is a weakly mixing cocycle. In particular, we see that it is enough to prove Theorem \ref{mainclaim} in the case where $\rho$ is of type $<k$ (the degenerate case) and in the case where $\rho$ is weakly mixing. In this section we prove the former, we begin with the following lemma about degenerate extensions.
\begin{lem} \label{factor}
Let $k\geq 1$ and let $X=Z_{<k}(X)\times_\sigma W$ be an ergodic $G$-system of order $<k+1$. Let $Y=Z_{<k}(Y)\times_\tau V$ be an ergodic extension of $X$ of the same order. Then, there exists a surjective homomorphism $\varphi:V\rightarrow W$ such that $\varphi\circ\tau $ is $(G,Z_{<k}(Y),W)$-cohomologous to $\sigma\circ\pi_{k}$ where $\pi_{k}:Z_{<k}(Y)\rightarrow Z_{<k}(X)$ is the factor map.
\end{lem}
\begin{proof}
	Let $\pi:Y\rightarrow X$ be the factor map. It is classical that $\pi$ defines a factor $\pi_{k}:Z_{<k}(Y)\rightarrow Z_{<k}(X)$ and we have that
	$$\pi(y,v) = (\pi_{k}(y),p(y,v))$$ where $p:Y\rightarrow W$.
	
	Since $\pi$ commutes with $G$, we conclude that $\Delta_g p(y,v) = \sigma(g,\pi_{k}(y))$. In particular, we see that $\Delta p$ is invariant under the action of $V$. Since $Y$ is ergodic, $\Delta_v p$ is a constant. We conclude that there exists a homomorphism $\varphi:V\rightarrow W$ and a measurable map $F:Z_{<k}(Y)\rightarrow W$ such that $p(y,v)=\varphi(v)\cdot F(y)$. Therefore, $\sigma(g,\pi_k(y))=\Delta_g p(y,v) = \varphi\circ\tau(g,y)\cdot \Delta F$. Since $Z_{<k}(Y)\times_{\sigma\circ\pi_{k}} W$ is a factor of $Y$, it is ergodic. It follows by Lemma \ref{minimal} that the image of $\varphi\circ\tau$ is $W$. We conclude that $\varphi$ is surjective as required.
\end{proof}
Our next result is that Theorem \ref{mainclaim} holds for degenerate extensions if the group $U$ is totally disconnected of bounded exponent. 
\begin{thm}[A degenerate version of Theorem \ref{mainclaim}] \label{redweaklymixing}
 Let $k,m\geq 1$ and let $X=\mG(X)/\Gamma$ be a finite dimensional $k$-step nilpotent system. Let $\rho:G\times X\rightarrow U$ be a cocycle of type $<k$ into some zero dimensional compact abelian group $U$ of exponent $m$. Then, there exists an ergodic $O_{k,m}(1)$-extension $\pi:Y_0\rightarrow X$ with the same properties as in Theorem \ref{mainclaim}.
\end{thm}
Recall that by Lemma \ref{factor} if $Y$ is an extension of $X$ of the same order then we can find compact abelian groups $V$ and $U$ such that $Y=Z_{<k}(Y)\times_\tau V$ and $X=Z_{<k}(X)\times_\sigma U$. Moreover, there exists a surjective homomorphism $\varphi:V\rightarrow U$ such that $\varphi\circ\tau$ is cohomologous to $\sigma\circ\pi_k$ where $\pi_k:Z_{<k}(Y)\rightarrow Z_{<k}(X)$ is a factor map. It will be convenient to lift the elements of $\mG(X)$ to an intermediate factor $\tilde{X} = Z_{<k}(Y)\times_{\sigma\circ\pi_k} U$ and only then to lift them to $Y$. Each of these steps will require extending the $k$-th Host-Kra factor further. For the first lift we need the following corollary of Lemma \ref{CLproperty:lem}.
\begin{cor} \label{liftweaklymixing}
	Let $k\geq 1$. Let $X$ be an ergodic $G$-system of order $<k$ and $\rho:G\times X\rightarrow U$ a weakly mixing cocycle. Let $\pi:Y\rightarrow X$ be an ergodic extension of $X$ of order $<k$ and $s\in\mathcal{G}(X)$. If $s$ has a lift in $\mathcal{G}(X\times_\rho U)$ and a lift in $\mathcal{G}(Y)$, then $s$ can be lifted to $\mathcal{G}(Y\times_{\rho\circ\pi} U)$.
\end{cor}
\begin{proof}
Write $d^{[k]}\rho =\Delta F$ and let $s\in \mG(X)$ be as in the claim. By Lemma \ref{CLproperty:lem} there exists a measurable map $\phi_s:X\rightarrow U$ such that $S_{s,\phi_s}\in\mathcal{G}(X\times_\rho U)$. Let $\tilde{s}$ be a lift of $s$ in $\mathcal{G}(Y)$, since $$\Delta_{\tilde{s}^{[k]}} F\circ \pi = (\Delta_{s^{[k]}} F) \circ \pi = d^{[k]}\phi_s\circ\pi$$ we conclude by Lemma \ref{extweakmixing} and Lemma \ref{CLproperty:lem} that the transformation $S_{\tilde{s},\phi_s\circ\pi}$ is a lift of $\tilde{s}$ in $\mathcal{G}(Y\times_{\rho\circ\pi} U)$, as required.
\end{proof}
For convenience, if $X$ is a factor of a system $Y$, $\pi^Y_X:Y\rightarrow X$ is the factor map and $F$ is a function on $X$ we refer to $F\circ\pi^Y_X$ as the lift of $F$ to $Y$. We turn to the proof of Theorem \ref{redweaklymixing}.
\begin{proof} 
Let $X=\mG/\Gamma$ and $\rho:G\times X\rightarrow U$ be as in Theorem \ref{redweaklymixing} and assume that $\rho$ is of type $<k$ and $U$ is totally disconnected of exponent $m$. By proposition \ref{abelext:prop} there exists a compact abelian group $W$ and a cocycle $\sigma:G\times Z_{<k}(X)\rightarrow W$ such that $X=Z_{<k}(X)\times_{\sigma} W$. Let $Y=X\times_\rho U$ and write $Y=Z_{<k}(Y)\times_{\tau} V$ for some compact abelian groups $V$ and a cocycle $\tau$. Let $\pi_k:Z_{<k}(Y)\rightarrow Z_{<k}(X)$ be the factor map, then by Lemma \ref{factor} there exist a surjective homomorphism $\varphi:V\rightarrow W$ and a measurable map $F:Z_{<k}(Y)\rightarrow W$ such that that $\sigma\circ \pi_k=\varphi\circ\tau\cdot \Delta F$.
\begin{figure}[H]
		\includegraphics{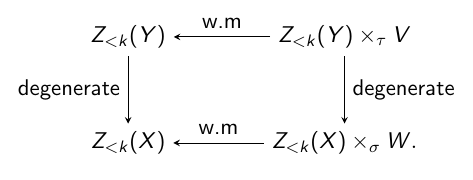}
	\end{figure}
Observe that $Z_{<k}(Y)$ is a degenerate extension of $Z_{<k}(X)$ by a zero dimensional group and $Y$ is a weakly mixing extension of $Z_{<k}(Y)$. Therefore we can apply the induction hypothesis in order to lift transformations from $Z_{<k}(X)$ to $Z_{<k}(Y)$ and then use the weakly mixing case to lift from $Z_{<k}(Y)$ to $Z_{<k}(Y)\times_{\tau}V$. Each time we have to pass to an extension. We conclude that there exists an extension $Y_0 = Z_{<k}(Y_0)\times_{\tau_0} V$ where $\tau_0=\tau\circ\pi^{Z_{<k}(Y_0)}_{Z_{<k}(Y)}$ and such that $$\mH = \{s\in \mG(Z_{<k}(X)) : \text{ there exists a lift of } s \text { in } \mG(Y_0)\}$$ is an open subgroup of $\mG(Z_{<k}(X))$. Indeed, $\mathcal{H}$ here equals to the group $\mathcal{L}$ defined in Theorem \ref{mainclaim}.\\

Now let $S_{s,\phi}$ be any transformation in $\mG(X)$. If $s\in\mH$ then there exists a lift $\overline{s}\in \mG(Z_{<k}(Y_0))$ and a measurable map $\psi:Y_0\rightarrow V$ such that $S_{\overline{s},\psi}\in \mG(Y_0)$.\\ Let $\tilde{\sigma}:G\times Z_{<k}(Y_0)\rightarrow W$ be the lift of $\sigma$ to $Z_{<k}(Y_0)$, that is $\tilde{\sigma} = \sigma\circ  \pi^{Z_{<k}(Y_0)}_{Z_{<k}(X)}$. Similarly let $\tilde{\phi}$ and $\tilde{F}$ be the lifts of $F$ and $\phi$ to $Z_{<k}(Y_0)$, respectively. We consider the intermediate factor $\tilde{X}= Z_{<k}(Y_0)\times_{\tilde{\sigma}} W$. By Lemma \ref{CLproperty:lem} and Corollary \ref{liftweaklymixing} we see that $S_{\overline{s},\tilde{\phi}}$  and $S_{\overline{s},\varphi\circ\psi\cdot \Delta_{\overline{s}}\tilde{F}}$ belong to $\mG(\tilde{X})$. Therefore,
$$\frac{\tilde{\phi}}{\varphi\circ\psi \cdot \Delta_{\overline{s}}\tilde{F}}\in P_{<k}(Z_{<k}(Y_0),W).$$
 By Theorem \ref{everything} we can pass to an extension $Y_1=Z_{<k}(Y_1)\times_{\tau_1} V$ where we can find a phase polynomial $p:Z_{<k}(Y_1)\rightarrow V$ such that $\varphi\circ p = \frac{\tilde{\phi}}{\varphi\circ\psi \cdot \Delta_{\overline{s}}\tilde{F}} $. Arguing as before we can pass to another extension $Y_2=Z_{<k}(Y_2)\times_{\tau_2} V$ and find an open subgroup $\mH'$ of $\mH$ of transformations in $\mG(Z_{<k}(X))$ which has a lift in $\mG(Y_2)$. Let $\overline{p}$ and $\overline{\psi}$ be the lifts of $p$ and $\psi$ to $Z_{<k}(Y_2)$, respectively. We conclude that for every $s\in \mH'$ the transformation $S_{\tilde{s},\overline{\psi}\cdot p}$ induces the action of $S_{s,\phi}$ on $X$. Since $\mH'$ is open, the group of all elements in $\mH$ which admits a lift in $\mG(X)$ contains the open subgroup $\mH'$ and is therefore also open. This completes the proof.
\end{proof}
The case where the group $U$ in theorem \ref{redweaklymixing} is of dimension greater than zero follows by a similar argument. Indeed, if $V$ is an extension of $W$ by a finite dimensional group $U$ of dimension $n$, then $V$ is an extension of $(S^1)^n\times W$ by a zero dimensional group. Since the extensions which arise from Theorem \ref{extensionlowchar:thm} are zero dimensional, this case is not needed in the proof of Theorem \ref{InvFDr1:thm} and we leave the details for the interested reader.
\section{Proof of Theorem \ref{Main2:thm} Part II} \label{Main2:proof2}
As mentioned in the previous section it is enough to prove Theorem \ref{mainclaim} in the case where the extension is weakly mixing. In this case we have the following result of Host and Kra \cite[Lemma 10.8]{HK}.
\begin{lem} \label{exactdiff2:lem}
	Let $k\geq 0$ and let $X$ be an ergodic $G$-system of order $<k+1$ and $\sigma:G\times X\rightarrow U$ be a weakly mixing cocycle with $d^{[k+1]}\sigma = \Delta F$. Let $j$ be an integer with $0\leq j <k+1$ and let $\mathcal{G}_j$ be the $j$-th group in the lower central series for $\mathcal{G}$. Then, for every $t\in \mathcal{G}_j$ and a measurable map $\phi:X\rightarrow U$ the following are equivalent:
	\begin{enumerate}
	\item { For every $(k-j+1)$-face $\beta$, $\Delta_{t^{[k+1]}_\beta} F  = d^{[k+1]}_\beta \phi$.}
	\item {For every $(k-j)$-face $\alpha$, $\frac{\Delta_{t^{[k+1]}_\alpha}F}{d^{[k+1]}_\alpha \phi}$ is an invariant function on $X^{[k+1]}$.}
	\end{enumerate}
\end{lem}
 \subsection{The main objects in the proof} \label{mobj} 
	 We begin by describing the objects which will be used in the proof of Theorem \ref{mainclaim}. For $0\leq j \leq k+1$ we construct a tower of extensions $Y_0\rightarrow Y_1\rightarrow...\rightarrow Y_{k+1}=X$, where each $Y_j$ is an $O_{k,m,j}(1)$- extension of $Y_{j+1}$ (and therefore of $X$) of order $<k$ with the following properties:
	 \begin{itemize}
	 	\item {The subgroup 
	 	$$\mathcal{L}_j:= \{s\in \mG(X) : s \text{ has a lift in } \mG(Y_j)\}$$ is open in $\mG(X)$.}
	 	\item { The subgroup 
	 	$$V_j:= \{s\in \mathcal{G}_j(X)\cap \mL_j : s \text{ has a lift in } \mG(Y_j)^\star\}$$ is open in $\mG_j(X)$.}
	 	\item {For technical reasons we will also prove that $V_j$ contains the subgroup generated by the commutators $\{[s,g]:s\in\mathcal{G}_{j-1}(X)\cap \mathcal{L}_j,g\in G\}$.}
	 \end{itemize}
 We note that in each step the group $G$ is replaced with some $O_{k,m,j}(1)$-extension, we abuse notation and denote all of them by $G$. Since we only construct $k+1$ extensions $(Y_k,Y_{k-1},...,Y_{0})$, the final extension $Y_0$ is an $O_{k,m}(1)$-extension of $X$.\\ We construct these objects by downward induction on $j$; For $j=k+1$ we can take $Y_{k+1}=X$, $\mathcal{L}_{k+1} = \mathcal{G}$ and $V_{k+1}=\{e\}$. For $j=k$ we have the following lemma.
\begin{lem}
In the setting of Theorem \ref{mainclaim}, there exists an ergodic zero dimensional $O_{k,m}(1)$-extension $Y_{k}$ of order $<k$, such that $\mL_k$ is open and $V_k=\mG_k\cap \mL_k$.
\end{lem}
The proof is a modification of the arguments of Host and Kra \cite[Lemma 10.9]{HK}.
\begin{proof}
	Let $F:X^{[k+1]}\rightarrow U$ be such that $d^{[k+1]}\rho = \Delta F$. Since $\mathcal{G}(Z_{<k}(X))$ is $(k-1)$-step nilpotent we conclude that any $t\in\mathcal{G}_{k}$ is an automorphism which fixes $Z_{<k}(X)$. Therefore, by Lemma \ref{dif:lem}, $\Delta_t \rho$ is of type $<1$. Let $\chi\in\widehat U$, then by Lemma \ref{type0} $\Delta_t \chi\circ \rho$ is $(G,X,S^1)$-cohomologous to a constant. We conclude by Lemma \ref{fixCL:lem} that there exists a zero dimensional $O_{k,m}(1)$-extension $\tilde{\pi}:\tilde{X}\rightarrow X$ by a cocycle of type $<1$ such that \begin{equation} \label{cleq} (\Delta_t \rho)\circ\tilde{\pi} = c_t \cdot \Delta F_t
	\end{equation} for some constant $c_t:G\rightarrow U$ and a measurable map $F_t:\tilde{X}\rightarrow U$. It follows that for every $1$-dimensional face $\alpha$, \begin{equation}\label{invariant}\frac{(\Delta_{t_\alpha^{[k]}}F)\circ \tilde{\pi}}{d_\alpha^{[k]}F_t}
	\end{equation}
	is invariant in $\tilde{X}^{[k]}$. Notice that $\tilde{X}$ is a degenerate extension of $X$, therefore by the induction hypothesis of Theorem \ref{redweaklymixing} there exists an $O_{k,m}(1)$-extension $\pi:Y_{k}\rightarrow \tilde{X}$ (and therefore of $X$). By the same theorem the group $\mathcal{L}_k\leq \mG$ with respect to this extension is open. Any $t\in \mathcal{L}_k$ has a lift in $\mG(Y_k)$. If in addition $t\in \mathcal{G}_k$ then equation (\ref{invariant}) and Lemma \ref{exactdiff2:lem} imply that $t$ has a lift in  $\mathcal{G}(Y_{k})^\star$, where $\mG(Y_k)^\star$ is defined with respect to the extension $Y_k\times_{\rho\circ\pi^{Y_k}_{X}} U$, as required. 
\end{proof}
 We climb up along the central series inductively. Suppose by induction that we have already constructed $Y_{j+1}$, $\mathcal{L}_{j+1}$  and $V_{j+1}$ as above. We prove:
\begin{lem} \label{above}
There exists an ergodic $O_{k,m}(1)$-extension $Y_{j}'$ of $Y_{j+1}$ such that $$\mathcal{L}_j' := \{s\in \mL_{j+1} : s \text{ has a lift in } \mG(Y_j')\}$$ is open in $\mL_{j+1}$ and at the same time, $$V_j':= \{s\in \mL_{j}'\cap \mG_j : s \text{ has a lift in } \mG(Y_j')^\star\}$$ is open in $\mL_j'\cap \mG_j$.
\end{lem}
\begin{rem}
For technical reasons, we will postpone the proof of this lemma until after Lemma \ref{lemcomplete}. In other words, we will assume for now that this lemma was already established, and proceed to prove Lemma \ref{lemcomplete} below. At the end of this section we will prove Lemma \ref{above} without relying on Lemma \ref{lemcomplete}. The advantage is that some ideas and notions used in the proof of Lemma \ref{above} are more natural in the settings of Lemma \ref{lemcomplete} and so we prefer to define these in the proof of Lemma \ref{lemcomplete} first.
\end{rem}
The following lemma is the final step in the proof. We describe a process which allow (by passing to an extension) to add arbitrary countable set of transformations of the form $[s,g]$ to $V_j'$, where $s\in\mL_j'\cap \mG_{j-1}$ and $g\in G$.
\begin{lem} \label{lemcomplete}
There exists an ergodic zero dimensional $O_{k,m,j}(1)$-extension $Y_j$ of $Y_j'$ such that $\mL_j = \{s\in \mL_j' : s \text{ has a lift in } \mG(Y_j)\}$ is open and $V_j\subseteq\mathcal{G}_j\cap \mathcal{L}_j$ satisfy the properties in section \ref{mobj}.
\end{lem}
\begin{proof}
	Let $s\in\mathcal{L}_j'$, and let $g\in G$ be a generator in the natural basis of $G$. By the structure of $G$, the order of $g$ is $p^{O_{k,m,j}(1)}$ for some prime $p$ (the exact power is not important). Since $\mathcal{G}$ is $k$-step nilpotent, the order of $[s,g]$ is also  $p^{O_{k,m,j}(1)}$ for possibly higher but bounded power. Since $V_j'$ is open in $\mL_j'\cap \mG_j$, it is of at most countable index in that group. Therefore, we can find a countable set $\{s_n\}_{n\in\mathbb{N}}$ of transformations in $\mL_{j}'\cap \mathcal{G}_{j}$ where each $s_n$ is of order $p_n^{O_{k,m,j}(1)}$, for some primes $p_n$ where the power is bounded uniformly for all $n$ and such that $\bigcup_{n\in \mathbb{N}} s_n V_j'$ contains $\{[s,g]:s\in \mathcal{G}_{j-1}\cap \mathcal{L}'_j,g\in G\}$. By adding inverses, we may assume that $\{s_n:n\in\mathbb{N}\}$ contains all of its inverses.\\
	Let $C_0:=\{s_n:n\in\mathbb{N}\}$, and for every $n\geq 1$ let $C_n:=\{[s,g]:s\in C_{n-1},g\in G\}$ and $C=\bigcup_{n\in\mathbb{N}} C_n$. \\
	\textbf{Definition}: For $s\in C$ we define the complexity of $s$ by $$\text{comp}(s):=\max\{n : s\in C_n\}$$
	We need the following result.
	\begin{lem} \label{mainlem} For every $n\in\mathbb{N}$, there exists a zero dimensional $O_{k,m}(1)$-extension $\pi_n:Y_{j,n}\rightarrow Y_{j}'$ such that for every $s\in C_n$ and for every $0$-dimensional face $\alpha$, there exists  $\psi_s:Y_{j,n}\rightarrow U$ such that 		 	\begin{equation}\label{eqt}(\Delta_{s_\alpha^{[k+1]}} F)\circ\pi_n - d_\alpha^{[k+1]}\psi_s
		 \end{equation}
		 is $T_g^{[k+1]}$-invariant for every $g\in G$.
	\end{lem}
\begin{proof}
 We prove the claim by downward induction on $n$. If $n\geq j$, then $C_j$ is trivial and the claim follows. Fix some $0\leq n<j$ and assume that Lemma \ref{mainlem} holds for all values greater than $n$. Let $s\in C_n$. For every $g\in G$, $[s^{-1},g^{-1}]$ has complexity greater than $n$. Therefore, by the induction hypothesis there exists an extension $\pi_{n+1}:Y_{j,n+1}\rightarrow Y_j'$ and a measurable map $\psi_{s,g}:Y_{j,n+1}\rightarrow U$ such that $(\Delta_{ [s^{-1},g^{-1}]_\alpha^{[k+1]}} F)\circ\pi_{n+1} - d^{[k+1]}_\alpha \psi_{s,g}$ is invariant for any $0$-face $\alpha$. Let $\overline{s}$ be any measure preserving transformation on $Y_{j,n+1}$ (not necessarily in $\mG(Y_{j,n+1})$), which induces the same action of $s$ on $Y_j'$.\footnote{Such lift always exists. Recall that $Y_j'$ is a tower of group extensions of $X$. Therefore, there exists a compact group $K$ such that $Y_j'\cong X\times K$ as measure spaces. In particular, $\overline{s}(x,k)=(sx,k)$ is a lift for $s$.} Consider the function
 \begin{equation} \label{theta}
\theta_s (g,y) : =   \psi_{s,g}(g\overline{s}y)\cdot \Delta_s \rho(g,\pi_{n+1}(y))
\end{equation}
As in the previous lemma, $\theta_s:G\times Y_{j,n+1}\rightarrow U$ is a function of type $<k-j+1$. Note that if the order of $g$ is co-prime to $p$ then $s$ commutes with $g$ and so we can take $\psi_{s,g}=1$.\\
We replace $\theta_s$ with a cocycle,\\
	\textbf{Claim:} By embedding $U$ into $(S^1)^\mathbb{N}$ using the Pontryagin dual, we view $\theta_s$ as a map into $(S^1)^\mathbb{N}$. We claim that there exists a constant $c_s:G\rightarrow (S^1)^\mathbb{N}$ and a cocycle $\theta_s':G\times Y_{j,n+1}\rightarrow (S^1)^\mathbb{N}$ such that for any generator $g\in G$, $\theta_s'(g)=c_s(g)\cdot\theta_s(g)$. Moreover, $c_s(g)=1$ whenever the order of $g$ is coprime to the order of $s$. \\
	\textbf{Proof of claim}: 
	Fix any $0$-face $\alpha$ and $h\in G$, by the induction hypothesis of Lemma \ref{mainlem} we know that $\Delta_h^{[k+1]}( \Delta_{[s^{-1},g^{-1}]^{[k+1]}_\alpha} F) \circ\pi_{n+1} = d_\alpha^{[k+1]} \Delta_h \psi_{s,g}$. The same calculation as in (\ref{eqF}) gives
	$$d_\alpha^{[k+1]} \Delta_h \theta_s(g) = \Delta_{h^{[k+1]}} \Delta_{g^{[k+1]}} (\Delta_{\overline{s}_\alpha^{[k+1]}} F\circ \pi_{n+1}).$$
	Since $g$ and $h$ commute, we have
 $d_\alpha^{[k+1]} \Delta_h \theta_s(g) = d_\alpha^{[k+1]} \Delta_g \theta_s (h)$ for any $0$-face $\alpha$. We conclude that 
 \begin{equation}
 \label{commute2}
\Delta_h \theta_s(g) = \Delta_g \theta_s (h). \end{equation}
Therefore, by ergodicity we have
	\begin{equation} \label{cocycle.eq} \frac{\theta_s (g+g')}{\theta_s(g)T_g \theta_s(g')} = c_s(g,g')
	\end{equation} for all $g,g'\in G$ and some constant $c_s(g,g')$.\\
	
	From this we conclude that $\prod_{k=0}^{\text{order}(g)-1} T_g^k \theta_s(g)$ is a constant in $U$. In order to take roots we embed $U$ in the divisible group $(S^1)^\mathbb{N}$ and choose $c_{s,\chi}(g)\in S^1$ such that $c_{s,\chi}(g)^{\text{order}(g)} = \prod_{k=0}^{p^d-1}T_g^{k}\chi\circ\theta_s$. Note that if $g$ is of order co-prime to $p$, then $s$ and $g$ commutes. In this case $\psi_{s,g}=1$ and $\prod_{k=0}^{\text{order}(g)-1}T_g^{k}\chi\circ\theta_s=1$. In other words, if $g$ is of order co-prime to $p$ we can take $c_{s,\chi}(g)=1$. Let $\tilde{\theta}_{s,\chi}(g):=\chi\circ\theta_s(g)/c_{s,\chi}(g)$. We see that $\prod_{k=0}^{\text{order}(g)-1}T_g^k \tilde{\theta}_{s,\chi}(g)=1$ and  $\Delta_h \tilde{\theta}_{s,\chi}(g) = \Delta_g \tilde{\theta}_{s,\chi}(h)$ for every $h,g\in G$. 
	Let $\theta_{s,\chi}'$ be the cocycle as in Proposition \ref{cocycle:prop}. Q.E.D.\\
	
	We return to the proof of Lemma \ref{mainlem}. Consider the cocycle $$\theta:=(\theta_{s,\chi}')_{s\in C_n,\chi\in\widehat U}:G\times Y_{j,n+1}\rightarrow (S^1)^\mathbb{N}$$ and choose a minimal cocycle $\sigma:G\times Y_{j,n+1}\rightarrow (S^1)^\mathbb{N}$ which is cohomologous to $\theta$. Let $Y_{j,n}=Y_{j,n+1}\times_\sigma W$ where $W$ is the image of $\sigma$ and assume for now that $W$ is zero dimensional (proof below). Since $\theta$ is cohomologous to $\sigma$ it is a coboundary on  $Y_{j,n}$. This implies that $\chi\circ\theta_s$ is cohomologous to constant for every $\chi\in\widehat U$ and $s\in C_n$. Using Theorem \ref{everything} we can replace $Y_{j,n}$ with an extension such that $\theta_s$ is $(G,Y_{j,n},U)$-cohomologous to a constant on that extension. In that case we can find $\psi_s:Y_{j,n}\rightarrow U$ such that $\theta_s - d\psi_s$ is a constant and this $\psi_s$ satisfies equation (\ref{eqt}). This completes the proof of Lemma \ref{mainlem}.
		\end{proof}
	
	We return to the proof of Lemma \ref{lemcomplete}. Consider the extension $Y_{j,0}$ from the previous lemma. This is a degenerate extension of $Y_j'$. Therefore, by Theorem \ref{redweaklymixing} we can find an extension $Y_j$ such that the subgroup $\mL_j = \{s\in \mL_j' : s \text{ has a lift in } \mG(Y_j)\}$ is open. Let $V_j$ be as in the theorem, since  $V_j'\cap\mL_j \subseteq V_j$, we have that $V_j$ is open in $\mL_j\cap \mG_j$. Recall that $\{[s,g] : s\in \mathcal{G}_{j-1}\cap \mathcal{L}_j',g\in G\}\subseteq C_0\cdot V_j'$. Let $s\in \mL_j'\cap \mG_{j-1}$ and $g\in G$. From \eqref{eqt} and Lemma \ref{CLproperty:lem} it follows that if $s$ has a lift in $\mG(Y_{j,0})$, then $[s,g]$ has a lift in $\mG(Y_{j,0})^\star$. From this and corollary \ref{liftweaklymixing}, we conclude that if $s$ has a lift in $\mG(Y_j')$, then $[s,g]$ it has a lift in $\mG(Y_j')^\star$. Since all elements in $\mL_j$ has lifts in $\mG(Y_j')$ we conclude that if $s\in \mL_j\cap \mG_{j-1}$ then $[s,g]$ has a lift in $\mG(Y_j')^\star$. We conclude that any element of the form $[s,g]$ where $s\in \mL_j\cap \mG_{j-1}$ and $g\in G$ has a lift in $\mG(Y_j')^\star$. In other words $\{[s,g]:s\in \mL_j\cap \mG_{j-1},g\in G\} \subseteq V_j$, as required. \\

	It is left to show that $W$ is zero dimensional. Fix $s\in C_n$ and $\chi\in\widehat U$, we prove that there exists $N$ such that $(\theta_{s,\chi}' )^ N$ is a coboundary. We need the following lemma.
	\begin{lem} \label{powers}
		 Let $Y$ be an ergodic extension of a $G$-system $X$ of order $<k$. We denote by $\mG$ be the Host-Kra group of $X$. Let $p$ be a prime number and write $G=G_p\oplus G_p^\perp$ where $G_p$ is the $p$-component of $G$. Let $m\geq 0$ and suppose that $f:G\times Y\rightarrow S^1$ satisfies that $d^{[m]}f=\Delta F$ for some $F:Y^{[m]}\rightarrow S^1$ which is measurable with respect to $X^{[m]}$. Then, for every $s\in\mathcal{G}$ of order $p^n$ for some $n\in\mathbb{N}$, there exists a function $\sigma_s:G\times X\rightarrow S^1$ and a natural number $N=O_{k,n,m}(1)$ such that with $\sigma_s(g,x)=1$ for all $g\in G_p^\perp$ and $$(\Delta_s f \cdot \sigma_s)^N$$ is a $(G,X,S^1)$-coboundary. Furthermore $(d^{[m]} \Delta_s f \cdot \sigma_s)^{N'} = \Delta \Delta_s^{[m]} F^{N'}$ for some natural number $N'=O_{k,n,m}(1)$.
	\end{lem}
We briefly explain the idea behind this result. Let $n$ be a natural number and let $s$ be a transformation of order $n$. Since $\Delta_{s^n} f=1$, the cocycle identity gives $$\Delta_s f^n = \prod_{k=0}^{n-1} \Delta_s \Delta_{s^k} \overline{f}.$$ Assume hypothetically that $s$ is an automorphism. Then, by Lemma \ref{dif:lem}, the type of  $\Delta_s f^n$ is smaller than the type of $\Delta_s f$. If we repeat this process iteratively, we will eventually get that some power of $\Delta_s f$ is a coboundary.\\
In the lemma we do not assume that $s$ is an automorphism. However, equation (\ref{theta}) indicates that up to a multiplication by some function it still behaves like one. The formal proof is given below.
	\begin{proof}[Proof of Lemma \ref{powers}]
		We prove the lemma by induction on $m$. For $m=0$ we have that $f=\Delta F$. Therefore, since $F$ is measurable with respect to $X$ we have, $$\Delta_s f = \Delta_s \Delta F = \Delta \Delta_s F \cdot V_s T_g \Delta_{[s^{-1},g^{-1}]} F$$ and the claim follows by taking $\sigma_s(g,x):=V_s T_g\Delta_{[s^{-1},g^{-1}]}\overline{F} (x)$. Note that if $g\in G_p^\perp$ then $s$ and $g$ commutes, in this case $\sigma_s(g,x)=1$.\\
		\textbf{Claim:} Fix $1\leq j\leq k$, let $s\in\mathcal{G}_j$ of order $p$ and  $\beta$ be an $(m-j)$-dimensional face (or a vertex if $j\geq m$). Then, there exists a natural number $M=O_{m,j,n}(1)$ and a function $\phi_s : Y\rightarrow S^1$ such that $\Delta_{s_\beta^{[m]}} F^M = d_\beta^{[m]} \phi_s^M$.
		
		We prove the claim by downward induction on $j$. If $j=k$, then $s=e$ and the claim is trivial. Fix $j<k$ and assume that the claim holds for all values greater than $j$. Let $s\in\mathcal{G}_j$ be as in the lemma, then by the induction hypothesis we see that for every $g\in G$ and every $(m-j-1)$-dimensional face $\beta$ there exist a power $M$, and $\phi_{s,g}:X\rightarrow S^1$ such that $$\Delta_{[s^{-1},g^{-1}]_\beta^{[m]}} F^M = d_\beta^{[m]} \phi_{s,g}^M.$$
		We use this to prove Lemma \ref{powers} for all $s\in \mG_j$ for this specific $j$ and then we use the lemma to prove the rest of the claim. Let $\sigma_s(g,x):=V_sT_g\phi_{s,g}(x)$ and observe that if $g\in G_p^\perp$ then $s$ and $g$ commute and so we can take $\phi_{s,g}=1$. Let $f'_s := \Delta_s f\cdot \sigma_s$. As in (\ref{eqF}) we have
		\begin{equation}
		\label{eqF'}
		\begin{split}
		\Delta (\Delta_{s_\beta^{[m]}}F^M) = V_{s_\beta}^{[m]}T_g^{[m]}\Delta_{[s^{-1},g^{-1}]_\beta^{[m]}} F^M \cdot \Delta_{s_\beta^{[m]}} \Delta F^M =\\
		=d_\beta^{[m]}V_s\circ T_g\phi_{s,g}^M \cdot \Delta_{s_\beta^{[m]}} d^{[m]}\rho^M = d_\beta^{[m]}{f'_s}^M.
		\end{split}
		\end{equation}
		Since this is true for every $(m-j)$-dimensional face $\beta$, we see by lemma \ref{typereduce} that $f_s'^M$ is of type $<m-j$ (or a coboundary if $j\geq m$). We conclude that there exists a measurable map $F_s:Y^{[m-j]}\rightarrow S^1$, which is measurable with respect to $X^{[m-j]}$ such that $d^{[m-j]}f_s'^M = \Delta F_s^M$. Moreover, we note here that $d^{[j]}F_s^M=\Delta_{{s}^{[m]}}F^M$. Since $f_s'^M$ is of smaller type, we can apply the induction hypothesis for the transformations $s,s^2,s^3,...,s^{p^n-1}$. We conclude that there exist $N, N' = O_{k,m,n}(1)$ and $\sigma_{s,l}'$ with
		\begin{equation} \label{cor:eq1} (d^{[m-j]}\Delta_{s^l} f_s' \cdot \sigma_{s,l}')^{N'} = \Delta \Delta_{{s^l}^{[m-j]}} F_{s}^{N'}
		\end{equation} 
		and
		\begin{equation} \label{cor:eq2}\Delta_{s^l} f_s'^N\cdot {\sigma'_{s,l}}^N\in B^1(G,X,S^1).
		\end{equation}
		By replacing $N$ and $N'$ with $NM$ and $N'M$ we can assume without loss of generality that $N$ and $N'$ are multiples of $M$.
		Recall, that $f'_s = \Delta_s f \cdot \sigma_s$. Let $p^n$ be the order of $s$ then by the cocycle identity we have,
		$$1=\Delta_{s^{p^n}} f = \prod_{k=0}^{p^n-1} \Delta_s V_{s^k} f = (\Delta_sf)^{p^n} \cdot \prod_{k=1}^{p^n-1} \Delta_{s^k} \Delta_s f$$ and it follows that $(\Delta_s f)^{p^n} = \prod_{k=1}^{p^n-1} \Delta_{s^k} \Delta_s \overline{f}$.\\
		From all of this we conclude that:
	\begin{equation}
	\label{all}
	\begin{split}
	(\Delta_s \overline{f})^{N'Np^n} \prod_{k=1}^{p^n-1}\Delta_{s^k}\sigma_s^{N'N}\cdot \prod_{k=1}^{p^n-1} \sigma_{s,k}'^{N'N}&=\\ \prod_{k=1}^{p^n-1} \Delta_{s^k}(\Delta_s f^{N'N}\sigma_s^{N'N})\cdot \sigma_{s,k}'^{N'N})&= \prod_{k=1}^{p^n-1} (\Delta_{s^k} f_s'\cdot \sigma_{s,k})^{N'N}
	\end{split}
		\end{equation} which by equation (\ref{cor:eq2}) is a coboundary. Moreover by equation (\ref{cor:eq1}) we have that \begin{equation} \label{cor:eq11}
		d^{[m-j]} 	(\Delta_s \overline{f})^{N'Np^n} \prod_{k=1}^{p^n-1}\Delta_{s^k}\sigma_s^{N'N}\cdot \prod_{k=1}^{p^n-1} \sigma_{s,k}'^{N'N} = \Delta \prod_{k=1}^{p^n-1} \Delta_{{s^k}^{[m-j]}} \overline{F}_{s}^{NN'}.
		\end{equation}
	Choose $\tilde{N} = N'Np^n$ and $\tilde{\sigma}_s$ any measurable function which satisfies that  $\sigma_s^{\tilde{N}} =  (\prod_{k=1}^{p^n-1}\Delta_{s^k}\sigma_s^{N'N}\cdot \prod_{k=1}^{p^n-1} \sigma_{s,k}'^{N'N})^{-1}$ and that $\tilde{\sigma}_s(g,x)=1$ whenever $g\in G_p^\perp$. We conclude that
	\begin{equation} \label{eq12}
	    d^{[m-j]}\Delta_s f^{\tilde{N}} \cdot \tilde{\sigma}_s^{\tilde{N}} = \Delta \prod_{k=1}^{p^n-1} \Delta_{{s^k}^{[m-j]}} F_s^{N'N}.
	\end{equation}
	From equation (\ref{cor:eq11}) and since $d^{[j]}F_s^M = \Delta_{s^{[m]}}F^M$ we get that for every $(m-j)$-dimensional face $\beta$,
	\begin{equation}\label{eq13}
	d^{[m]}_\beta \Delta_s f \tilde{\sigma}_s^{\tilde{N}} = \Delta \Delta_{{s}_\beta^{[m]}} F^{\tilde{N}}.
	\end{equation} Since this is true for every $\beta$ we conclude that $d^{[m]}\Delta_s f^{\tilde{N}}\cdot \tilde{\sigma_s}^{\tilde{N}}= \Delta \Delta_{s^{[m]}} F^{\tilde{N}}$ which completes the proof of the lemma. It is left to prove the claim for this $j$.\\ Observe that $$d_\beta^{[m]} \Delta_s f = \Delta_{{s^{[m]}_\beta}} \Delta F = \Delta \Delta_{{s^{[m]}_\beta}} F \cdot V_{s_\beta^{[k+1]}}T_{g^{[k+1]}}\Delta_{[s^{-1},g^{-1}]_\beta^{[m]}}F.$$
	Plugging this above we get that $$\left(\Delta_{{[s^{-1},g^{-1}]}_\beta^{[m]}} F^{\tilde{N}}(y)\right) \left(d_{\beta}^{[m]} T_{g}^{-1}V_{s}^{-1}\sigma_s(g,y)^{\tilde{N}}\right)^{-1}$$ is invariant with respect to the diagonal action of $G$ on $Y^{[m]}$. Since $\beta$ is an $(m-j)$-dimensional face, the claim follows by Lemma \ref{exactdiff2:lem}.
	\end{proof}
We return to the proof of Lemma \ref{lemcomplete}. By what we just proved, there exists a power $N$, and a function $\sigma_{s,\chi}$ such that ${\theta'_{s,\chi}}^N\cdot \sigma_{s,\chi}$ is a coboundary. Since $\theta'_{s,\chi}$ is a cocycle of type $<k-j+1$, so is $\sigma_{s,\chi}$. As in the lemma above, $\sigma_{s,\chi}(g,\cdot)$ is trivial for any $g\in G_p^\perp$. Therefore, by the cocycle equation it is invariant under the action of $G_p^\perp$. Recall that the action of $G$ on $Y_{j,n+1}$ is ergodic and let $Y_{j,n+1}'$ be the factor of $Y_{j,n+1}$ which corresponds to the $\sigma$-algebra of the $G_p^\perp$-invariant functions. The induced action of $G_p$ on $Y_{j,n+1}'$ is therefore ergodic.\\ We consider this system as an $\mathbb{Z}/{p^d}\mathbb{Z}^\omega$-system for some fixed $d$. By the main theorem of Bergelson, Tao and Ziegler \cite{Berg& tao & ziegler}, any $\mathbb{Z}/{p^d}\mathbb{Z}^\omega$-cocycle is cohomologous to a phase polynomial of some bounded degree. By proposition \ref{PPC} there exists some power $p^n$ such that  $\sigma_{s,\chi}^{p^n}$ is a coboundary. Therefore ${\theta'_{s,\chi}}^{Np^n}$ is a coboundary. As $n=O_{k,m}(1)$, we conclude that the image of the minimal cocycle cohomologous to $\theta_s'$ takes values in a finite dimensional group of exponent $O_{k,m}(1)$. The proof of Lemma \ref{mainclaim} is now complete.
\end{proof}
It is left to prove Lemma \ref{above}. 
\begin{proof}	Let $s\in \mG(Y_{j+1})$ be any lift of an element $s'\in\mathcal{G}_j \cap \mathcal{L}_{j+1}$. By assumption, $[s'^{-1},g^{-1}]\in V_{j+1}$ for every $g\in G$. We conclude that there exists $\psi_{s,g}:Y_{j+1}\rightarrow U$ such that for every $(k-j+1)$-face $\beta$ we have,
$$\Delta_{[s^{-1},g^{-1}]_\beta^{[k+1]}} F\circ \pi = d_\beta ^{[k+1]}\psi_{s,g}$$
where $\pi:Y_{j+1}^{[k+1]}\rightarrow X$ is the factor map. Consider the function $\theta_s(g,y):= \psi_{s,g}(g\overline{s}y) \cdot \Delta_s \rho(g,\pi(y))$ where $\pi:Y_{j+1}\rightarrow X$ is the factor map. The following computation is taken from \cite[Proposition 10.10]{HK}:
\begin{equation}
\label{eqF}
\begin{split}
 \Delta_{g^{[k+1]}} (\Delta_{s_\beta^{[k+1]}}F(\pi(y))) =& V_{s_\beta^{[k+1]}}\circ T_{g^{[k+1]}}(\Delta_{[s^{-1},g^{-1}]_\beta^{[k+1]}} F(\pi(y))) \cdot \Delta_{s_\beta^{[k+1]}} \Delta_{g^{[k+1]}} F(\pi(y))\\
=&d_\beta^{[k+1]}V_s \circ T_g(\psi_{s,g}(y)) \cdot \Delta_{s_\beta^{[k+1]}} d^{[k+1]}\rho(g,\pi(y)) = d_\beta^{[k+1]}\theta_s
\end{split}
\end{equation}
where the last equality follows from the fact that $\Delta_{s_\beta^{[k+1]}} d^{[k+1]}\rho(g,\pi(y))=d_\beta^{[k+1]}\Delta_s \rho(g,\pi(y))$. It follows that $d_\beta^{[k+1]}\theta_s$ is a $(G,Y_{j+1}^{[k+1]},S^1)$-coboundary for every $(k-j+1)$-face $\beta$ and therefore $\theta_s$ is a function of type $<k-j+1$ (Lemma \ref{typereduce}). By Lemma \ref{line} we can extend $Y_j$ and assume that all $\theta_s$ are line-cocycles (the polynomial term $p$ in Lemma \ref{line} can be ignored by changing $\psi_{s,g}$ with $\psi_{s,g}/p$).\\

The map $s\mapsto \theta_s$ is a measurable map from $\mathcal{G}_j$ to functions of type $<k-j+1$. By Theorem \ref{countable} and Baire theorem we have for every $\chi$ a non-meagre measurable set $\mathcal{A}_\chi\subseteq \mathcal{G}_j$ such that for every $s,t\in\mathcal{A}_\chi$,  $\chi(\theta_s/\theta_t)$ is $(G,Y_{j+1},S^1)$-cohomologous to a phase polynomial of degree $<k-j+1$. Assume for now that we can choose the same set $\mathcal{A}$ for all $\chi\in\widehat U$ simultaneously. This assumption will be explained at the end of the proof. 

By Lemma \ref{fixCL:lem} we can find an $<O_{k,m}(1)$-extension $\tilde{Y}_{j+1}$ of $Y_{j+1}$ such that as a function on $\tilde{Y}_{j+1}$, $\theta_s/\theta_t$ is $(G,\tilde{Y}_{j+1},U)$-cohomologous to a phase polynomial of degree $<k-j+1$.   Therefore, for every $s,t\in \mG_j\cap \mL_j$ there exists a measurable function $\theta_{s,t}:\tilde{Y}_{j+1}\rightarrow U$ with $d_\beta^{[k+1]}\theta_s/\theta_t = \Delta d_\beta^{[k+1]}\theta_{s,t}$. Since $\tilde{Y}_{j+1}$ is a degenerate extension of $Y_{j+1}$, we can use Theorem \ref{redweaklymixing}. Thus, we can pass to an extension $Y_j'$ of order $<k+1$ such that $$\mathcal{L}_j' := \{s\in \mL_{j+1} : s \text{ has a lift in } \mG(Y_j')\}$$ is open. By lifting everything to $Y_j'$ it follows from (\ref{eqF}) that
$$\Delta (V_{\overline{s}_\beta^{[k+1]}} F\circ\pi^{Y_{j}'}_X - V_{\overline{t}_\beta^{[k+1]}} F \circ \pi^{Y_{j}'}_X) = \Delta d_\beta^{[k+1]}\theta_{s,t}$$
where $F$ and $\theta_{s,t}$ are viewed as functions on $Y_j'$ and $\overline{s}$, $\overline{t}$ are any lifts of $s$ and $t$. It follows that that $\Delta_{\overline{s}_\beta^{[k+1]}} F\circ \pi^{Y_{j}'}_X - V_{\overline{t}_\beta^{[k+1]}} F\circ \pi^{Y_{j}'}_X - d_\beta^{[k+1]}\theta$ is invariant in $(Y_j')^{[k+1]}$. Since $t\in\mathcal{G}_{j}$, it maps the $\sigma$-algebra $\mathcal{I}_{k+1}(X)$ to itself. Moreover, since $F\circ \pi^{Y_{j}'}_X$ is measurable with respect to $X$, we have that $\Delta_{{\overline{s}\overline{t}^{-1}_\beta}^{[k+1]}} F\circ \pi^{Y_{j}'}_X - \Delta d_\beta^{[k+1]}V_{\overline{t}^{-1}}\theta_{s,t}$ is invariant with respect to the diagonal action of $G$ on $X^{[k+1]}$. Now, by Lemma \ref{CLproperty:lem}, we conclude that for every lifts of $s,t\in\mathcal{A}$ the element corresponding to $st^{-1}$ in $\mathcal{G}(Y_j')$ is in $\mathcal{G}(Y_{j}')^\star$, where $\mathcal{G}(Y_j')^\star$ is defined with respect to the extension $Y_j'\times_{\rho\circ \pi^{Y_{j}'}_X} U$. Thus, $V_j'$ contains $\mathcal{A}\cdot \mathcal{A}^{-1}$ and so the proof is complete by Lemma \ref{Pettis}.

It is left to establish the assumption above about the existence of a measurable set $\mathcal{A}$ of positive measure, which satisfies that $\chi\circ\theta_s/\chi\circ\theta_t$ is cohomologous to a polynomial of degree $<k-j+1$ for every $s,t\in \mathcal{A}$ and every $\chi\in\hat U$. Let $\Delta\leq U$ be a subgroup of bounded exponent such that $U/\Delta$ is a Lie group. Let $\chi_1,...,\chi_n\in\widehat U$ be a lift of a basis of the dual of $U/\Delta$ and $\pi_1,\pi_2,...$ a lift of the coordinate maps in the dual of $\Delta$. Since $\{\chi_1,...,\chi_n\}$ is a finite set of characters we can apply Theorem \ref{HK1functions} and find a set $\mathcal{A}$ of positive measure such that for every $1\leq i\leq n$, $\chi_i(\theta_s/\theta_t)$ is cohomologous to a phase polynomial of degree $<k-j+1$, simultaneously.\\
    
    We also notice, as in the proof of Lemma \ref{above} that $\chi(\theta_s/\theta_t)$ is cohomologous to a phase polynomial of degree $<k-j+1$ if and only if $\Delta_{(st^{-1})_\beta^{[m]}} \chi\circ F = d_\beta^{[m]} \phi $ for some $\phi:Y_{j+1}\rightarrow S^1$. Let $\mathcal{H}$ be the subgroup generated by $\mathcal{A}\cdot\mathcal{A}^{-1}$, then for every $1\leq i \leq n$, $h\in H$ and an $(k-j+1)$-dimensional face $\beta$ there exists $\phi_{h,i}$ such that 
    $$\Delta_{h_\beta^{[m]}} \chi_i\circ F = d_\beta^{[m]}\phi_{h,i}$$
    Now, let $\pi$ be one of the maps $\pi_1,\pi_2,\ldots$. Then, by Theorem \ref{HK1functions} we can find a set of positive measure $\mathcal{A}_\pi\subseteq \mathcal{A}$ such that $\pi(\theta_s/\theta_t)$ is cohomologous to a phase polynomial of degree $<k-j+1$. As before, let $\mathcal{H}_\pi$ be the group generated by $\mathcal{A}_\pi\cdot\mathcal{A}_\pi^{-1}$. We conclude that for every $h\in\mathcal{H}_\pi$ we have that
    $$\Delta_{h_\beta^{[k+1]}} \pi\circ F = d_\beta^{[k+1]}\phi_{h,\pi}.$$ In particular, we see that for every $s\in\mathcal{H}_\pi$, $\pi\circ\theta_s$ is cohomologous to a phase polynomial of degree $<k-j+1$.
    We want to extend $\mathcal{H}_\pi$ to $\mathcal{H}$. For every $i$, $\mathcal{H}_{\pi_i}$ is an open subgroup of $\mathcal{H}$. We conclude that the index $[\mathcal{H}:\mathcal{H}_{\pi_i}]$ is a most countable. Thus, for each $\pi_i\in\{\pi_i:i\in\mathbb{N}\}$ we find a set of countably many transformations $\{s_{n,\pi_i} : n\in\mathbb{N}\}$ such that $\bigcup_{n\in\mathbb{N}}s_{n,\pi_i}\mathcal{H}_{\pi_i}=\mathcal{H}$. It is left to show that, $\Delta_{s_{n,\pi_i}}^{[k+1]}\pi_i\circ F = d_\beta^{[k+1]}\phi_{n,\pi_i}$ for some $\phi_{n,\pi_i}:Y_{j+1}\rightarrow S^1$. Indeed, in this case we have that $s\in\mathcal{H}$ and $i\in\mathbb{N}$, the cocycle $\pi_i(\theta_s)$ is cohomologous to a phase polynomial of degree $<k-j+1$. In particular, we can take $\mathcal{A}=\mathcal{H}$ and the proof is complete.
    
  Since the set $\{s_{n,\pi_i} : n\in\mathbb{N}\}$ is countable, we can use the same argument as in Lemma \ref{mainlem} with one minor modification. This time the elements $s_{n,\pi_i}$ are not of finite order (but the commutators are). Therefore in the last step we can not use Lemma \ref{powers} in order to deduce that $W$ is zero dimensional. Instead, recall that for each $\pi_i$ there exists a constant $m_i$ such that $\pi_i^{m_i}\in \left<\chi_1,...,\chi_n\right>$. This means, in particular, that some power $d=O_{k,m}(1)$ of $\pi_i (\theta_s)$ is $(G,Y_{j+1},S^1)$-cohomologous to a phase polynomial $p_{s,i}$ of degree $<k-j+1$. As in the claim in Lemma \ref{mainlem} we can find a constant $c_{s,i}$ and a cocycle $\theta_{s,i}' = \pi_i (\theta_s)\cdot c_{s,i}$. It follows that $\theta_{s,i}'^d$ is a phase polynomial of degree $<k-j+1$. By Lemma  \ref{extensionlowchar:thm} we can also find a phase polynomial cocycle $q_{s,i}$ of degree $<k-j+1$ such that $q_{s,i}^d = p_{s,i}$ (by passing to an extension). We conclude that $\theta_{s,i}'/q_{s,i}$ is a cocycle, and the $d$-th power of this cocycle is a coboundary. As in Lemma \ref{mainlem}, by extending with a minimal cocycle which is cohomologous to $\theta'_{s,i}/q_{i,s}$ we can assume that the latter is a coboundary. Therefore, $\pi_i\circ \theta_{s}$ is cohomologous to a phase polynomial of degree $<k-j+1$ for all $s\in\mathcal{H}$ which completes the proof. 
\end{proof}
\subsection{Concluding everything}
To finish the proof of Theorem \ref{Main2:thm} we need to following variant of a theorem by Furstenberg and Weiss \cite{F&W}.
\begin{lem} \label{liftext:lem} Let $X$ and $Y$ be ergodic $G$ systems and $\pi:Y\rightarrow X$ be the factor map. Let $\rho:G\times X\rightarrow U$ be a cocycle and suppose that $X'=X\times_\rho U$ is ergodic. If $\sigma$ is the minimal cocycle cohomologous to $\rho\circ \pi:G\times Y\rightarrow U$ and $V$ is the image of $\sigma$, then $X'$ is a factor of $Y\times_\sigma V$.
\end{lem}
\begin{proof}
	Consider the (possibly non-ergodic) system $Y\times_{\rho\circ\pi} U$. It follows by the theory of Mackey (see \cite[Proposition 7.1]{F&W}) that every ergodic component of $Y\times_{\rho\circ\pi} U$ is isomorphic to $Y\times_\sigma V$ for some $V\leq U$. Choose any ergodic invariant measure $\mu_{Y'}$ on $Y\times_{\rho\circ\pi} U$. It is easy to see that the push-forward of $\mu_{Y'}$ to $Y$ is $\mu_{Y}$. Moreover, since $X$ is a factor of $Y$ we conclude that the push-forward of $\mu_{Y'}$ to $X$ is $\mu_X$. Let $\mu_{X'}$ be the push-forward of $\mu_{Y'}$ to $X'$. Since $X'$ is ergodic $\mu_{X'}$ must be the product measure $\mu_X\times m_U$ where $m_U$ is the Haar measure on $U$ (see \cite[Section 2.2, Lemma 4]{HKbook}). In other words, $X'$ is a factor of $Y\times_\sigma V$ as required.
\end{proof}
Given an ergodic $G$-system $X$, by Lemma \ref{InvFD:thm} it is an inverse limit of finite dimensional systems $X=\underset{\longleftarrow}{\lim} X_n$. By Theorem \ref{InvFDr1:thm}, we can find a constant $l=O_k(1)$ and for each $(X_n,G^{(l)})$ we can find an extension $(Y_n,G^{(l)})$ which is an finite dimensional nilsystem. By increasing $Y_n$ we may assume that $Y_{n-1}$ is a factor of $Y_n$. More concretely, in the proof of Theorem \ref{mainclaim} we build $Y_n$ as a sequence of extensions of $X_n$ (by zero dimensional groups). In each step, instead of extending by the groups associated to $X_n$ we can also extend by the groups associated to the previous systems $X_{n-1},...,X_1$ (and replace with minimal cocycles, as in Lemma \ref{liftext:lem}). In this case $Y:=\underset{\longleftarrow}{\lim} Y_n$ is an inverse limit of nilsystems. It is standard that $(X,G^{(l)})$ is a factor of $(Y,G^{(l)})$ or equivalently that $Y$ is an $l$-extension of $X$ as required.\footnote{Another approach would be to use a version of Lemma A.4 from \cite{FranHost} and replace $Y_n$ with an ergodic joining of $Y_n,Y_{n-1},....,Y_1$.}
\section{Proving the identification $\mG(X)/\Gamma\cong X$}
\label{identification:sec}
The goal of this section is to deduce Theorem \ref{InvFDr1:thm} from Theorem \ref{mainclaim} and thus complete the proof of Theorem \ref{Main2:thm}. Given a finite dimensional system $X$ we have already established the existence of a finite dimensional extension $Y$ which is an inverse limit of systems $Y_n$ where the action of $\mG(Y_n)$ on $Y_n$ is transitive. It is therefore enough to derive the identification $Y_n\cong \mG(Y_n)/\Lambda_n$ for some totally disconnected closed subgroup $\Lambda_n$ of $\mG(Y_n)$. In other words it suffices to prove the following theorem.
\begin{thm}
    Let $X$ be an ergodic $G$-system of order $k$ and suppose that $\mG(X)$ acts transitively on $X$ (as a near-action), then there exists a totally disconnected subgroup $\Lambda\leq \mG(X)$ so that $X$ and $\mG(X)/\Lambda$ are isomorphic as $\mathcal{G}(X)$-systems.
\end{thm}
In order to prove this theorem we construct a topological model for $X$. That is a compact Hausdorff space $\widehat X$ with a continuous action $\widehat T:\mathcal{G}(X)\times\widehat X\rightarrow \widehat X$ such that $X$ and $\widehat X$ are isomorphic as measure spaces. Write $X=Z_k(X)\times_\rho U$, by induction hypothesis we may write $Z_{<k}(X)= \mathcal{L}/\Lambda_{\mathcal{L}}$ and we have a projection map $p:\mathcal{G}(X)\rightarrow \mathcal{L}$ which is onto.
\begin{defn}
Let $H$ be a Polish group with a near-action on $X$. A function $f\in L^\infty(X)$ is said to be $H$-continuous if $f(hx)\rightarrow f(x)$ in $L^\infty(X)$ as $h\rightarrow 1_H$.
\end{defn}
\begin{prop}\label{dense}
Let $\mathcal{A}\subseteq L^\infty(X)$ denote the algebra of $\mathcal{G}(X)$-continuous functions. We claim that the unit ball of this algebra is dense in the unit ball of $L^\infty(X)$ with respect to the $L^2$-topology.
\end{prop}
The following result of Gleason \cite[Theorem 3.3]{Gleason} will be used to lift the $\mathcal{L}$-continuity to a $\mathcal{G}(X)$-continuity.
\begin{thm}\label{gleason}
    Let $U$ be a compact Lie group acting freely and continuously on a completely regular topological space $X$. Let $q:X\rightarrow X/U$ be the quotient map where $x\sim_U y$ if there exists $u\in U$ so that $ux=y$ and $X/U$ is equipped with the quotient topology. Then every point $x\in X/U$ has an open neighborhood $x\in V\subseteq U/X$ such that there is a local continuous section $s:V\rightarrow X$ so that $p\circ s = Id_V$.
\end{thm}
\begin{proof}[Proof of Proposition \ref{dense}]
Let $f$ be a continuous function on $X$ with $\|f\|_\infty \leq 1$ and let $\varepsilon>0$ be arbitrary. Recall that $X=\mathcal{L}/\Lambda_{\mathcal{L}}\times U$. Since $f$ is uniformly continuous, we can find an open subset $U'\leq U$ so that $\|f(ux)-f(x)\|_{\infty}<\varepsilon/2$ for all $u\in U'$. By Gleason-Yamabe, we can find a subgroup $J\leq U$ so that $J\subseteq U'$ and $U/J$ is a Lie group. Let $\tilde{f}(x) = \int_J f(jx) dj$ where $dj$ is the Haar measure on $J$. We see that $\tilde{f}$ is $J$-invariant and \begin{equation}\label{ftildef}\|\tilde{f}(x)-f(x)\|<\varepsilon/2
\end{equation} Recall that we have a surjective homomorphism $p:\mathcal{G}(X)\rightarrow\mathcal{L}$. By Lemma \ref{CLproperty:lem} we can identify the kernel of $p$ with $P_{<k+1}(Y,U)$, where $Y:=Z_{k-1}(X)=\mL/\Lambda_{\mL}$. Quotienting by $P_{<k+1}(Y,J)$ we get the following short exact sequence
\begin{equation}\label{exact}1\rightarrow P_{<k+1}(Y,U)/P_{<k+1}(Y,J)\rightarrow \mathcal{G}(X)/P_{<k+1}(Y,J)\rightarrow \mathcal{L}\rightarrow 1.\end{equation}
Since $U/J$ is a Lie group, we deduce by Lemma \ref{sep:lem} that so is $K:=P_{<k+1}(Y,U)/P_{<k+1}(Y,J)$. In particular, $K$ is locally compact and admits a Haar measure $dK$. Observe that $\tilde{f}$ is invariant to translations by $P_{<k+1}(Y,J)$, hence for every continuous function $\phi$ on $K$ the convolution
$$\tilde{f}\ast \phi(x) := \int_K \tilde{f}(kx)\phi(k) dk$$ is well defined on $X/J := Y\times U/J$ and $K$-continuous. Letting $\phi$ be a suitable approximation to the identity (non-negative, supported on a small neighborhood of
the identity, and of total mass one) we deduce that \begin{equation}\label{tildeftildefphi}\|\tilde{f}-\tilde{f}\ast\phi\|_{L^2(X/J)}<\varepsilon/2.
\end{equation}
Moreover, $\tilde{f}\ast\phi$ is a $K$-continuous function and an $\mathcal{L}$-continuous function. 

We use Gleason theorem and \eqref{exact} to show that $\tilde{f}\ast \phi$ is $\mathcal{G}(X)/P_{<k+1}(Y,J)$-continuous. By Lemma \ref{sep:lem}, $K = U/J\times D$, where $D$ is a countable discrete group. Fix a complete metric $d$ on $U$ which induces its topology and normalize so that $d(u,v)\leq 1$ for all $u,v\in U$. For every two functions $f$, $f'$ taking values in $U$ we define $\|f-f'\|_\infty = \sup d(f(x),f'(x))$. We make an observation that $D$ remains discrete even when $\mathcal{G}(X)$ is equipped with the following finer metric\footnote{We note that $\mG(X)$ is no longer a topological group with respect to that metric.} $d(S_{l,f},S_{l',f'}) = d_{\mathcal{L}}(l,l')+\|f-f'\|_\infty$. Moreover, the action of $K$ on $\mG(X)$ is still continuous. 

Let $B\subseteq \mG(X)$ be an open neighborhood of the identity so that $B\cap D=\{1\}$. Identify $U/J$ with the subgroup $U/J\times \{1\}\subseteq K$ and replacing $B$ with $U/J\cdot B$ we can assume that $K\cap B = U/J$. The Lie group $U/J$ then acts continuously on $B$. Moreover, any element in the quotient $B/(U/J)$ can be identified with a unique element of $\mathcal{L}$. By Gleason theorem, we can find an open neighborhood $V$ of the identity in $\mathcal{L}$ and a local continuous section $l\mapsto S_{l,\phi_l}\in B \subseteq\mathcal{G}(X)$, where the latter is equipped with the finer metric introduced above. Since $\tilde{f}$ is continuous on $X/J$, so is $\tilde{f}\ast \phi$. Without loss of generality we can take $\phi_1 = 1$ by replacing $\phi_l$ with $\phi_l/\phi_1$ for all $l\in V$.

We now prove that $\tilde{f}\ast\phi$ is $\mG(X)/P_{<k+1}(Y,J)$-continuous. Since $\tilde{f}\ast\phi$ is continuous on $X$, it suffices to show that $g\mapsto \tilde{f}\ast\phi(gx)$ is continuous at $g=1$. Let $\varepsilon'>0$, by uniform continuity, there exists $\delta>0$ sufficiently small so that \begin{align}\label{uniformlycont}\|\tilde{f}\ast \phi(x)-\tilde{f}\ast\phi(y)\|_{L^\infty}<\varepsilon'
\intertext{for all $|x-y|<\delta$.}
\end{align}  By shrinking $V$, we can assume that \begin{equation}\label{philsmall}\|\phi_l-1\|_\infty<\delta/2\end{equation} for all $l\in V$. We need to find $\delta'>0$ so that if $S_{l,\phi}\in \mathcal{G}(X)/P_{<k+1}(Y,J)$ satisfies $d(S_{l,\phi},1)<\delta'$, then $\|S_{l,\phi_l}\tilde{f}\ast \phi - \tilde{f}\ast \phi\|_{L^\infty}<\varepsilon''$, where here $\mathcal{G}(X)/P_{<k+1}(Y,J)$ is equipped with the usual quotient metric.

By taking $\delta'>0$ sufficiently small, we can guarantee that $l\in V$. In that case we must have that $\phi = \phi_l \cdot p_l$ for some phase polynomial $p_l:Y\rightarrow U/J$ of degree $k+1$. By \eqref{philsmall} and the triangle inequality we deduce that $\|p_l-1\|_2 \leq \delta/2 + \delta'$. Choosing $\delta'<\frac{1}{6}\delta$ and then choosing $\delta$ sufficiently small we can guarantee by Lemma \ref{sep:lem} that $p_l$ is a constant and $|p_l-1|<2/3\delta$. By the triangle inequality $$\|\phi-1\|_\infty = \|\phi_l\cdot p_l-1\|_\infty \leq \|\phi_l-1\| + |p_l-1|<2/3\delta+\delta'<\delta.$$ Equation \eqref{uniformlycont} now gives  $\|S_{l,\phi}\tilde{f}\ast\phi - \tilde{f}\ast\phi\|_{L^\infty}<\varepsilon'$, as required. We deduce that $\tilde{f}\ast\phi$ is $\mathcal{G}(X)/P_{<k+1}(Y,J)$-continuous. We now combine everything to deduce that $f$ can be approximated by a $\mathcal{G}(X)$ continuous function. Let $f'$ denote the lift of $\tilde{f}\ast\phi$ to $X$ under the canonical projection map $X\rightarrow X/J$. It follows that $f'$ is $\mathcal{G}(X)$-continuous. Moreover, by \eqref{ftildef}, \eqref{tildeftildefphi} and the triangle inequality it follows that $\|f - \tilde{f}\ast\psi \|_{L^2(X)}<\varepsilon$. Since $\varepsilon>0$ is arbitrary, we deduce that the $\mathcal{G}(X)$-continuous functions are dense in the unit ball of $C(X)$, and therefore also in the unit ball of $L^\infty(X)$. This completes the proof.
\end{proof}
We can now construct $\widehat X$. Applying the Gelfand-Riesz theorem and letting $\widehat X$ be the spectrum of $\mathcal{A}$, we see that $\widehat X$ is a compact Hausdorff topological space satisfying that $C(\widehat X)$ is isomorphic as a $C^\star$-algebra to $\mathcal{A}$. The following properties of $\widehat X$ were established in \cite{JST}.
\begin{lem}\label{properties}
Let $\widehat X$ as above. Then
\begin{itemize}
    \item[(i)] There exists a Radon measure on $\widehat X$. In particular, every open subset in $\widehat X$ has positive measure.
    \item[(ii)] The natural action $\widehat T:\mathcal{G}(X)\times \widehat X\rightarrow \widehat X$ is jointly continuous in $\mathcal{G}(X)$ and $\widehat X$.
    \item[(iii)] Every $G$-continuous function $f\in \mathcal{A}$ has a unique continuous representative $\widehat f$ on $\widehat X$.
\end{itemize}
\end{lem}
From property (ii) and Theorem \ref{Effros} we see that in order to show that $X\cong \mG(X)/\Lambda$ is suffices to show that $\mG(X)$ acts transitively on $\widehat X$ and the stabilizer is totally disconnected. We prove these in two separate lemmas.
\begin{lem}
The action of $\mG(X)$ on $\widehat X$ is transitive.
\end{lem}
\begin{proof}
Observe that any continuous function $f\in C(\mathcal{L}/\Lambda_{\mathcal{L}})$ gives rise to a $\mG(X)$-continuous function on $X=\mathcal{L}/\Lambda_{\mathcal{L}}\times U$ by $(x,u)\mapsto f(x)$. This gives a $C^\star$-algebra homomorphism from the continuous functions on $\mathcal{L}/\Lambda_{\mathcal{L}}$ to $\mathcal{A}$ which by Gelfand-Riesz gives rise to a continuous factor map $\pi:\widehat X\rightarrow \mathcal{L}/\Lambda_{\mathcal{L}} $ of $\mG(X)$-systems, where the group $\mG(X)$ acts on $\mathcal{L}/\Lambda_{\mathcal{L}}$ through the projection $p:\mG(X)\rightarrow\mathcal{L}$. Since $p$ is surjective, the action of $\mG(X)$ on $\mathcal{L}/\Lambda_{\mathcal{L}}$ is transitive. Thus it suffices to show that for every $x_1,x_2\in \widehat X$ with $\pi(x_1)=\pi(x_2)$ we can find $g\in \mG(X)$ with $x_1=gx_2$. We follow the argument from \cite[\textsection 19.3.3 Lemma 10]{HKbook}: if not, then by continuity we can find an open neighborhood $V\subseteq \widehat X$ of $x_1$ which contains no element in the orbit of $x_2$ with respect to the action of $\mG(X)$. By Uryshon's lemma we can find a non-negative continuous function $f:\widehat X \rightarrow \mathbb{R}$ which is supported on $V$ with $f(x_1)>0$. Recall that all translations by $u\in U$ belong to $\mG(X)$ and let $\tilde{f}(x) = \int_{U} \tilde{f}(ux) du$ where $du$ is the Haar measure on $U$. Then $\tilde{f}$ is a $U$-invariant function on $\widehat X$ satisfying $\tilde{f}(x_1)>0$ and $\tilde{f}(x_2)=0$. By property (iii) of the Lemma above we can identify $\tilde{f}$ with a $\mathcal{G}(X)$-continuous function which is also $U$-invariant, which is therefore identified with a continuous function $f'$ on $\mathcal{L}/\Lambda_{\mathcal{L}}$ and $\tilde{f}=f'\circ \pi$. This gives a contradiction as $x_1,x_2$ lie in the same fiber of $\pi$.
\end{proof}
\begin{lem}
    The stabilizer of some point in $\widehat X$ under the action of $\mG(X)$ is totally disconnected.
\end{lem}
\begin{proof}
Recall that for every $u\in U$, the translations $V_u$ are automorphisms which belong in $\mG(X)$. It is easy to see that the action of $U$ on $\widehat X$ is free. Indeed, if not then there exists $x\in \widehat X$ so that $ux=x$. But since $u$ commutes with any $g\in \mG(X)$ we see that $u$ stabilizes the orbit of $x$, which by the previous lemma is everything. Now let $\pi:\widehat X\rightarrow \mathcal{L}/\Lambda_{\mathcal{L}}$ as before and let $x_0$ be any element in the pre-image of the coset $1\cdot\Lambda_{\mathcal{L}}$. Let $\Lambda$ be the stabilizer of $x_0$. By construction we have $p(\Lambda) = \Lambda_{\mathcal{L}}$ and the kernel is isomorphic to a subgroup of $P_k(Z_k(X),U)$. Since the action of $U\leq P_k(Z_k(X),U)$ is free we have that $U\cap \Lambda = \{e\}$. Thus, we can identify the kernel of $p$ with a closed subgroup of $P_k(Z_k(X),U)/P_1(Z_k(X),U)$ which is totally disconnected (by Pontryagin duality  it is embedded as a closed subgroup in the direct product of countably many copies of $P_k(Z_k(X),S^1)/P_1(Z_k(X),S^1)$ which by Lemma \ref{sep:lem} is a product of discrete groups).
\end{proof}
By theorem \ref{Effros} we deduce that $\widehat X \cong \mathcal{G}(X)/\Lambda$, where $\Lambda$ is the stabilizer of some $x_0\in \widehat X$. Since $X$ is isomorphic to $\widehat X$ as $\mathcal{G}(X)$-systems and $T_g\in\mathcal{G}(X)$ for every $g\in G$, they are also isomorphic as $G$-systems. This completes the proof.
\section{Limit formula and convergence result} \label{characteristic}
In this section we deduce the convergence result (Theorem \ref{convergence}) and the limit formula (Theorem \ref{formula}). We begin with the following proposition of Bergelson, Tao and Ziegler \cite[Theorem 3.2]{BTZ} generalized for $\bigoplus_{p\in P}\mathbb{Z}/p\mathbb{Z}$-systems.
	\begin{lem} [Characteristic factors] \label{char}
		Let $X$ be an ergodic $\bigoplus_{p\in P}\mathbb{Z}/p\mathbb{Z}$ system and $f_1,f_2,...,f_{k+1}\in L^\infty(X)$. If for some $i$, we have that $E(f_i|Z_{<k+1}(X))=0$ then,
		$$\limsup_{N\rightarrow\infty} \|\mathbb{E}_{g\in\Phi_N} T_g f_1 T_{2g} f_2\cdot...\cdot T_{(k+1)g} f_{k+1}\|_{L^2}=0. $$
	\end{lem}
The proof is the same as in \cite{BTZ} and therefore is omitted. We deduce that in order to prove Theorem \ref{convergence}, it is enough to prove Theorem \ref{formula}.
\begin{prop}
Theorem \ref{convergence} follows from Theorem \ref{formula}.
\end{prop}
\begin{proof}
We denote by $\tilde{f}$ the projection of $f$ to $L^2(Z_{<k-1}(X))$. By Lemma \ref{char}, the limit of the average (\ref{average}) exists if and only if it exists for $\tilde{f}_1$,...,$\tilde{f}_{k+1}$. Now let $(Y,G^{(m)})$ be as in Theorem \ref{Main2:thm} and let $h_1,...,h_{k+1}$ be the lifts of $\tilde{f}_1,...,\tilde{f}_{k+1}$ to $Y$ respectively. Note that for any F{\o}lner sequence $\Phi_N$ of $G$, there exists a F{\o}lner sequence $\tilde{\Phi}_N$ for $G^{(m)}$ such that for every $y\in Y$, 
\begin{equation} \label{folner}
\mathbb{E}_{g\in \tilde{\Phi}_N} T_gh_1(y)\cdot...\cdot T_{(k+1)g}h_{k+1}(y) = \mathbb{E}_{g\in \Phi_N} T_g\tilde{f}_1(\pi(y))\cdot...\cdot T_{(k+1)g}\tilde{f}_{k+1}(\pi(y)).
\end{equation}
Therefore, since $Y$ is an inverse limit of $k$-step nilpotent systems we can approximate $h_1,...,h_{k+1}$ in $L^2$ by bounded functions $h_{1,n},...,h_{k+1,n}$ such that for every $1\leq i \leq k+1$ and $n\in\mathbb{N}$, $h_{i,n}$ is measurable with respect to the $k$-step nilpotent system $Y_n$. The dominated convergence theorem implies that the pointwise convergence in Theorem \ref{formula} is also an $L^2$ convergence. Since $L^2(X)$ is a complete metric space, we conclude that the average associated with $h_1,...,h_{k+1}$ converges. By (\ref{folner}), we conclude that the average associated with $\tilde{f}_1,...,\tilde{f}_{k+1}$ also exists, as required.
\end{proof}

We prove Theorem \ref{formula} and the following theorem simultaneously by induction on $k$.
\begin{thm}	\label{ufc}
	Let $X=\mG/\Gamma$ be as in Theorem \ref{formula}. Then, for every $1\leq r \leq k+1$, $Z_{<r}(X)\cong \mG/\mG_r \Gamma$.
\end{thm}
Let $k=0$. Then $Z_{<1}(X)$ is trivial and the claim in Theorem \ref{ufc} follows. As for Theorem \ref{formula}, the case $k=0$ follows by the pointwise mean ergodic theorem and Lemma \ref{open=full}. Fix $k\geq1$. Throughout the rest of this section we assume that Theorem \ref{ufc} and Theorem \ref{formula} hold for all smaller values of $k$. We prove Theorem \ref{ufc} for this value of $k$ and then, we deduce Theorem \ref{formula} from this result.

\begin{claim} [The induction hypothesis] \label{induction} Let $k\geq 1$ be such that Theorem \ref{formula} and Theorem \ref{ufc} holds for all smaller values of $k$. Let $X$ be as in Theorem \ref{formula}, then for every $1\leq r \leq k$ and every $f_1,...,f_{r}\in L^\infty(X)$ the following $r$-term formula holds.
	\begin{equation} \label{eqr1}
	\begin{split}
	\lim_{N\rightarrow\infty}\mathbb{E}_{g\in\Phi_N} T_g f_1(x) T_{2g} f_2(x)\cdot...\cdot T_{rg} f_r(x)&=\\ \int_{\mG/\Gamma} \int_{\mG_2/\Gamma_2}...\int_{\mG_r/\Gamma_r}  \prod_{i=1}^{r}f_i(x\cdot y_1^i\cdot y_2^{\binom{i}{2}}\cdot ...\cdot & y_i^{\binom{i}{i}})d\prod_{i=1}^r m_i(y_i\Gamma_i)
	\end{split}
	\end{equation}
	with the abuse of notation that $f(x)=f(x\Gamma)$.
\end{claim}
\begin{proof}
	Let $\tilde{f}_1,...,\tilde{f}_r$ be the projections of $f_1,...,f_r$ respectively into $L^2(Z_{<r}(X))$. By the induction hypothesis of Theorem \ref{formula} and Theorem \ref{ufc}, we have that 
	\begin{equation} \label{eqr}
	\begin{split}\lim_{N\rightarrow\infty}\mathbb{E}_{g\in\Phi_N} T_g \tilde{f}_1(x) T_{2g} \tilde{f}_2(x)\cdot...\cdot T_{rg} \tilde{f}_r(x)&=\\\int_{\mG/\mG_r\Gamma}\int_{\mG_2/\mG_r\Gamma_2}...\int_{\mG_{r-1}/\mG_r} \prod_{i=1}^{r}\tilde{f}_i(x\cdot y_1^i\cdot y_2^{\binom{i}{2}}\cdot ...\cdot & y_i^{\binom{i}{i}})d\prod_{i=1}^r m_i(y_i\Gamma_i).
	\end{split}
	\end{equation}
	
We lift each $\tilde{f}_i$ to $X$. Then, equation (\ref{eqr}) remains unchanged and by Lemma \ref{char} and the fact that each lift is invariant to $\mG_r$, we get that equation (\ref{eqr1}) holds. As required.
\end{proof}

\subsection{Proof of Theorem \ref{ufc} and corollaries}
We recall that the Host-Kra group induces an action on each of the universal characteristic factors. 
\begin{lem} [$\mG$ induces an action on the universal characteristic factors] \label{induce}
	Let $X$ be an ergodic $G$-system of order $<k$ and $\mG(X)$ be the Host-Kra group. Then for every $1\leq l<k$ there exists a projection $p_l:\mG(X)\rightarrow \mG(Z_{<l}(X))$ where $\mG(Z_{<l}(X))$ is the Host-Kra group of the factor $Z_{<l}(X)$.
\end{lem}
The proof of Theorem \ref{ufc} is a modification of the argument of Ziegler from \cite[Lemma 4.5]{Z}.
\begin{proof} [Proof of Theorem \ref{ufc}]
Let $X=Z_{<k}(X)=\mG/\Gamma$ be as in Theorem \ref{formula}. Let $1\leq r \leq k-1$ and consider the factor map $\pi:\mG/\Gamma \rightarrow Z_{<r}(X)$. By Lemma \ref{induce} the action of $\mG_{r}$ on $Z_{<r}(X)$ is trivial and so $\pi$ induces a factor $\pi_r : \mG/\mG_r\Gamma\rightarrow Z_{<r}(X)$. We prove that $\pi_r$ an isomorphism. Let $f:\mG/\mG_r\Gamma\rightarrow S^1$. Then, since the coset $g\mG_r$ is uniquely determined by the cosets associated with $gy_1,gy_1^2y_2,...,gy_1^{r}y_2^{\binom{r}{2}}...y_{r-1}^{\binom{r}{r-1}}$, we conclude that there is a measurable map $F:(\mG/\mG_r\Gamma)^r\rightarrow S^1$ with $F(gy_1,gy_1^2y_2,...,gy_1^{r}y_2^{\binom{r}{2}}...y_{r-1}^{\binom{r}{r-1}}) = f(g\Gamma)$ for almost all $g\in \mG$, $y_1\in \mG_1$, $y_2\in \mG_2$,..., $y_{r-1}\in \mG_{r-1}$. We conclude that
$$f(g\Gamma) = \int_{\mG/\Gamma}\int_{\mG_2/\Gamma_2}...\int_{\mG_k/\Gamma_k} F(gy_1,gy_1^2y_2,...,gy_1^{r}y_2^{\binom{r}{2}}...y_{r-1}^{\binom{r}{r-1}}) d\prod_{i=1}^r m_i(y_i\Gamma_i).$$
By approximating $F$ with functions of the form $(x_1,...,x_r)\mapsto f_1(x_1)\cdot f_2(x_2)\cdot...\cdot f_r(x_r)$, it follows from Claim \ref{induction} that $F$ is spanned by limits of the form $$\lim_{N\rightarrow\infty}\mathbb{E}_{g\in\Phi_N}  T_g f_1\cdot...\cdot T_{rg}f_r.$$ By Proposition \ref{char}, these terms are measurable with respect to $Z_{<r}(X)$. This completes the proof.
\end{proof}
Let $X$ be an ergodic $G$-system. It is well known \cite[Proposition 4.11]{HK} that every factor of $X$ of order $<k$ factors through $Z_{<k}(X)$. We refer to this fact as the maximal property of the $k$-th universal characteristic factor. We have the following result.
\begin{lem} \label{structuregroups}
	Let $X$ be a $(k+1)$-step nilpotent system and let $\mG$ be an open subgroup of the Host-Kra group of $X$ and $\Gamma$ be the stabilizer of $(1,1,...,1)\in U_0\times U_1\times...\times U_{k-1}$ where $U_0,U_1,...,U_{k-1}$ are the structure groups of $X$. Then, for every $0\leq i \leq k-1$ we have that $U_i \cong \mG_i/\mG_{i+1}\Gamma_i$ as topological groups and measure spaces, where $\mG_{i+1}\Gamma_i$ is the closed subgroup generated by all the products of elements in $\mG_{i+1}$ and $\Gamma_i$.
\end{lem}
\begin{proof}
	By Lemma \ref{open=full}, we have that $X=\mG/\Gamma$ and by Theorem \ref{ufc} that $Z_{<r}(X)=\mG/\mG_r\Gamma$ for every $1\leq r\leq k$. Let $A_r:=\mG_r/\mG_{r+1}\Gamma_r$. Then, by Lemma \ref{induce}, $A_r$ acts on $Z_{<r+1}(X)$, fixes $Z_{<r}(X)$ and commutes with $\mG(Z_{<r+1}(X))$. It follows that the action of any $t\in A_r$ equals to a translations by an element in $U_r$. In particular, this means that we can identify $A_r$ with a closed subgroup of $U_r$ (as topological groups and measure spaces). Therefore, $Z_{<r+1}(X)/A_r$ is an extension of $Z_{<r}(X)$ by $U_r/A_r$. On the other hand $Z_{<r+1}(X)/A_r = \mG/\mG_r\Gamma\cong Z_{<r}(X)$ is a system of order $<r$. By the maximal property of $Z_{<r}(X)$ it follows that $U_r/ A_r$ is trivial, as required.
\end{proof}
As a corollary we have the following result.
\begin{cor} \label{G=L}
	Let $X$ be as in Lemma \ref{structuregroups}, and let $\mG$ be an open subgroup of $\mG(X)$ which contains $T_g$ for every $g\in G$. Then for every $1\leq i \leq k$, $\mG_i/\mG_{i+1}\Gamma_{\mG_i} \cong \mG(X)_i/\mG(X)_{i+1}\Gamma(X)_i$.
\end{cor}
Let $(H,\cdot)$ be any group and $n\in\mathbb{N}$. We say that $H$ is $n$-divisible if for every $h\in H$ there exists $x\in H$ with $x^{n!}=h$ where $n!=n\cdot (n-1)\cdot...\cdot 1$.  Our goal is to show that $\mG_r/\mG_{r+1}\Gamma_r$ is $k$-divisible for every $k<\min_{p\in P} p$ and every $1\leq r \leq k$. We use a type argument by Host and Kra which requires analysis of the ergodic components of $(X^{[1]},\mu^{[1]})$. We recall the following result by Host and Kra \cite[Lemma 9.1 and Lemma 9.3]{HK}.
\begin{lem}\label{typeargument}
    Let $X$ be an ergodic $G$-system, U a compact abelian
group, $\rho:G\times X\rightarrow U$ a cocycle and $k\geq 0$ an integer. Let $(Z_{<2}(X),\nu)$ be the Kronecker factor and $\mu^{[1]}=\int_{Z_{<2}(X)} \mu_s d\nu(s)$ be the ergodic decomposition of $\mu^{[1]}$ with respect to the diagonal action of $G$. The set
$$ A = \{s\in Z_{<2}(X) : d^{[1]}\rho \text{ is a cocycle of type } <k \text{ of } X_s\}$$
is measurable. Furthermore, the cocycle $\rho$ is of type $<k+1$
if and only if $\nu(A)=1$. Moreover, if $X$ is of order $<k$, then for $\nu$-almost every $s$, the ergodic component $(X^{[1]},\mu_s)$ is a system of order $<k$.
\end{lem}
We deduce the following result.
\begin{lem}\label{claim} Let $1\leq m\leq l$ and $k<\min_{p\in P} p$. Let $X$ be an ergodic $G$-system of order $<m$ and $\rho:G\times X\rightarrow S^1$ a cocycle of type $<l$ such that $\rho^{k!}$ of type $<m-1$. Then $\rho$ is of type $<m-1$.
\end{lem}
Assume this lemma for now, we prove the following result.
\begin{thm} \label{structurediv}
	Let $r\geq 1$, $k<\min_{p\in P}P$ and $X$ be an ergodic $r$-step nilpotent $G$-system  Then, for each $1\leq i \leq r$ we have that $\mG_i/\mG_{i+1}\Gamma_i$ is $k$-divisible.
\end{thm}
\begin{proof}
Write $Z_{<r+1}(X)=Z_{<r}(X)\times_\rho U$. Since $U\cong \mG_r/\mG_{r+1}\Gamma_r$, it is enough to prove that $U$ is $k$-divisible. Assume by contradiction that this is not the case and let $\chi:U\rightarrow C_{k!}$ be a lift of a non-trivial character of the quotient $U/U^{k!}$. By Lemma \ref{claim}, we see that $\chi\circ\rho$ is a cocycle of type $<r-1$. This means that the extension $Z_{<r}(X)\times_{\chi\circ\rho} \chi(U)$ is degenerate. In particular, the maximal property of $Z_{<r}(X)$ provides a contradiction. 
\end{proof}
\begin{proof}[Proof of Lemma \ref{claim}]
We prove the lemma by induction on $m$. If $m=1$, then $X$ is trivial. By assumption $\rho^{k!}$ is a coboundary, hence $\rho^{k!}\equiv 1$. We conclude that  $\rho:G\rightarrow C_{k!}$ is a homomorphism. Since $k<\min_{p\in P} p$, $\rho$ is trivial and the claim follows. Fix $2\leq m$ and assume inductively that the claim holds for smaller values of $m$. Let $X$ be as in the lemma and write $X=Z_{<m-1}(X)\times_\sigma U$. By the induction hypothesis we also know that Theorem \ref{structurediv} holds and so we can assume that $U^{k!}=U$. Our goal is to show that $\rho$ is cohomologous to a cocycle that is measurable with respect to $Z_{<m-1}(X)$. The first step is to reduce matters to the case where $U$ is finite. By Theorem \ref{CL:thm}, there exists an open subgroup $U'\leq U$ such that for every $u\in U'$, there exists a phase polynomial $p_u\in P_{<l-1}(X,S^1)$ and a measurable map $F:X\rightarrow S^1$ such that 
\begin{equation}\label{C.L.}\Delta_u \rho  = p_u \cdot \Delta F_u.
\end{equation}
We claim that $p_u$ is trivial. The cocycle identity implies that 
$$\Delta_{u^{k!}}\rho = \Delta_u \rho ^{k!} \cdot \prod_{i=0}^{k!-1} \Delta_{u^i}\Delta_u \rho.$$ 
Since $\rho^{k!}$ is of type $<m-1$, we conclude by Lemma \ref{dif:lem} that $\Delta_u \rho ^{k!}$ is a coboundary. Moreover, by equation (\ref{C.L.}) and Lemma \ref{vdif:lem}, we see that $\prod_{l=0}^{k!-1} \Delta_{u^l}\Delta_u \rho$ is cohomologous to a phase polynomial of degree $<l-2$. It follows that $\Delta_{u^{k!}}\rho$ is cohomologous to a phase polynomial of degree $<l-2$. Since $U^{k!}=U$, we conclude that $\Delta_u\rho$ is cohomologous to a phase polynomial of degree $<l-2$. Repeating this argument (by induction on the degree of $p_u$), we conclude that $\Delta_u \rho$ is a coboundary for every $u\in U'$. Therefore by Lemma \ref{cob:lem}, $\rho$ is cohomologous to a cocycle $\rho'$ which is invariant with respect to some open subgroup $U''\leq U$.\\

Let $X' = X\times_{\sigma'} U/U''$ where $\sigma'$ is the composition of $\sigma$ with the quotient map $U\mapsto U/U''$. We view $\rho'$ as a cocycle on $X'$. By Lemma \ref{Cdec:lem}, $\rho'$ is of type $<l$ and $\rho'^{k!}$ of type $<m-1$.\\

Now we deal with the finite case. Let $n=|U/U''|$ and let $u\in U/U''$. By Lemma \ref{dif:lem}, the cocycle $\Delta_u \rho'$ is of type $<l-m+1$ and $(\Delta_u \rho')^{k!}$ is a coboundary. We prove that $\Delta_u \rho'$ is also a coboundary. By the cocycle identity
$$1=\Delta_{u^{n}}\rho' =(\Delta_u {\rho'}) ^{n} \cdot \prod_{l=0}^{n-1} \Delta_{u^l}\Delta_u \rho'.$$
By Lemma \ref{dif:lem}, $\prod_{l=0}^{n-1} \Delta_{u^l}\Delta_u \rho'$ is of type $<l-2m+2$ and it follows that so is $(\Delta_u \rho')^m$. Since $U^{k!}=U$, we conclude that $n$ is co-prime to $k!$. In particular, there exists a natural number $d$ such that $nd = 1 \mod k!$. We conclude that $\Delta_u \rho'$ is cohomologous to $\Delta_u \rho'^{nd}$ which is of type $<l-2m+2$, hence $\Delta_u\rho'$ is of type $<l-2m+2$. Since $m\geq 2$, we can continue this argument by induction until the type of $\Delta_u \rho'$ is $<0$. Therefore, for every $u\in U/U''$ we can find a measurable map $F_u:Z_{<m-1}(X)\rightarrow S^1$ such that 
\begin{equation}\label{coboundary}\Delta_u \rho' = \Delta F_u.
\end{equation}
By ergodicity and the cocycle identity, we conclude that for every $u,v\in U/U''$ there exists a constant $c(u,v)$ such that 
\begin{equation} \label{constant}
    F_{uv}/F_u V_u F_v = c(u,v).
\end{equation}
Let $b(u,v) = \frac{\Delta_u F_v}{\Delta_v F_u}$. Since $\Delta_u \Delta_v \rho' = \Delta_v \Delta_u \rho'$, equation \ref{coboundary} implies that $b(u,v)$ is a constant in $x$. Direct computation using equation (\ref{constant}) shows that $b$ is a bilinear map. For instance we have,
$$b(uu',v) = \frac{\Delta_{uu'}F_v}{\Delta_v F_{uu'}} = \frac{\Delta_u F_v V_u \Delta_{u'}F_v}{\Delta_v F_u V_u F_u'}=\frac{\Delta_u F_v}{\Delta_v F_u} V_u \left(\frac{\Delta_u'F_v}{\Delta_v F_u'}\right) = b(u,v)\cdot b(u',v).$$
In particular, we see that $b^n(u,v)=1$ for every $u,v\in U/U''$. Since $n$ is co-prime to $k!$, it is enough to show that $\rho'^n$ is of type $<m-1$. Therefore, we can assume without loss of generality that $b=1$. In this case it follows that the group $$H = \{S_{u,F} : u\in U/U'', F\in\mathcal{M}(Z_{<m-1}(X),S^1), \Delta_u\rho' = \Delta F\}$$ is abelian. By \eqref{coboundary}, the projection $p:H\rightarrow U/U''$ is onto. Moreover, the kernel is isomorphic to $S^1$ and so $H$ is a compact group (Corollary \ref{compact}). Since the torus is injective in the category of compact abelian groups we conclude that there exists a cross-section $u\mapsto S_{u,F_u}$ such that $\Delta_u \rho'  = \Delta F_u$. In particular, $F_{uv}=F_uV_u F_v$ for every $u,v\in U/U''$. Let $F(x,u) = F_u(x,1_{U/U''})$, direct computation shows that $\Delta_v F(x,u) = F_v(x,u)$ for almost every $x\in Z_{<m-1}(X)$ and every $u,v\in U/U''$. We conclude that $\rho'/\Delta F$ is invariant to $U/U'$. In other words, $\rho'$ is cohomologous to a cocycle $\rho''$ which is measurable with respect to $Z_{<m-1}(X)$.\\
We view $\rho''$ as a cocycle of $Z_{<m-1}(X)$. By Lemma \ref{Cdec:lem}, $\rho''$ is of type $<l$ and $\rho''^{k!}$ of type $<m-1$.\\

Now we use an inductive type argument. By Lemma \ref{typeargument}, $d^{[1]}\rho''$ is of type $<l-1$ and $d^{[1]}\rho''^{k!}$ is of type $<m-2$ on every ergodic component of $Z_{<m-1}(X)^{[1]}$. Since $Z_{<m-1}(X)$ is a system of order $<m-1$ so is every ergodic component of $Z_{<m-1}(X)^{[1]}$. We conclude, by the induction hypothesis that $d^{[1]}\rho''$ is of type $<m-2$ on every ergodic component. Therefore, by Lemma \ref{typeargument}, $\rho''$ is of type $<m-1$ on $Z_{<m-1}(X)$. Lifting everything up using the factor map $X\rightarrow Z_{<m-1}(X)$, we conclude that $\rho$ is of type $<m-1$, as required.
\end{proof}
\subsection{The group of arithmetic progressions}
The proof of Theorem \ref{formula} follows the methods of Bergelson, Host and Kra from \cite{BHK}. Let $X=\mG/\Gamma$ be as in Theorem \ref{formula}, we define a function $$\imath:\mG\times \mG_1\times \mG_2\times...\times \mG_k \rightarrow \mG^{k+1}$$ by $$\imath(g,g_1,g_2,...,g_k) = (g,gg_1,gg_1^2g_2,...,gg_1^kg_2^{\binom{k}{2}}\cdot...\cdot g_k^{\binom{k}{k}})$$ and let $\tilde{\mG}$ to be the image of $\imath$ in $\mG^{k+1}$.
\begin{thm}[Leibman \cite{Leib98}] \label{group}
 $\tilde{\mG}$ is a group.
\end{thm}
The group $\tilde{\Gamma}:=\Gamma^{k+1}\cap \tilde{\mG}$ is a closed zero dimensional co-compact subgroup of $\tilde{\mG}$. Let
$$T_g^\star = Id\times T_g\times T_g^2\times...\times T_g^k\text{ and } T_g^{\triangle} = T_g\times T_g\times...\times T_g.$$
 It is easy to see that $T_g^\star$ and $T_g^\triangle$ belongs to $\tilde{\mG}$ and therefore acts on $\tilde{G}/\tilde{\Gamma}$. Our next goal is to prove that the action of $G\times G$ on $\tilde{G}/\tilde{\Gamma}$ by $T_g^\triangle\circ T_h^\star$ is uniquely ergodic.
  \subsection{Green's Theorem}
 Green's theorem \cite{Green} states that in a nilsystem $(\mG/\Gamma,R_a)$ where $\mG$ is a connected simply connected Lie group the action of $R_a$ on $\mG/\Gamma$ is ergodic if and only if the induced action of $R_a$ on the factor $\mG/\mG_2\Gamma$ is ergodic. In \cite{Parry} Parry gave an alternative simpler proof which was then used by Leibman \cite[Theorem 2.17]{Leib3} to generalize this result to arbitrary nilsystems. Parry's proof relies on the fact that on a connected nilsystem $N/\Gamma$ the eigenfunctions are invariant with respect to $N_2$. In Theorem \ref{almostgreen} below we generalize this result for polynomials of higher order and some special nilpotent systems that may not be connected. First we need the following technical lemma.
 \begin{lem} \label{key}
 	Let $f:\tilde{\mG}/\tilde{\Gamma}\rightarrow S^1$ be a measurable function. Let $V\leq \tilde{\mG}$ be an open subgroup which contains the elements $T_g^\star$ and $T_g^\triangle$ for every $g\in G$ and $2\leq r \leq k+1$. Then, if $f$ is invariant with respect to left multiplication by the $r$-commutator subgroup $V_r$ of $V$, then $f$ is invariant to the action of $\tilde{\mG}\cap \mG_r^{k+1}$.
 \end{lem}
\begin{proof}
	Since $V$ is open and $\imath$ is a continuous map, we have that $\imath^{-1}(V)$ contains a subgroup of the form $\mL\times \mL\times\{e\}\times...\times\{e\}$ where $\mL\leq \mG$ is open. Moreover, since $V$ contains $T_g^\star$ and $T_g^\triangle$ we can also assume that $\mL$ contains $T_g$. For each $2\leq r \leq k+1$ let, $$\mH_r:= \tilde{\mG}\cap \mG_{r}^{k+1}.$$
    Now, let $f$ be as in the lemma. We prove by downward induction on $2\leq r\leq k+1$ that $f$ is also invariant with respect to the action of $\mH_r$. If $r=k+1$, then $\mH_{k+1}$ is trivial and the claim follows. Let $2\leq r<k+1$ and assume inductively that $f$ is already invariant to left multiplication by elements in $\mH_{r+1}$.\\
	For convenient, we write the elements of $\tilde{\mG}$ as sequences $x(n)$ where $x:\{0,1,...,k\}\rightarrow \mG$. Since $\tilde{\mG}$ is a group, a general form of an element in $\mH_r$ is $x(n)=g_0g_1^ng_2^{\binom{n}{2}}\cdot...\cdot g_k^{\binom{n}{k}}$ where $g_0,g_1,...,g_r\in \mG_t$ and $g_{r+1}\in \mG_{r+1},...,g_k\in\mG_k$ with the convention that $\binom{n}{m} = \frac{n!}{(n-m)!m!}$ when $m\leq n$ and zero otherwise.\\
	Fix $0\leq m\leq k$ and let $x_m(n)=g_m^{\binom{n}{m}}\in \tilde{G}$, it is enough to show that $f$ is invariant to left multiplication by $x_m$. If $r<m<k+1$, then $x_m\in\tilde{H}_{r+1}$ and the claim follows by induction hypothesis. Otherwise, we can assume that $m\leq r$. In that case $g_m\in\mG_r$.\\
	\textit{Step 1:} We replace $g_m$ with an $m!$-root.\\
	By Lemma \ref{structurediv} we can find an element $h\in \mG_r$ such that $h^{m!}\cdot g' = g_m$ where $g'\in \mG_{r+1}\cdot \Gamma_r$. Hence, $g'^{\binom{n}{m}} = g''^{\binom{n}{m}}\cdot \gamma^{\binom{m}{n}}\cdot y(n)$ where $g''\in\mG_{r+1}$, $\gamma\in \Gamma_r$ and $y$ takes values in $\mH_{r+1}$. Since $\tilde{G}$ is a group we see that $y(n)$ and $g''^{\binom{n}{m}}$ are in $\mH_{r+1}$. As for $\gamma^{\binom{m}{n}}$, we have that
	$$f(\gamma^{\binom{m}{n}}x\Gamma) = f([(\gamma^{-1})^{\binom{m}{n}},x]x\Gamma)$$ and $[(\gamma^{-1})^{\binom{m}{n}},x]\in \mH_{r+1}$. We conclude that $f$ is invariant to left multiplication by $g_m^{\binom{n}{m}}$ if and only if it is invariant to left multiplication by $h^{p_m(n)}$ where $p_m(n)= n!/(n-m)!$ is a polynomial of degree $<n-m$ with natural coefficients.\\
	\textit{Step 2.} We replace $h$ with an element in $\mL_r$.\\
	By Lemma \ref{G=L} we can write $h=l\cdot h' \cdot \delta$ where $l\in \mL_r$, $h'\in \mG_{r+1}$ and $\delta\in \Gamma_r$. Then, we have that $$h^{p_m(n)} = l^{p_m(n)}\cdot \delta^{p_m(n)} \cdot y'(n)$$ where $y'(n)$ takes values in $\mG_{r+1}$. Since $\tilde{G}$ is a group we conclude that $y'\in \mH_{r+1}$. As in the previous step we also know that $f$ is invariant to left multiplication by $\delta^{p_m(n)}$. Therefore, $f$ is $h^{p_m(n)}$-invariant if and only if it is invariant to left multiplication by $l^{p_m(n)}$.\\
	\textit{Step 3:} We show that $f$ is invariant to $l^{p_m(n)}$ and complete the proof.\\
	Since $\mL_r$ is generated by commutators of $r$ elements and $f$ is invariant with respect to the action of $\mH_{r+1}$, we can assume that $l$ is an $r$-commutator. Write $l=[s_1,s_2,s_3,...s_r]$ for some $s_1,s_2,....,s_r\in \mL$. We consider two sequences in $\tilde{\mG}$ for each $s_i$. The first is the constant sequence which we denote by $c_i(n)=s_i$. The second is the arithmetic progression with no constant term, namely $d_i(n)=s_i^n$. Observe that for each $1\leq j \leq r$ we have $[d_1,d_2,...,d_j,c_{j+1},...,c_r] = l^{n^j}\cdot z_j(n)$ where $z_j(n)$ takes values in $\mG_{r+1}$, and so is in $\mH_{r+1}$. We conclude that $f$ is $l^{n^j}$-invariant for all $1\leq j \leq r$ and so it is also invariant to left multiplication by $l^{p_m(n)}$. This completes the proof.
\end{proof}
 We note that since $\mG_r \trianglelefteq \mG$, it follows that $\mH_r$ is a normal subgroup of $\tilde{G}$.\\
 \textbf{Convention:} For the sake of the proof of the ergodicity of $\tilde{G}/\tilde{\Gamma}$ with respect to $T_g^\star$ and $T_g^\triangle$ we say that a homogeneous space $N/\Gamma$ with an action of $\varphi:G\rightarrow N$ is \textit{special} if the induced action of $\varphi$ on $N/N_2\Gamma$ is ergodic and for every open subgroup $V\leq N$ which contains $\varphi(G)$ we have that for every $2\leq r \leq k$ any function $f:N/\Gamma\rightarrow S^1$ is invariant with respect to the action of $V_r$ if and only if it is invariant with respect to $N_r$.\\ We note that by the previous lemma, $\tilde{G}/\tilde{\mG}_l\tilde{\Gamma}$ is an $l$-step special homogeneous space for every $1\leq l \leq k$. We generalize Green theorem.
 \begin{thm}[Green theorem for special homogeneous spaces]\label{almostgreen}
 Let $N/\Gamma$ be a $k$-step special homogeneous space. Then, for every $1\leq d \leq k$ and $1\leq r < d$ we have the following results.
 \begin{enumerate}
     \item {$f$ is invariant with respect to the action of $N_d$.}
     \item{For every $n\in N_r$, $\Delta_n f$ is a phase polynomial of degree $<d-r$.}
     \item{For every $n\in N$, $V_n f$ is a phase polynomial of degree $<d$.}
 \end{enumerate}
 \end{thm}
  Note that from the case $d=1$ in the theorem above we can deduce that every special homogeneous space is ergodic.
 \begin{cor}\label{ergodic}
 	Let $X=N/\Gamma$ be a special $k$-step homogeneous space. Then $X$ is ergodic. In particular, $\tilde{\mG}/\tilde{\Gamma}$ is ergodic with respect to the action generated by $T_g^\star$ and $T_g^\triangle$. 
 \end{cor}
 A nilpotent system is ergodic if and only if it is uniquely ergodic \cite[Section 2, Lemma 1]{Parryb} (see also \cite[Theorem 5]{Parrya}).
\begin{thm} \label{uniquelyergodic}
    The action generated by $T_g^\triangle$ and $T_g^\star$ on $\tilde{\mG}/\tilde{\Gamma}$ is uniquely ergodic.
\end{thm}
 \begin{proof}[Proof of Theorem \ref{almostgreen}]
 We prove the claims by induction on $k$ and then by induction on $d$. If $k=1$ then the claims follow because the system is ergodic and every $n\in N$ is an automorphism. Fix $k\geq 2$ and assume that the claims hold for all smaller values of $k$.\\
 \textbf{Induction basis:} The case $d=1$ follows by adapting the argument of Parry. Let $f:N/\Gamma\rightarrow\mathbb{C}$ be an invariant function. The compact abelian group $N_k/\Gamma_k$ defines a unitary action on $L^2(N/\Gamma)$ by translations. In particular, there is a decomposition of $f$ to eigenfunctions with respect to this action. Namely $f=\sum_\lambda f_\lambda$ where $\Delta_n f_\lambda = \lambda(n)$ for every $n\in N_k/\Gamma_k$ where $\lambda:N_k/\Gamma_k\rightarrow S^1$ is a character. Since the action of $G$ commutes with the action of $N_k$ we can also assume that the $f_\lambda$'s are eigenfunctions with respect to the $G$-action. Thus, $|f_\lambda|$ is $G$-invariant and invariant with respect to $N_k$ and so by induction hypothesis and Corollary \ref{ergodic} we can write $f=\sum_\lambda a_\lambda f_\lambda$ where $a_\lambda\in\mathbb{C}$ and $f_\lambda$ take values in $S^1$. Now, we claim by downward induction on $1\leq r\leq k$ that:\\
 
 \noindent \textbf{Claim:} For every $n\in N_r$ and $\lambda$, $\Delta_n f_\lambda$ is a phase polynomial of degree $<k-r+1$.
 \begin{proof}[Proof of Claim:] If $r=k$ then $\Delta_n f_\lambda = \lambda(n)$ is a constant. Fix $r<k$ and assume inductively that the claim holds for larger values of $r$ and let $n\in N_r$. Observe that for every $g\in G$ we have
 $$\Delta_g \Delta_n f_\lambda = \Delta_n \Delta_g f_\lambda \cdot V_n T_g \Delta_{[n^{-1},g^{-1}]} f_\lambda.$$
 Since $\Delta_g f_\lambda$ is a constant the term $\Delta_n \Delta_g f_\lambda$ vanishes. Moreover, by the induction hypothesis $\Delta_{[n^{-1},g^{-1}]}f_\lambda$ is a phase polynomial of degree $<k-r$. Observe that $\Delta_{[n^{-1},g^{-1}]}f_\lambda$ is invariant with respect to the action of $N_k$ and therefore by the induction hypothesis on $k$, we conclude that $V_n T_g \Delta_{[n^{-1},g^{-1}]} f_\lambda$ is also of degree $<k-r$. It follows that $\Delta_g \Delta_n f_\lambda$ is of degree $<k-r$ for every $g\in G$ and therefore $\Delta_n f_\lambda$ of degree $<k-r+1$, as required.
 \end{proof}
 Now, we apply the claim with $r=1$. We deduce that for every $n\in N$, $\Delta_n f_\lambda$ is a phase polynomial of degree $<k$. Since $\Delta_n f_\lambda$ and is invariant with respect to the action of $N_k$ and $N/N_k\Gamma$ is ergodic, Lemma \ref{sep:lem} implies that the group $$V_\lambda:=\{n\in N : \Delta_n f_\lambda \text{ is a constant }\}$$ is open. The map $v\mapsto \Delta_v f_\lambda$ is a homomorphism from $V_\lambda$ to the abelian group $S^1$ and so is trivial on $(V_\lambda)_2 = N_2$. We conclude that $f=\sum_\lambda a_\lambda f_\lambda$ is also invariant with respect to $N_2$. Since $N/N_2\Gamma$ is ergodic, $f$ is a constant and the rest of the claims follow.\\
 \textbf{Induction step:} Fix $d>1$ and assume inductively that the claims hold for smaller values of $d$.\\ Observe that by the case $d=1$ we can assume that $N/\Gamma$ is ergodic. Let $f:N/\Gamma\rightarrow \mathbb{C}$ be a phase polynomial of degree $<d$. By setting $g_1=...=g_d=0$ we conclude that $|f|^{2^d}=1$ and therefore $f$ takes values in $S^1$. We show that $f$ is invariant with respect to $N_d$ by adapting the argument from the induction basis. As before, we prove by induction on $r$ that $\Delta_n f$ is of degree $<k-r+1$. If $r=k$, then since $\Delta_g f$ is invariant to $N_{d-1}$ we see that $\Delta_g \Delta_n f = \Delta_n \Delta_g f = 1$, as required. Fix $r<k$ and let $n\in N_r$ then,
\begin{equation}\label{commutator}
\Delta_g \Delta_n f = \Delta_n \Delta_g f \cdot V_n T_g \Delta_{[n^{-1},g^{-1}]} f.
\end{equation}
 By induction hypothesis $\Delta_{[n^{-1},g^{-1}]} f$ is of degree $<k-r$, since this function is invariant with respect to $N_k$, the same argument as in the induction basis gives that $V_n T_g \Delta_{[n^{-1},g^{-1}]} f$ is also a phase polynomial degree $<k-r$. If $r\geq d-1$ then $\Delta_n \Delta_g f$ vanishes and therefore $\Delta_n f$ is of degree $<k-r+1$, as required. If $r<d-1$ then, by the induction hypothesis on $d$, we conclude that $\Delta_n \Delta_g f$ is of degree $<d-r$. Since $d<k$, \eqref{commutator} implies that $\Delta_n f$ is of degree $<k-r+1$.\\  In particular, by the case $r=1$, we conclude that for every $n\in N$, $\Delta_n f$ is a phase polynomial of degree $<k$. This time, consider the subgroup
 $$V= \{n\in N : \Delta_n f \text{ is a phase polynomial of degree } <d-1\}$$
 As in the induction basis this is an open subgroup which contains the image of $\varphi:G\rightarrow N$.\\
Write $\Delta_v f = p_v$ and observe that since $p_v$ is invariant to $N_k$, we have by induction hypothesis that for every $v'\in V$, $\Delta_{v'}p_v$ is of degree $<d-1$. Moreover, by the cocycle identity we have that $p_{vv'} = p_v\cdot p_{v'} \cdot \Delta_{v'}p_v$. It follows that $v\mapsto p_v\cdot P_{<d-2}(X,S^1)$ is a homomorphism and so trivial with respect to $V_2$. In other words, for every $v\in V_2$, $p_v$ is a phase polynomial of degree $<d-2$. Continue this way by induction, we see that $p_v=1$ for every $v\in V_d$. Since $N$ is special, $V_d=N_d$ and the first claim follows.
Viewing $f$ as a polynomial of degree $<d$ on $N/N_d\Gamma$ the rest of the claims follow by the induction hypothesis on $k$.
 \end{proof}
\subsection{The proof of Theorem \ref{formula}}
We construct new systems.\\
Let $X=\mG/\Gamma$ and $\tilde{X}=\tilde{\mG}/\tilde{\Gamma}$ as in the previous sections. For every $x\in\mG/\Gamma$ the set
$$\tilde{X}_x := \{(x_1,x_2,...,x_k)\in X^k : (x,x_1,x_2,...,x_k)\in \tilde{X}\}$$ is a compact subset of $X^{k}$.  As in Bergelson, Host and Kra \cite{BHK}, the group $\tilde{\mG}^\star$ acts on $\tilde{X}_x$ transitively. Moreover, if  $\tilde{\Gamma}_x\leq \tilde{\mG}^\star$ is the stabilizer of $(x,x,x,...,x)$, then $\tilde{X}_x \cong \tilde{\mG}^\star/\tilde{\Gamma}_x$. Let $\mu$ be the Haar measure on $X$, $\tilde{\mu}$ the Haar measure on $\tilde{X}$ and $\tilde{\mu}_x$ on $\tilde{X}_x$. Using the fact that $\tilde{X}$ is uniquely ergodic Bergelson, Host and Kra proved,
\begin{equation}
\tilde{\mu} = \int_X \delta_x \otimes \tilde{\mu}_x d\mu(x).
\end{equation}
We prove the limit formula in Theorem \ref{formula} following the argument in \cite[Theorem 5.4]{BHK}.
\begin{proof}
	We first prove the claim in the case where the functions are continuous. Since $(x,x,x,...,x)\in \tilde{X}_x$, we can apply the pointwise mean ergodic theorem for the space $\tilde{X}_x$ with respect to the action of $(T_g,T_g^2,T_g^3,...,T_g^k)\in \tilde{\mG}^\star$. The limit is some function $\phi$ on $X$. Let $f$ be any continuous function on $X$ we have,
	\begin{equation} \label{ip}
\int f(x)\phi(x) d\mu(x) = \lim_{N\rightarrow\infty} \mathbb{E}_{g\in\Phi_N} \int_X f(x)\prod_{j=1}^k f_j(T_g^jx)d\mu(x).
		\end{equation}
	 We translate the functions in equation (\ref{ip}) by $T_h$ and then take an average over $h\in G$. Since $\mu$ is $T_h$ invariant for every $h\in G$ the limit above equals to
	$$  \lim_{N\rightarrow\infty} \mathbb{E}_{g\in \Phi_N}\mathbb{E}_{h\in\Phi_N}\int_X f(T_h x) \prod_{j=1}^k f_j (T_g^jT_hx)d\mu(x).$$
		The action generated by $T_h^\triangle$ and $T_g^\star$ on $\tilde{X}$ is uniquely ergodic. Therefore, by the mean ergodic theorem, the limit above converges everywhere to $$\int_{\tilde{X}} f(x_0)f_1(x_1)\cdot...\cdot f_k(x_k)d\tilde{\mu}(x_0,x_1,...,x_k) = \int_X f(x) \int_{\tilde{X}_x} f_1(x_1)\cdot...\cdot f_k(x_k)d\tilde{\mu}(x_1,...,x_k)d\mu(x_0).$$
		Since this holds for all continuous functions $f$, we conclude that $$\phi(x)= \int_{\tilde{X}_x} f_1(x_1)\cdot...\cdot f_k(x_k)d\tilde{\mu}(x_1,...,x_k)d\mu(x_0)$$ whenever $f_1,...f_k$ are continuous.\\
		The map $(g_1,g_2,...,g_k)\mapsto (gg_1,gg_1^2g_2,...,gg_1^kg_2^{\binom{k}{2}}\cdot...\cdot g_k^{\binom{k}{k}})$ from $\mG/\Gamma\times \mG_2/\Gamma_2\times...\times\mG_k$ to $\tilde{X}_x$ is an isomorphism and so $\phi(x)$ equals to the function in (\ref{formulaequation}). This completes the proof for continuous functions and by approximation argument the convergence holds for all bounded functions as well.
		\end{proof}
\appendix
\section{Survey of some notations and previous results} \label{background}
The goal of this section is to survey some definitions and known results from previous work. Most of these theorems and definitions appear in  \cite{Berg& tao & ziegler} or in \cite{HK}.
\subsection{Notations}
\begin{defn} [Abelian cohomology] Let $G$ be a countable discrete abelian group. Let $X=(X,\mathcal{B},\mu,G)$ be a $G$-system and let $U=(U,\cdot)$ be a compact abelian group. 
	\begin{itemize}
		\item{We denote by $\mathcal{M}(X,U)$ or $\mathcal{M}(G,X,U)$ the group of all measurable functions $\phi:X\rightarrow U$ or $\phi:G\times X\rightarrow U$, respectively. We say that two functions $f_1,f_2\in \mathcal{M}(X,U)$ are equal if $f_1(x)=f_2(x)$ for $\mu$-almost every $x$. Similarly, if $f_1,f_2\in \mathcal{M}(G,X,U)$ then they are equal if $f_1(g,x)=f_2(g,x)$ for $\mu$-almost $x\in X$ and every $g\in G$.}
		\item {A $(G,X,U)$-cocycle is a measurable function $\rho:G\times X\rightarrow U$ which satisfies that $\rho(g+g',x)=\rho(g,x)\rho(g',T_gx)$ for all $g,g'\in G$ and $\mu$-almost every $x\in X$. We let $Z^{1}(G,X,U)$ denote the subgroup of all cocycles.}
		\item {Given a cocycle $\rho:G\times X\rightarrow U$, we define the abelian extension $X\times_{\rho} U$ of $X$ by $\rho$ to be the product space $(X\times U,\mathcal{B}_X\otimes\mathcal{B}_U,\mu_X\otimes\mu_U)$ where $\mathcal{B}_U$ is the Borel $\sigma$-algebra on $U$ and $\mu_U$ the Haar measure. We define the action of $G$ on this product space by $(x,u)\mapsto (T_g x,\rho(g,x)u)$ for every $g\in G$. In this situation, we define by  $V_u(x,t)=(x,u t)$ the vertical rotation of some $u\in U$ on $X\times_{\rho} U$.}
		\item {If $F\in \mathcal{M}(X,U)$, we define the derivative $\Delta F\in \mathcal{M}(G,X,U)$ of $F$ to be the function $\Delta F(g,x):=\Delta_g F(x)$. We write $B^1(G,X,U)$ for the image of $\mathcal{M}(X,U)$ under the derivative operation. We refer to the elements in $B^1(G,X,U)$ as $(G,X,U)$-coboundaries.}
		\item {We say that $\rho,\rho'\in \mathcal{M}(G,X,U)$ are $(G,X,U)$-cohomologous if $\rho/\rho'\in B^1(G,X,U)$. In that case it is easy to see that the abelian extensions defined by $\rho$ and $\rho'$ are isomorphic.}
	\end{itemize}
\end{defn}
\subsection{Cubic measure spaces and type of functions}
We begin by introducing the cubic spaces from \cite[Section 3]{HK} (Generalized for arbitrary countable abelian group).
\begin{defn} [Cubic measure spaces]\label{CMS} Let $X=(X,\mathcal{B},\mu,G)$ be a $G$-system for some countable abelian group $G$. For each $k\geq 0$ we define $X^{[k]} =(X^{[k]},\mathcal{B}^{[k]},\mu^{[k]},G)$ where $X^{[k]}$ is the Cartesian product of $2^k$ copies of $X$, endowed with the product $\sigma$-algebra $\mathcal{B}^{[k]}=\mathcal{B}^{2^k}$, and $G$ acting on $X^{[k]}$ diagonally (i.e. $T_g((x_\omega)_{\omega\in\{0,1\}^k}) = (T_g x_\omega)_{\omega\in\{0,1\}^k})$.). We define the cubic measures $\mu^{[k]}$ and $\sigma$-algebras $\mathcal{I}_k\subseteq \mathcal{B}^{[k]}$ inductively. $\mathcal{I}_0$ is defined to be the $\sigma$-algebra of invariant sets in $X$ and $\mu^{[0]}:=\mu$. Let $k\geq 0$ and suppose that  $\mu^{[k]}$ and $\mathcal{I}_k$ are already defined. We identify $X^{[k+1]}$ with $X^{[k]}\times X^{[k]}$ and define $\mu^{[k+1]}$ by the formula 
	$$\int f_1(x)f_2(y) d\mu^{[k+1]}(x,y) = \int E(f_1|\mathcal{I}_k)(x)E(f_2|\mathcal{I}_k)(x) d\mu^{[k]}(x).$$
	For $f_1,f_2$ functions on $X^{[k]}$ and $E(\cdot|\mathcal{I}_k)$ the conditional expectation and let $\mathcal{I}_{k+1}$ be the $\sigma$-algebra of invariant sets in $X^{[k+1]}$.\\
	We adapt the notion of face from \cite[Section 2]{HK}. Let $V_k:=\{-1,1\}^{2^k}$ and for every $0\leq l \leq k$, if $J\in 2^k$ (equivalently $J\subseteq \{1,...,k\}$) is a set of size $k-l$ and $\eta\in \{-1,1\}^J$ then the subset
	$$\alpha:=\{\varepsilon\in V_k :\forall_{j\in J} \varepsilon_j = \eta_j  \}$$ is called an $l$-dimensional face. For any transformation $u:X\rightarrow X$ and a face $\alpha$ we define a transformation $u_\alpha^{[k]}$ on $X^{[k]}$ by
	$$(u_\alpha^{[k]})_{\varepsilon\in V_k} = \begin{cases} u & \varepsilon\in \alpha\\ Id & \text{otherwise.} \end{cases}$$
\end{defn}
We survey some results from \cite{HK}. We begin with the following result about the relation between measure preserving transformations, faces and the measure $\mu^{[k]}$.
\begin{lem} \cite[Lemma 5.3]{HK} \label{facetransformations}
    Let $G$ be a countable abelian group and $X$ be an ergodic $G$ system. Let $0\leq l \leq k$ be integers. For a measure-preserving transformation $t: X\rightarrow X$ the following are equivalent:
    \begin{enumerate}
        \item {For any $l$ dimensional face $\alpha$ of $V_k$, the transformation $t_\alpha^{[k]}$ leaves the measure $\mu^{[k]}$ invariant and maps the $\sigma$-algebra $\mathcal{I}^{[k]}$ to itself.}
        \item{ For any $l+1$ dimensional face $\beta$ of $V_{k+1}$, the transformation $t_\beta^{[k+1]}$ leaves $\mu^{[k+1]}$-invariant.}
       \item{ For any $l+1$ dimensional face $\gamma$ of $V_k$, the transformation $t_\gamma^{[k]}$ leaves the measure $\mu^{[k]}$-invariant and acts trivially on the $\sigma$-algebra $\mathcal{I}^{[k]}$.}
    \end{enumerate}
\end{lem}
We also need the following result related to the ergodic decomposition of $\mu^{[k]}$.
\begin{lem} \label{ergdec} \cite[Corollary 3.5]{HK}
Let $X$ be an ergodic $G$-system and $k\geq 1$ then the following holds,
\begin{itemize}
    \item {There exists a measure space $(\Omega_k,P_k)$ and an ergodic decomposition $$\mu^{[k]}:= \int_{\Omega_k} \mu_{\omega} dP(\omega)$$}
    \item{For every $(k-1)$-face $\alpha$ and every $g\in G$ the transformation $g_\alpha^{[k]}$ sends an ergodic component to an ergodic component. In other words $g_{\alpha}^{[k]}$ acts on $(\Omega_k,P_k)$.}
    \item{The action of the group generated by $g_{\alpha}^{[k]}$ for all $g\in G$ and all $(k-1)$-faces $\alpha$ on $(\Omega_k,P_k)$ is ergodic.}
\end{itemize}
\end{lem}
The definition of cubic measure spaces (Definition \ref{CMS}) leads to the following definition of type for measurable functions.
\begin{defn} [Functions of type $<k$] \cite[Definition 4.1]{Berg& tao & ziegler}  \label{type:def} Let $G$ be a countable abelian group, let $X=(X,\mathcal{B},\mu,G)$ be a $G$-system. Let $k\geq 0$ and let $X^{[k]}$ be the cubic system associated with $X$.
	\begin{itemize}
		\item{For each measurable $f:X\rightarrow U$, we define $d^{[k]}f:X^{[k]}\rightarrow U$ by $$d^{[k]}f((x_w)_{w\in \{-1,1\}^k}):=\prod_{w\in \{-1,1\}^k}f(x_w)^{\text{sgn}(w)}$$ where $\text{sgn}(w)=w_1\cdot w_2\cdot...\cdot w_k.$}
		\item {Similarly for each measurable $\rho:G\times X\rightarrow U$ we define $d^{[k]}\rho:G\times X^{[k]}\rightarrow U$ by
			$$d^{[k]}\rho(g,(x_w)_{w\in \{-1,1\}^k}):=\prod_{w\in \{-1,1\}^k} \rho(g,x_w)^{\text{sgn}(w)}.$$}
		\item {A function $\rho:G\times X\rightarrow U$ is said to be a function of type $<k$ if $d^{[k]}\rho$ is a $(G,X^{[k]},U)$-coboundary. We let $\mathcal{M}_{<k}(G,X^{[k]},U)$ denote the subspace of functions $\rho:G\times X\rightarrow U$ of type $<k$.}
	\end{itemize}
\end{defn}
Using the Pontryagin dual, Moore and Schmidt \cite[Theorem 4.3]{MS} proved the following result.
\begin{thm} \label{MScob}
    Let $X$ be a $G$-system and $U$ a compact abelain group. Let $k\geq 0$ be an integer and $f:G\times X\rightarrow U$ a measurable map. Then, $f$ is of type $<k$ if and only if $\chi\circ f:G\times X\rightarrow S^1$ is of type $<k$ for every $\chi\in\widehat U$.
\end{thm}
We summarize previous results about type of functions. We begin with the following definition.
\begin{defn}
    Let $X$ be a $G$-system. Let $U=(U,\cdot)$ be a group and $f:G\times X\rightarrow U$ a function. For every $k\in\mathbb{N}$ and every face $\alpha\in V_k$ we define a function $d_\alpha^{[k]}f:G\times X^{[k]}\rightarrow U$ by
    $d_\alpha^{[k]}f(g,(x_\omega)_{\omega\in 2^k}) = \prod_{\omega\in \alpha} f(g,x_\omega)^{\sgn(\omega)}$.
\end{defn}
We have the following result \cite[Lemma C.7]{HK}.
\begin{lem}\label{typereduce}
    Let $X$ be an ergodic $G$-system and $U$ be a compact abelian group. Let $f:X\rightarrow U$ be a function and $\alpha$ be an $m$-dimensional face of $V_k$ for some $1\leq m < k$. If $d_\alpha^{[k]}f$ is a $(G,X^{[k]},U)$-coboundary, then $f$ is of type $<m$.
\end{lem}
The following lemma studies the interactions between type and factors.
\begin{lem} [Decent Lemma] \label{Cdec:lem} \cite[Proposition 8.11]{Berg& tao & ziegler}
	Let $X$ be an ergodic $G$-system of order $<k$. Let $Y$ be a factor of $X$, with factor map $\pi:X\rightarrow Y$. Let $\rho:G\times Y\rightarrow S^1$ be a cocycle. If $\rho\circ\pi$ is of type $<k$, then $\rho$ is of type $<k$.
\end{lem}
We also have the following more general version for quasi-cocycles \cite[Proposition 8.11]{Berg& tao & ziegler}.
\begin{lem} \label{dec:lem}
	Let $X$ be an ergodic $G$-system of order $<k$ and $\pi:X\rightarrow Y$ be a factor of $X$. If $f:G\times Y\rightarrow S^1$ is a quasi-cocycle of order $<k-1$, such that $f\circ \pi$ is of type $<k$ then $f$ is of type $<k$.
\end{lem}
We are particularly interested in certain measure preserving transformations $t:X\rightarrow X$ on a $G$-system $X$.
\begin{defn} [Automorphism] \label{Aut:def} Let $X$ be a $G$-system. A measure-preserving transformation $u:X\rightarrow X$ is called an Automorphism if the action it induces on $L^2(X)$ by $f\mapsto f\circ u$ commutes with the action of $G$. In particular we set $\Delta_u f = f\circ u\cdot \overline{f}$
\end{defn}
Automorphisms arise naturally from Host-Kra's theory. For instance, given an abelian extension $Y\times_\rho U$, the group $U$ acts on this extensions by automorphisms defined by $V_u (y,t)=(y,tu)$.\\
Given a function $f:G\times X\rightarrow U$ of type $<k$, the derivative of $f$ by an automorphism $t$ decreases the type of $\Delta_t f$.
\begin{lem}  [Differentiation Lemma] \cite[Lemma 5.3]{Berg& tao & ziegler}\label{dif:lem} Let $k\geq 1$, and let $X$ be a $G$-system of order $<k$. Let $f:G\times X\rightarrow S^1$ be a function of type $<m$ for some $m\geq 1$. Then for every automorphism $t:X\rightarrow X$ which preserves $Z_{<k}(X)$ the function $\Delta_t f(g,x) := f(g,tx)\cdot\overline{f}(g,x)$ is of type $<m-\min(m,k)$.
\end{lem}
\subsection{Phase polynomials}
Phase polynomials play an important role throughout this paper. We begin with the following definition.
\begin{defn} [Phase polynomials] \label{Phasepoly:def}
	Let $G$ be a countable abelian discrete group, $X$ be a $G$-system, let $\phi\in L^\infty(X)$, and let $k\geq 0$ be an integer. A function $\phi:X\rightarrow \mathbb{C}$ is said to be a phase polynomial of degree $<k$ if we have $\Delta_{h_1}...\Delta_{h_k}\phi = 1$ $\mu_X$-almost everywhere for all $h_1,...,h_k\in G$. (In particular, setting $h_1=...=h_k=0$ we see that $\phi$ must take values in $S^1$, $\mu_X$-a.e.). We write $P_{<k}(X)=P_{<k}(X,S^1)$ for the set of all phase polynomials of degree $<k$. Similarly, a function $\rho:G\times X\rightarrow \mathbb{C}$ is said to be a phase polynomial of degree $<k$ if $\rho(g,\cdot)\in P_{<k}(X,S^1)$ for every $g\in G$. We let $P_{<k}(G,X,S^1)$ denote the set of all phase polynomials $\rho:G\times X\rightarrow \mathbb{C}$ of degree $<k$.
\end{defn}
\begin{rem}
	The notion of phase polynomials can be generalized for an arbitrary abelian group $(U,\cdot)$ . Let $\phi:X\rightarrow U$ be a measurable function and $g\in G$, we can define the derivative $\Delta_g \phi(x)$ by the formula $\phi(T_gx)\cdot \phi(x)^{-1}$. A function $\phi:X\rightarrow U$ is said to be a phase polynomial of degree $<k$ if $\Delta_{h_1}...\Delta_{h_k}\phi = 1$ $\mu_X$-a.e. for every $h_1,...,h_k\in G$. We let $P_{<k}(X,U)$ denote the phase polynomials of degree $<k$ which take values in $U$.
\end{rem}
We have the following characterization of phase polynomials \cite[Lemma 5.3]{Berg& tao & ziegler}.
\begin{lem} \label{PP}
	Let $k\geq 0$ be an integer. Let $G$ be a countable abelian group and $X$ be an ergodic $G$-system. $f:X\rightarrow U$ be a measurable map into  compact abelian group $U=(U,\cdot)$. Then, $f$ is a phase polynomial of degree $<k$ if and only if $d^{[k]}f=1$, $\mu^{[k]}$-almost everywhere. 
\end{lem}
We need the following a counterpart of Lemma \ref{dif:lem} for phase polynomials \cite[Lemma 8.8]{Berg& tao & ziegler}.
\begin{lem} \label{polydiff}
    	In the settings of Lemma \ref{dif:lem} if $f$ is a phase polynomial of degree $<m$ then $\Delta_t f(x)$ is of degree $<m-\min(m,k)$.
\end{lem}
The following lemma implies, in particular, that there are at most countably many $(X,S^1)$-phase polynomials in any ergodic $G$-system $X$.
\begin{lem}  [Separation Lemma] \cite[Lemma C.1]{Berg& tao & ziegler}\label{sep:lem} Let $X$ be an ergodic $G$-system, let $k\geq 1$, and $\phi,\psi\in P_{<k}(X,S^1)$ be such that $\phi/\psi$ is non-constant. Then $\|\phi-\psi\|_{L^2(X)} \geq \sqrt{2}/2^{k-2}$.
\end{lem}
The famous Theorem of Bergelson, Tao and Ziegler \cite{Berg& tao & ziegler} states the following result.
\begin{thm}  [Structure theorem for $Z_{<k}(X)$ for ergodic $\mathbb{Z}/p\mathbb{Z}^\omega$-systems] \label{BTZ} There exists a constant $C(k)$ such that for any ergodic $\mathbb{F}_p^\omega$-system $X$, $L^2(Z_{<k}(X))$ is generated by phase polynomials of degree $<C(k)$.  Moreover, if $p$ is sufficiently large with respect to $k$ then $C(k)=k$.
\end{thm}
In \cite{OS} we generalized this result for totally disconnected systems (see Definition \ref{TD:def} and Theorem \ref{Main:thm} below). 
\subsection{The structure of systems of order $<k$}
Let $X$ be an ergodic $G$-system of order $<k$. Then $X$ can be written as a tower of abelian extensions \cite[Proposition 6.3]{HK}.
\begin{prop}  [Order $<k+1$ systems are abelian extensions of order $<k$ systems]\label{abelext:prop} Let $G$ be a discrete countable abelian group, let $k\geq 1$ and $X$ be an ergodic $G$-system of order $<k+1$. Then $X$ is an abelian extension $X=Z_{<k}(X)\times_{\rho} U$ for some compact abelian metrizable group $U$ and a cocycle $\rho:G\times Z_{<k}(X)\rightarrow U$ of type $<k$.
\end{prop}
In particular, it follows that every ergodic $G$-system of order $<k+1$ is isomorphic to a tower of abelian extensions $U_0\times_{\rho_1}U_1\times...\times_{\rho_k}U_k$ where $\rho_i:G\times Z_{<i-1}(X)\rightarrow U_i$ is a cocycle of type $<i$. This leads to the following definitions.
\begin{defn} [Totally disconnected and Weil systems] \label{TD:def}
	Let $X$ be an ergodic $G$-system of order $<k$ and write $X=U_0\times_{\rho_1} U_1\times_{\rho_2}...\times_{\rho_{k-1}}U_{k-1}$. We say that $X$ is a totally disconnected system if $U_0,U_1,...,U_{k-1}$ are totally disconnected groups.
\end{defn}
In \cite{OS} we proved a generalization of Theorem \ref{BTZ} for totally disconnected $\bigoplus_{p\in P}\mathbb{Z}/p\mathbb{Z}$-systems.
\begin{thm} [Functions of finite type on totally disconnected systems are cohomologous to phase polynomials]\label{Main:thm}
	Let $k,m,l\geq 0$ be integers and $P$ be a multiset of primes. If $X$ is an ergodic totally disconnected $\bigoplus_{p\in P}\mathbb{Z}/p^m\mathbb{Z}$-system of order $<k$, then every function $f:\bigoplus_{p\in P}\mathbb{Z}/p^m\mathbb{Z}\times X\rightarrow S^1$ of type $<l$ is $(\bigoplus_{p\in P}\mathbb{Z}/p^m\mathbb{Z},X,S^1)$-cohomologous to a phase polynomial of degree $<O_{k,m,l}(1)$.\footnote{We denote by $O_{k,m,l}(1)$ a quantity which is bound by a constant depending only on $k$, $m$ and $l$.}
\end{thm}
Note that the proof in \cite{OS} is only given in the case $m=1$, but the general case follows similarly. 
\subsection{Conze-Lesigne equations}  In \cite{CL84},\cite{CL87},\cite{CL88} Conze and Lesigne studied the structure of ergodic $\mathbb{Z}$-systems of order $<3$. They identified a particular functional equation involving the cocycle $\rho$ defining the extension $Z_{<3}(X)=Z_{<2}(X)\times_\rho U$. We refer to this equation (Equation (\ref{CLequation:def}) below) as a Conze-Lesigne type equation.
\begin{defn} \label{CLe}
Let $X$ be an ergodic $G$-system, $\rho:G\times X\rightarrow S^1$ be a cocycle and $U$ a compact abelian group which acts on $X$. Let $m\geq 0$  we say that $\rho$ is a \textit{Conze-Lesgine cocycle of degree $<m$ with respect to $U$} if for every $u\in U$ we have that \begin{equation}\label{CLequation:def}\Delta_u \rho(g,x) = p_u(g,x) \cdot \Delta_g F_u(x)
\end{equation} for some $(G,X,S^1)$-phase polynomial $p_u$ of degree $<m$ and a measurable map $F_u:X\rightarrow S^1$ is a measurable map.
\end{defn}
Below are some results regarding Conze-Lesgine cocycles. The first lemma implies that we can choose the terms $p_u$ and $F_u$ measurable in $u$ \cite[Lemma C.4]{Berg& tao & ziegler}.
\begin{lem}  [Measure selection Lemma] \label{sel:lem} Let $X$ be an ergodic $G$-system, and let $k\geq 1$. Let $U$ be a compact abelian group. If $u\mapsto h_u$ is Borel measurable map from $U$ to $\mathcal{P}_{<k}(G,X,S^1)\cdot \mathcal{B}^1(G,X,S^1) \subseteq \mathcal{M}(G,X,S^1)$ where $\mathcal{M}(G,X,S^1)$ is the group of measurable maps of the form $G\times X\rightarrow S^1$ endowed with the topology of convergence in measure, then there is a Borel measurable choice of $f_u,\psi_u$ (as functions from $U$ to $\mathcal{M}(X,S^1)$ and $U$ to $P_{<k}(G,X,S^1)$ respectively) obeying that $h_u = \psi_u \cdot \Delta f_u$.
\end{lem}
The following lemma studies Conze-Lesigne cocycles of degree $<0$, \cite[Lemma C.9]{HK}.
\begin{lem} [Straightening nearly translation-invariant cocycles] \label{cob:lem}
	Let $X$ be an ergodic $G$-system, let $K$ be a compact abelian group acting freely on $X$ and commuting with the $G$-action, and let $\rho:G\times X\rightarrow S^1$ be such that $\Delta_k\rho$ is a $(G,X,S^1)$-coboundary for every $k\in K$, then $\rho$ is $(G,X,S^1)$-cohomologous to a function which is invariant under the action of some open subgroup of $K$. 
\end{lem}
\begin{rem}
	Note that if $K$ is connected then it has no non-trivial open subgroups (see Lemma \ref{connectedcomponent}). In this case we have that $\rho$ is $(G,X,S^1)$-cohomologous to a function which is invariant under $K$. Moreover it is important to note that such result does not work for cocycles which takes values in an arbitrary compact abelian group.
\end{rem}
The next lemma asserts that we can \textit{locally linearize} the term $p_u$ in the Conze-Lesigne equation.
\begin{lem} [Linearization of the $p_u$-term] \label{lin:lem}
	Let $X$ be an ergodic $G$-system, let $U$ be a compact abelian group acting freely on $X$ and commuting with the action of $G$. Let $\rho:G\times X\rightarrow S^1$ be a cocycle and suppose that there exists $m\in\mathbb{N}$ such that for every $u\in U$ there exist phase polynomials $p_u\in P_{<m}(G,X,S^1)$ and a measurable map $F_u:X\rightarrow S^1$ such that $\Delta_u\rho = p_u\cdot \Delta F_u$. Then there exists a measurable choice $u\mapsto p'_u$ and $u\mapsto F'_u$ such that $\Delta_u\rho = p'_u\cdot \Delta F'_u$ for phase polynomials $p'_u\in P_{<m}(G,X,S^1)$ which satisfies that $p'_{u v} = p'_u\cdot V_u p'_v$ whenever $u,v,u v\in U'$ where $U'$ is some neighborhood of $U$.
\end{lem}
The proof of this Lemma is given in \cite{Berg& tao & ziegler} as part of the proof of Proposition 6.1 (see in particular equation (6.5) in that proof).
\section{Topological groups and measurable homomorphisms}
In this section we survey some results about topological groups. 
\begin{lem} [A.Weil]  \cite[Lemma 2.3]{Ch}\label{A.Weil}
	Let $G$ be a locally compact Polish group and let $A\subseteq G$ be a measurable subset of positive measure. Then $A\cdot A^{-1}$ contains an open neighborhood of the identity.
\end{lem}
At some point we have to work with non-locally compact groups. The following variant of the theorem above will therefore be useful for us (see \cite[Theorem 9.9]{Ke}).
\begin{lem}[Pettis Lemma] \label{Pettis}
    Let $\mG$ be a Polish group, and $A\subseteq \mG$ be a Baire-measurable, non-meagre subset of $G$. Then $1_G$ lies in the interior of $A\cdot A^{-1}$.
\end{lem}
The following is a variant of the open mapping theorem for Polish groups
\begin{thm} \cite[Chapter 1]{BK}  \label{open}
	Let $G$ and $H$ be Polish groups and let $p: G \rightarrow H$ be a
	group homomorphism that is continuous and onto. Then $p$ is an open map. Moreover, $p$ admits a Borel cross section, that is, a Borel map $s:H\rightarrow G$ with $p\circ s = Id$.
\end{thm}
From this it is easy to conclude the following result.
\begin{cor} \label{compact}
	Let $H$ be a closed normal subgroup of the Polish group $G$.
	If $H$ and $G/H$ are (locally) compact, then $G$ is (locally) compact.
\end{cor}
Another corollary of theorem \ref{open} is the following result about quotient spaces due to Effros \cite{Effros}.
	\begin{thm} \label{Effros}
		If $\mG$ is a Polish group which acts transitively on a compact metric space $X$. Then for any $x\in X$ the stabilizer $\Gamma = \{g\in \mG : gx=x\}$ is a closed subgroup of $\mG$ and $X$ is homeomorphic to $\mG/\Gamma$.
	\end{thm}
\subsection{Totally disconnected groups}
\begin{defn} \cite[Exercise E8.6]{HM} Let $X$ be a compact Hausdorff space. Then the following are equivalent
	\begin{itemize}
		\item{Every connected component in $X$ is a singleton.}
		\item {$X$ has a basis consisting of open closed sets.}
	\end{itemize}
	We say that $X$ is totally disconnected if one of the above is satisfied.
\end{defn}
In this section we will be interested in compact (Hausdorff) totally disconnected groups. These groups are also called profinite groups.
\begin{prop}\label{opensubgroup}
	Let $G$ be a compact Hausdorff totally disconnected group. Let $1\in U\subseteq G$ be an open neighborhood of the identity, then $U$ contains an open subgroup of $G$.
\end{prop}
The proof of this Proposition can be found in \cite[Proposition 1.1.3]{profinite}. As a corollary we have the following result.
\begin{cor}[The dual of totally disconnected group is a torsion group]  \label{chartdg} Let $G$ be a compact abelian totally disconnected group, then the image of any continuous character $\chi:G\rightarrow S^1$ is finite.
\end{cor}
\begin{proof}
	Choose an open neighborhood of the identity $U$ in $S^1$ that contains no non-trivial subgroups. Then, $\chi^{-1}(U)$ is an open neighborhood of $G$. By the previous Proposition there exists an open subgroup $H$ such that $H\subseteq \chi^{-1}(U)$. It follows that $\chi(H)$ is trivial and so $\chi$ factors through the finite group $G/H$. This implies that the image is finite.
\end{proof}
We need the following classical structure theorem \cite[Chapter 5, Theorem 18]{M}.
\begin{thm}[Structure theorem for abelian groups of bounded torsion]\label{torsion}
	Let $G$ be a compact abelian group and suppose that there exists some $n\in\mathbb{N}$ such that $g^n=1_G$ for every $g\in G$. Then, $G$ is topologically and algebraically isomorphic to $\prod_{i=1}^\infty C_{m_i}$ where for every $i$, $m_i$ is an integer which divides $n$.
\end{thm}
One way to generate totally disconnected groups is to begin with an arbitrary compact abelian group and quotient it out by its connected component.
\begin{lem}\label{connectedcomponent}
	Let $G$ be a compact abelian group and $G_0$ be the connected component of the identity in $G$. Since the multiplication and the inversion maps are continuous one has that $G_0$ is a subgroup of $G$ and
	\begin{itemize}
		\item{$G_0$ has no non-trivial open subgroups.}
		\item {Every open subgroup of $G$ contains $G_0$.}
		\item {$G/G_0$ equipped with the quotient topology is totally disconnected compact group.}
	\end{itemize}
\end{lem}
We also need the following important fact that connected groups are divisible \cite[Corollary 8.5]{HM}.
\begin{lem}\label{divisible}
    Let $G$ be a compact abelian connected group. Then for every $g\in G$ and $n\in\mathbb{N}$ there exists $h\in G$ such that $h^n=g$.
\end{lem}
\subsection{Lie groups}
\begin{defn}
	A topological group $G$ is said to be a Lie group if it has the structure of a finite dimensional differentiable manifold over $\mathbb{R}$ such that the multiplication and inversion maps are smooth.
\end{defn}
A compact abelian group is a Lie group if and only if its Pontryagin dual is finitely generated. The structure theorem for finitely generated abelian group gives
\begin{thm}[Structure Theorem for compact abelian Lie groups]\cite[Theorem 5.2]{S} \label{structureLieGroups} A compact abelian group $G$ is a Lie group if and only if there exists $n\in\mathbb{N}$ such that $G\cong (S^1)^n\times C_k$ where $C_k$ is some finite group with discrete topology.
\end{thm} \label{approxLieGroups}
The famous Gleason-Yamabe theorem implies that every compact abelian group can be approximated by compact abelian Lie groups using inverse limits.
\begin{thm}\cite[Corollary 8.18]{HM} \label{GY:thm}
	Let $G$ be a compact abelian group and let $U$ be a neighborhood of the identity in $G$. Then $U$ contains a subgroup $N$ such that $G/N$ is a Lie group.
\end{thm}
It follows from the above (see also \cite[Lemma 2.2]{Iwasawa}) that any compact connected nilpotent group is abelian.
\begin{prop}\label{connilabel:prop}
	If $G$ is a compact connected $k$-step nilpotent group, then $G$ is abelian.
\end{prop}

\section{Some results about phase polynomials}
In this appendix we survey some results about phase polynomials from \cite{OS} and \cite{Berg& tao & ziegler}. Let $G=\bigoplus_{p\in P}\mathbb{Z}/p\mathbb{Z}$ for some multiset of primes $P$. For a natural number $m\in\mathbb{N}$ we let $C_m$ denote the group of $m$-roots of unity in the unit circle.
\begin{prop}  [Values of phase polynomial cocycles]\label{PPC} Let $X$ be an ergodic $G$-system. Let $d\geq 0$, and let $q:G\times X\rightarrow S^1$ be a phase polynomial of degree $<d$ that is also a cocycle. Then for $g\in G$, $q(g,\cdot)$ takes values in $C_m$ where $m$ is the order of $g$ to the power of $d$.
\end{prop} 
\begin{proof}
We prove the proposition by induction on $d$. If $d=0$ then $q\equiv 1$ and the claim is trivial. Fix $d\geq 1$ and assume inductively that the claim holds for smaller values of $d$. Let $q:G\times X\rightarrow S^1$ be a phase polynomial of degree $<d$ and fix $g\in G$ of order $n$. The cocycle identity implies that $$1=q(ng,x)=\prod_{k=0}^{n-1}q(g,T_{kg}x).$$
	Since $q(g,T_{kg}x)=q(g,x)\cdot \Delta_{kg} q(g,x)$ we have that $q(g,x)^n \cdot \prod_{k=0}^{n-1}\Delta_{kg} q(g,x)=1$. By the induction hypothesis, $\prod_{k=0}^{n-1}\Delta_{kg} q(g,x)$ is in $C_{n^{d-1}}$ and it follows that $q(g,x)\in C_{n^d}$, as required.
\end{proof}
We need the following version of \cite[Lemma D.3(i)]{Berg& tao & ziegler}.
\begin{lem} \label{PPC1}
Let $X$ be an ergodic $G$-system. Let $p^n$ be a power of a prime number $p$ and let $Q:X\rightarrow C_{p^n}$ be a phase polynomial of degree $<d$ for some $d>1$. Then $Q^p$ is a phase polynomial of degree $<d-1$.
\end{lem}
\begin{proof}
Let $G_p = \{g\in G : pg=0\}$ and let $G'$ be the complement so that $G=G_q\oplus G'$. By the proposition above and the assumption, $P$ is invariant with respect to the action of $G'$. Let $X_p$ be the factor of $X$ generated by the $G'$-invariant functions. The induced action of $G_p$ on $X_p$ is ergodic and so $X_p$ admits an ergodic action of $\mathbb{F}_p^\omega$ (note that if $G_p$ is finite one can still define an action of $\mathbb{F}_p^\omega$ by letting some of the coordinates act trivially). Therefore, the claim in the lemma follows by \cite[Lemma D.3(i)]{Berg& tao & ziegler}.
\end{proof}
We need the following lemma is a simple but useful case of Lemma \ref{polydiff}.
\begin{lem}  [Vertical derivatives of phase polynomials are phase polynomials of smaller degree]\label{vdif:lem}
	Let $X$ be an ergodic $G$-system. Let $U$ be a compact abelian group acting freely on $X$ by automorphisms and $P:X\rightarrow S^1$ be a phase polynomial of degree $<d$ for some integer $d\geq 1$. Then $\Delta_u P$ is a phase polynomial of degree $<d-1$ for every $u\in U$.
\end{lem}
Proposition \ref{PPC} and Lemma \ref{vdif:lem} implies the following result.
\begin{cor} \label{ker:cor} 
	Let $X$ be an ergodic $G$-system and $U$ be a compact abelian group acting freely on $X$ by automorphisms. Suppose that there exists a measurable map $u\mapsto f_u$ from $U$ to $P_{<d}(X,S^1)$ which satisfies the cocycle identity (i.e. $f_{uv}=f_u V_u f_v$) for all $u,v\in U$. Then there exists an open subgroup $V$ of $U$ such that $f_v\in P_{<1}(X,S^1)$ for every $v\in V$.
\end{cor}
\begin{proof}
	We prove the claim by induction on $d$; for $d=1$ we can take $V=U$. Let $d>1$ and assume by induction that the claim holds for all smaller values of $d$. Let $u\mapsto f_u$ be a map from $U$ to $P_{<d}(X,S^1)$, the cocycle identity implies that $f_{uv}=f_u f_v \cdot \Delta_u f_v$ for every $u,v\in U$. By Lemma \ref{vdif:lem} $\Delta_u f_v\in P_{<d-1}(X,S^1)$, therefore $u\mapsto f_u\cdot P_{<d-1}(X,S^1)$ is a homomorphism. Since $d>1$, Lemma \ref{sep:lem} (separation Lemma) implies that $P_{<d-1}(X,S^1)$ has at most countable index in $P_{<d}(X,S^1)$. Let $U'$ be the kernel of $u\mapsto f_u\cdot P_{<d-1}(X,S^1)$. Since $U'$ has positive Haar measure it is open.\\ Since $f_{u'}\in P_{<d-1}(X,S^1)$ for all $u'\in U'$, the the induction hypothesis implies that there exists an open subgroup $V$ of $U'$ such that $f_v\in P_{<1}(X,S^1)$ for all $v\in V$. As $V$ is open in $U'$ and $U'$ is open in $U$ we have that $V$ is open in $U$, as required.
\end{proof}
\begin{prop} [Phase polynomial are invariant under connected components] \label{pinv:prop} (see \cite[Lemma 2.1]{BTZ})  Let $X$ be a $G$-system of order $<k$, let $U$ be a compact abelian connected group acting freely on $X$ (not necessarily commuting with the $G$-action). Let $P:G\times X\rightarrow S^1$ be a phase polynomial of degree $<d$ such that for every $g\in G$ there exists $M_g\in\mathbb{N}$ such that $P(g,\cdot)$ takes at most $M_g$ values. (e.g. $P$ is a phase polynomial cocycle - Proposition \ref{PPC}). Then $P$ is invariant under the action of $U$.
\end{prop}
\begin{proof}
	Fix $g\in G$ and consider the map $u\mapsto \Delta_u P(g,\cdot)$. Since $P(g,\cdot)$ is a measurable map $X\rightarrow S^1$, we have that $\Delta_u P$ converges in measure to the constant $1$ as $u$ converges to the identity in $U$. Since convergence in measure implies convergence in $L^2$, we can use Lemma \ref{sep:lem} to conclude that $\Delta_u P(g,\cdot)$ must be almost everywhere constant for $u$ close to the identity. From the cocycle identity, we have that the subset $U_g' = \{u\in U : \Delta_u P(g,\cdot) \text{ is a constant}\}$ is an open subgroup of $U$. As $U$ is connected, we conclude that $U_g'=U$ for every $g\in G$. We conclude that for every $g\in G$, there exists a character $\chi_g:U\rightarrow S^1$ such that  $\Delta_u P(g,\cdot) = \chi_g(u)$ for every $u\in U$. Since $U$ is connected and $\chi_g$ is continuous we have that the image of $\chi_g$ is either trivial or is $S^1$. But, the latter contradicts the assumption that $P(g,\cdot)$ takes finitely many values. It follows that $\Delta_u P(g,\cdot)=1$ for every $u\in U$ and $g\in G$. In other words, $P$ is invariant under the action of $U$, as required.
\end{proof}
\begin{rem}
In some cases the group $S^1$ in the proposition can be replaced by any compact abelian group using Pontryagin duality. For instance, if $P:G\times X\rightarrow V$ is a phase polynomial cocycle for some compact abelian group $V$, then for every $\chi\in\widehat V$ we have that $\chi\circ P:G\times X\rightarrow S^1$ is a phase polynomial cocycle of the same degree. By Proposition \ref{PPC} and Proposition \ref{pinv:prop} we have that $\chi(\Delta_uP)=1$ for every $u\in U$. As the characters separates points this would imply that $\Delta_u P = 1$, hence $P$ is invariant with respect to the action of $U$.
\end{rem}
	
\address{Einstein Institute of Mathematics\\
	The Hebrew University of Jerusalem\\
	Edmond J. Safra Campus, Jerusalem, 91904, Israel \\ Or.Shalom@mail.huji.ac.il}
\end{document}